\documentclass[a4paper]{amsart}
\usepackage{amsmath}
\usepackage{amsthm}
\usepackage{amssymb}
\usepackage{enumerate}
\usepackage{latexsym}
\usepackage{empheq}
\newtheorem{Thm}{Theorem}[section]
\newtheorem{Prop}[Thm]{Proposition}
\newtheorem{Cor}[Thm]{Corollary}
\newtheorem{Lem}[Thm]{Lemma}
\newtheorem{Clm}[Thm]{Claim}

\newtheorem{Ma}{Main Theorem 1}[section]
\newtheorem{Mb}{Main Theorem 2}[section]
\newtheorem{Md}{Main Theorem 3}[section]
\newtheorem{Me}{Main Theorem 4}[section]

\theoremstyle{definition}
\newtheorem{notation}[Thm]{Notation}
\newtheorem{Def}[Thm]{Definition}
\newtheorem{Asu}[Thm]{Assumption}
\theoremstyle{remark}
\newtheorem{Rem}{Remark}[section]

\newcommand{\Ric}{\mathop{\mathrm{Ric}}\nolimits}
\newcommand{\tr}{\mathop{\mathrm{tr}}\nolimits}
\newcommand{\Imag}{\mathop{\mathrm{Im}}\nolimits}
\newcommand{\Id}{\mathop{\mathrm{Id}}\nolimits}
\newcommand{\Vol}{\mathop{\mathrm{Vol}}\nolimits}
\newcommand{\diam}{\mathop{\mathrm{diam}}\nolimits}
\newcommand{\Card}{\mathop{\mathrm{Card}}\nolimits}
\newcommand{\Span}{\mathop{\mathrm{Span}}\nolimits}

\newcommand{\LIP}{\mathop{\mathrm{LIP}}\nolimits}
\newcommand{\Test}{\mathop{\mathrm{Test}}\nolimits}
\newcommand{\TestForm}{\mathop{\mathrm{TestForm}}\nolimits}
\newcommand{\Hess}{\mathop{\mathrm{Hess}}\nolimits}

\title[Almost Parallel $p$-form]{Lichnerowicz-Obata Estimate, Almost Parallel $p$-form and Almost Product Manifolds}
\author{Masayuki Aino}
\address{RIKEN Center for Advanced Intelligence Project (AIP), 1-4-1 Nihonbashi, Tokyo 103-0027, Japan}
\email{masayuki.aino@riken.jp}
\begin{document}
\maketitle
\begin{abstract}
We show a Lichnerowicz-Obata type estimate for the first eigenvalue of the Laplacian of $n$-dimensional closed Riemannian manifolds with an almost parallel $p$-form ($2\leq p \leq n/2$) in $L^2$-sense, and give a Gromov-Hausdorff approximation to a product $S^{n-p}\times X$ under some pinching conditions when $2\leq p<n/2$.
\end{abstract} 
\tableofcontents
\section{Introduction}
In this paper we give an estimate for the first eigenvalue of the Laplacian of closed Riemannian manifolds with positive Ricci curvature and an almost parallel form, and show the Gromov-Hausdorff closeness to a product space for the almost equality case.

One of the most famous theorem about the estimate of the first eigenvalue of the Laplacian is the Lichnerowicz-Obata theorem.
Lichnerowicz showed the optimal comparison result for the first eigenvalue when the Riemannian manifold has positive Ricci curvature, and Obata showed that the equality of the Lichnerowicz estimate implies that the Riemannian manifold is isometric to the standard sphere.
In the following, $\lambda_k(g)$ denotes the $k$-th eigenvalue of the Laplacian $\Delta:=-\tr_g \Hess$ acting on functions.
\begin{Thm}[Lichnerowicz-Obata theorem]
Take an integer $n\geq 2$.
Let $(M,g)$ be an $n$-dimensional closed Riemannian manifold. If $\Ric \geq (n-1) g$, then $\lambda_1(g)\geq n$.
The equality holds if and only if $(M,g)$ is isometric to the standard sphere of radius $1$.
\end{Thm}

Petersen \cite{Pe1}, Aubry \cite{Au} and Honda \cite{Ho} showed the stability result of the Lichnerowicz-Obata theorem.
In the following, $d_{GH}$ denotes the Gromov-Hausdorff distance and $S^n$ denotes the $n$-dimensional standard sphere of radius $1$. (see Definition \ref{DGH} for the definition of the Gromov-Hausdorff distance).
\begin{Thm}[\cite{Au}, \cite{Ho}, \cite{Pe1}]\label{PA}
For given an integer $n\geq 2$ and a positive real number $\epsilon>0$, there exists $\delta(n,\epsilon)>0$ such that if $(M,g)$ is an $n$-dimensional closed Riemannian manifold with $\Ric \geq (n-1) g$ and $\lambda_n(g)\leq n+\delta$, then $d_{GH}(M,S^n)\leq \epsilon$.
\end{Thm}
Note that Petersen considered the pinching condition on $\lambda_{n+1}(g)$, and Aubry and Honda improved it independently.

We mention some improvements of the Lichnerowicz estimate when the Riemannian manifold has a special structure.
If $(M,g)$ is a real $n$-dimensional K\"{a}hler manifold with $\Ric\geq (n-1)g$, then the Lichnerowicz estimate is improved as follows:
\begin{equation}\label{kae}
\lambda_1(g)\geq 2(n-1).
\end{equation}
See \cite[Theorem 11.49]{Be} for the proof.
If $(M,g)$ is a real $n$-dimensional quaternionic K\"{a}hler manifold with $\Ric\geq (n-1)g$, then we have
\begin{equation}\label{qk}
\lambda_1(g)\geq \frac{2n+8}{n+8}(n-1).
\end{equation}
See \cite{AM} for the proof.
In these cases, the Riemannian manifold  $(M,g)$ has a non-trivial parallel $2$ and $4$-form, respectively.
When $(M,g)$ is an $n$-dimensional product Riemannian manifold $(N_1\times N_2,g_1+g_2)$ with $\Ric\geq (n-1)g$,
then we have
$$
\lambda_1(g)\geq \min_{i\in\{1,2\}}\left\{\frac{\dim N_i}{\dim N_i-1}\right\}(n-1),
$$
and $M$ has a non-trivial parallel form if either $N_1$ or $N_2$ is orientable.

Grosjean \cite{gr} gave a unified proof of the improvements of the Lichnerowicz estimate when the Riemannian manifold has a non-trivial parallel form.
\begin{Thm}[\cite{gr}]\label{grosjean}
Let $(M,g)$ be an $n$-dimensional closed Riemannian manifold.
Assume that $\Ric\geq (n-p-1)g$ and that there exists a nontrivial parallel $p$-form on $M$ $(2\leq p\leq n/2)$.
Then, we have
\begin{equation}\label{grs}
\lambda_1(g)\geq n-p.
\end{equation}
 
Moreover, if $p<n/2$ and if in addition $M$ is simply connected, then the equality in $(\ref{grs})$ implies that $(M,g)$ is isometric to a product $S^{n-p}\times (X,g')$, where $(X,g')$ is some $p$-dimensional closed Riemannian manifold.
\end{Thm}
\begin{Rem}
We give several remarks on this theorem.
\begin{itemize}
\item When $\Ric\geq (n-p-1)g$, the Lichnerowicz estimate is $\lambda_1(g)\geq n(n-p-1)/(n-1)$.
Since $n-p>n(n-p-1)/(n-1)$ for $2\leq p\leq n/2$, the estimate (\ref{grs}) improves the Lichnerowicz estimate.
\item Grosjean also showed this type theorem when $M$ has a convex smooth boundary.
\item Though Grosjean originally assumed the manifold is orientable, the assumption can be easily removed by taking the orientable double covering.
\item If $(M,g)$ is either a K\"{a}hler manifold with $n\geq 6$ or a quaternionic K\"{a}hler manifold, then the estimate (\ref{kae}) or (\ref{qk}) (with scaling) is stronger than (\ref{grs}).
\item There exists no non-trivial parallel $1$-form on any closed Riemannian manifold with positive Ricci curvature.
\item The assumption $2\leq p\leq n/2$ (resp. $2\leq p< n/2$) implies $n\geq 4$ (resp. $n\geq 5$).
For the case $n=4$ and $p=n/2=2$, the complex projective space $\mathbb{C}P^2$ also satisfies the equality in (\ref{grs}).
\item If there exists a non-trivial parallel $p$-form $\omega$ ($1\leq p\leq n-1$) on an $n$-dimensional Riemannian manifold $(M,g)$, then $\omega(x)\in \bigwedge^p T^\ast_x M$ ($x\in M$) is invariant under the Holonomy action, and so the Holonomy group coincides with neither $\mathrm{SO}(n)$ nor $\mathrm{O}(n)$.
\end{itemize}
\end{Rem} 

The main aim of this paper is to show the almost version of Grosjean's result.
We also give the almost version of the estimate (\ref{kae}) in Appendix B.

We first note that, for a closed Riemannian manifold $(M,g)$, there exists a non zero $p$-form $\omega$ with $\|\nabla \omega\|_2^2\leq \delta\|\omega\|_2^2$ for some $\delta>0$ if and only if $\lambda_1(\Delta_{C,p})\leq \delta$ holds, where $\lambda_1(\Delta_{C,p})$ is defined by
$$
\lambda_1(\Delta_{C,p}):=\inf\left\{\frac{\|\nabla \omega\|_2^2}{\|\omega\|_2^2}: \omega\in\Gamma(\bigwedge^p T^\ast M) \text{ with }\omega\neq 0\right\}.
$$

Let us state our eigenvalue estimate.
\begin{Ma}
For given integers $n\geq 4$ and $2\leq p \leq n/2$, there exists a constant $C(n,p)>0$ such that if $(M,g)$ is an $n$-dimensional closed Riemannian manifold with $\Ric_g\geq (n-p-1)g$,
then we have
\begin{equation*}
\lambda_1(g)\geq n-p-C(n,p)\lambda_1(\Delta_{C,p})^{1/2}.
\end{equation*}
\end{Ma}

We immediately have the following corollary:
\begin{Cor}
For given integers $n\geq 4$ and $2\leq p \leq n/2$, there exists a constant $C(n,p)>0$ such that if $(M,g)$ is an $n$-dimensional closed Riemannian manifold with $\Ric_g\geq (n-p-1)g$ and
$$\frac{n(n-p-1)}{n-1}\leq \lambda_1(g)\leq n-p,$$ then we have
\begin{equation*}
\lambda_1(\Delta_{C,p})\geq \left(\frac{n-p-\lambda_1(g)}{C(n,p)}\right)^2.
\end{equation*}
\end{Cor}
Note that we always have the lower bound on the eigenvalue of the Laplacian $\lambda_1(g)\geq n(n-p-1)/(n-1)$ if $\Ric_g\geq (n-p-1)g$ by the Lichnerowicz estimate.
An upper bound on $C(n,p)$ is computable. However, we do not know the optimal value of it.

We next state the eigenvalue pinching result.
\begin{Mb}
For given integers $n\geq 5$ and $2\leq p < n/2$ and a positive real number $\epsilon>0$, there exists $\delta=\delta(n,p,\epsilon)>0$ such that if $(M,g)$ is an $n$-dimensional closed Riemannian manifold with $\Ric_g\geq (n-p-1)g$,
\begin{equation*}
\lambda_{n-p+1}(g)\leq n-p+\delta
\end{equation*}
and
\begin{equation*}
\lambda_1(\Delta_{C,p})\leq \delta,
\end{equation*}
then $M$ is orientable and
$$d_{GH}(M,S^{n-p}\times X)\leq \epsilon,$$
where $X$ is some compact metric space.
\end{Mb}
\begin{Rem}
In fact, we prove that there exist constants $C(n,p)>0$ and $\alpha(n)>0$ such that
$$d_{GH}(M,S^{n-p}\times X)\leq C(n,p)\delta^{\alpha(n)}$$
under the assumption of Main Theorem 2.
One can easily find the explicit value of $\alpha(n)$ (see Notation \ref{order} and Theorem \ref{MT2}).
However, it might be far from the optimal value.
By the Gromov's pre-compactness theorem, we can take $X$ to be a geodesic space.
However, we lose the information about the convergence rate in that case.
\end{Rem}
Based on Theorem \ref{PA}, one might expect that we can replace the assumption ``$\lambda_{n-p+1}(g)\leq n-p+\delta$'' in Main Theorem 2 to the weaker assumption ``$\lambda_{n-p}(g)\leq n-p+\delta$''.
However, an example shows that we cannot do it even if $\delta=0$ (see Proposition \ref{p3e}).
Instead of that, replacing $\lambda_1(\Delta_{C,p})$ to $\lambda_1(\Delta_{C,n-p})$, we have the following theorems:
\begin{Md}
For given integers $n\geq 4$ and $2\leq p \leq n/2$, there exists a constant $C(n,p)>0$ such that if $(M,g)$ is an $n$-dimensional closed Riemannian manifold with $\Ric_g\geq (n-p-1)g$,
then we have
\begin{equation*}
\lambda_1(g)\geq n-p-C(n,p)\lambda_1(\Delta_{C,n-p})^{1/2}.
\end{equation*}
\end{Md}
\begin{Me}
For given integers $n\geq 5$ and $2\leq p < n/2$ and a positive real number $\epsilon>0$, there exists $\delta=\delta(n,p,\epsilon)>0$ such that if $(M,g)$ is an $n$-dimensional closed Riemannian manifold with $\Ric_g\geq (n-p-1)g$,
\begin{equation*}
\lambda_{n-p}(g)\leq n-p+\delta
\end{equation*}
and
\begin{equation*}
\lambda_1(\Delta_{C,n-p})\leq \delta,
\end{equation*}
then we have
$$d_{GH}(M,S^{n-p}\times X)\leq \epsilon,$$
where $X$ is some compact metric space.
\end{Me}
Note that the assumption ``$\lambda_1(\Delta_{C,n-p})\leq \delta$'' is equivalent to the assumption ``$\lambda_1(\Delta_{C,p})\leq \delta$'' if the manifold is orientable.

We would like to point out that our work was motivated by Honda's spectral convergence theorem \cite{Ho2}, which asserts the continuity of the eigenvalues of the connection Laplacian $\Delta_{C,p}$ acting on $p$-forms with respect to the non-collapsing Gromov-Hausdorff convergence assuming the two-sided bound on the Ricci curvature.
By virtue of his theorem, we can generalize our main theorems to Ricci limit spaces under such assumptions.
Note that we show our main theorems without the non-collapsing assumption, i.e., without assuming the lower bound on the volume of the Riemannian manifold. 

Our work was also motivated by the Cheeger-Colding almost splitting theorem (see \cite[Theorem 9.25]{Ch}), whose conclusion is the Gromov-Hausdorff approximation to a product $\mathbb{R}\times X$.
As the almost splitting theorem, we need to show the almost Pythagorean theorem under the assumption of Main Theorem 2.

The structure of this paper is as follows.

In section 2, we recall some basic definitions and facts, and give calculations of differential forms.


In section 3, we estimate the error terms of the Grosjean's formula when the Riemannian manifold has a non-trivial almost parallel $p$-form.
As a consequence, we prove Main Theorem 1 and Main Theorem 3.

In section 4, we prove Main Theorem 2 and Main Theorem 4.
In subsection 4.1, we list some useful techniques for our pinching problem.
In subsection 4.2, we show some pinching conditions on the eigenfunctions along geodesics under the assumption $\lambda_k(g)\leq n-p+\delta$ and $\lambda_1(\Delta_{C,p})\leq \delta$.
In subsection 4.3, we show that similar results hold under the assumption $\lambda_k(g)\leq n-p+\delta$ and $\lambda_1(\Delta_{C,n-p})\leq \delta$.
In subsection 4.4, we show that the eigenfunctions are almost cosine functions in some sense under our pinching condition.
In subsection 4.5, we construct an approximation map and show Main Theorem 2 except for the orientability.
In subsection 4.6, we give some lemmas to prove the remaining part of main theorems.
In subsection 4.7, we show the orientability of the manifold under the assumption of Main Theorem 2, and complete the proof of it.
In subsection 4.8, we show that the assumption of Main Theorem 4 implies that $\lambda_{n-p+1}(g)$ is close to $n-p$, and complete the proof of Main Theorem 4.

In Appendix A, we discuss Ricci limit spaces.
Using the technique of subsection 4.7, we show the stability of unorientability under the non-collapsing Gromov-Hausdorff convergence assuming the two-sided bound on the Ricci curvature and the upper bound on the diameter.

In Appendix B, we give the almost version of the estimate (\ref{kae}) assuming that there exists a $2$-form $\omega$ which satisfies that $\|\nabla \omega\|_2$ and $\|J_\omega^2+\Id\|_1$ are small, where $J_\omega\in\Gamma(T^\ast M\otimes T M)$ is defined so that $\omega=g(J_\omega\cdot,\cdot)$.
\begin{sloppypar}
{\bf Acknowledgments}.\ 
I am grateful to my supervisor, Professor Shinichiroh Matsuo for his advice.
I also thank Professor Shouhei Honda for helpful discussions about the orientability of Ricci limit spaces.
I thank Shunsuke Kano for the discussions about the examples.
The works in section 3 were done during my stay at the University of C\^{o}te d'Azur.  
I would like to thank Professor Erwann Aubry for his warm hospitality.
I am grateful to the referee for careful reading of the paper and making valuable suggestions.
This work was supported by JSPS Overseas Challenge Program for Young Researchers and by JSPS Research Fellowships for Young Scientists (JSPS KAKENHI Grant Number JP18J11842).
\end{sloppypar}

\section{Preliminaries}
\subsection{Basic Definitions}
We first recall some basic definitions and fix our convention.
\begin{Def}[Hausdorff distance]\label{Dhau}
Let $(X,d)$ be a metric space.
For each point $x_0\in X$, subsets $A,B\subset X$ and $r>0$, define
\begin{align*}
d(x_0,A):=&\inf\{d(x_0,a):a\in A\},\\
B_{r}(x_0):=&\{x\in X: d(x,x_0)<r\},\\
B_{r}(A):=&\{x\in X:d(x,A)<r\},\\
d_{H,d}(A,B):=&\inf\{\epsilon>0:A\subset B_{\epsilon}(B) \text{ and } B\subset B_{\epsilon}(A)\}
\end{align*}
We call $d_{H,d}$ the Hausdorff distance.
\end{Def}
The Hausdorff distance defines a metric on the collection of compact subsets of $X$.
\begin{Def}[Gromov-Hausdorff distance]\label{DGH}
Let $(X,d_X),(Y,d_Y)$ be metric spaces.
Define
\begin{align*}
d_{GH}(X,Y):=\inf\Big\{d_{H,d}(X,Y): &\text{ $d$ is a metric on $X\coprod Y$ such that}\\
&\qquad\qquad\quad\text{$d|_X=d_X$ and $d|_Y=d_Y$}\Big\}.
\end{align*}
\end{Def}
The Gromov-Hausdorff distance defines a metric on the set of isometry classes of compact metric spaces (see \cite[Proposition 11.1.3]{Pe3}).

\begin{Def}[$\epsilon$-Hausdorff approximation map]\label{hap}
Let $(X,d_X),(Y,d_Y)$ be metric spaces.
We say that a map $f\colon X\to Y$ is an $\epsilon$-Hausdorff approximation map for $\epsilon>0$ if the following two conditions hold.
\begin{itemize}
\item[(i)] For all $a,b\in X$, we have $|d_X(a,b)-d_Y(f(a),f(b))|< \epsilon$,
\item[(ii)] $f(X)$ is $\epsilon$-dense in $Y$, i.e., for all $y\in Y$, there exists $x\in X$ with $d_Y(f(x),y)< \epsilon$.
\end{itemize}
\end{Def}
If there exists an  $\epsilon$-Hausdorff approximation map $f\colon X\to Y$, then we can show that $d_{GH}(X,Y)\leq 3\epsilon/2$ by considering the following metric $d$ on $X\coprod Y$:
\begin{empheq}[left={d(a,b)=\empheqlbrace}]{align*}
&\qquad d_X(a,b)&& (a,b\in X),\\
&\,\frac{\epsilon}{2} +\inf_{x\in X}(d_X(a,x)+d_Y(f(x),b))&&(a\in X,\,b\in Y),\\
&\qquad d_Y(a,b)&&(a,b\in Y).
\end{empheq}
If $d_{GH}(X,Y)< \epsilon$, then there exists a $2\epsilon$-Hausdorff approximation map from $X$ to $Y$.

Let $C(u_1,\ldots,u_l)>0$ denotes a positive function depending only on the numbers $u_1,\ldots,u_l$.
For a set $X$, $\Card X$ denotes a cardinal number of $X$.

Let $(M,g)$ be a closed Riemannian manifold.
For any $p\geq 1$, we use the normalized $L^p$-norm:
\begin{equation*}
\|f\|_p^p:=\frac{1}{\Vol(M)}\int_M |f|^p\,d\mu_g,
\end{equation*}
and $\|f\|_{\infty}:=\mathop{\mathrm{ess~sup}}\limits_{x\in M}|f(x)|$ for a measurable function $f$ on $M$. We also use these notation for tensors.
We have $\|f\|_p\leq \|f\|_q$ for any $p\leq q \leq \infty$.

Let $\nabla$ denotes the Levi-Civita connection.
Throughout in this paper, 
 $0=\lambda_0(g)< \lambda_1(g) \leq \lambda_2(g) \leq\cdots \to \infty$
denotes the eigenvalues of the Laplacian $\Delta=-\tr\Hess$ acting on functions. 
We sometimes identify $TM$ and $T^\ast M$ using the metric $g$.
Given points $x,y\in M$, let $\gamma_{x,y}$ denotes one of minimal geodesics with unit speed such that $\gamma_{x,y}(0)=x$ and $\gamma_{x,y}(d(x,y))=y$.
For given $x\in M$ and $u\in T_x M$ with $|u|=1$, let $\gamma_{u}\colon \mathbb{R}\to M$ denotes the geodesic with unit speed such that $\gamma_u(0)=x$ and $\dot{\gamma}_u(0)=u$.

For any $x\in M$ and $u\in T_x M$ with $|u|=1$, put
$$
t(u):=\sup\{t\in\mathbb{R}_{>0}: d(x,\gamma_u(t))=t\},
$$
and define $I_x\subset M$ to be the complement of the cut locus at $x$ (see also \cite[p.104]{Sa}), i.e.,
$$
I_x:=\{\gamma_u (t): u\in T_x M \text{ with $|u|=1$ and } 0\leq t< t(u)\}.
$$
Then, $I_x$ is open and $\Vol(M\setminus I_x)=0$ \cite[III Lemma 4.4]{Sa}.
For any $y\in I_x\setminus \{x\}$, the minimal geodesic $\gamma_{x,y}$ is uniquely determined. The function $d(x,\cdot)\colon M\to \mathbb{R}$ is differentiable in $I_x\setminus\{x\}$ and $\nabla d(x,\cdot)(y)=\dot{\gamma}_{x,y}(d(x,y))$ holds for any $y\in I_x\setminus \{x\}$ \cite[III Proposition 4.8]{Sa}.

Let $V$ be an $n$-dimensional real vector space with an inner product $\langle,\rangle$.
We define inner products on $\bigwedge^k V$ and $V\otimes \bigwedge^k V$ as follows:
\begin{equation*}
\begin{split}
&\langle v_1\wedge\ldots\wedge v_k,w_1\wedge \ldots\wedge w_k\rangle=\det \{\langle v_i,w_j\rangle\}_{i,j},\\
&\langle v_0\otimes v_1\wedge\ldots\wedge v_k,w_0\otimes w_1\wedge \ldots\wedge w_k\rangle=\langle v_0,w_0\rangle \det \{\langle v_i,w_j\rangle \}_{i,j},
\end{split}
\end{equation*}
for $v_0,\ldots,v_k,w_0,\ldots,w_k\in V$.
For $\alpha\in V$ and $\omega\in \bigwedge^k V$, there exists a unique $\iota(\alpha)\omega\in \bigwedge^{k-1} V$ such that $\langle\iota(\alpha)\omega,\eta\rangle=\langle\omega,\alpha \wedge \eta\rangle$ holds for any $\eta\in \bigwedge^{k-1} V $.
If $k=0$, we define $\iota(\alpha)\omega=0$ and $\bigwedge^{-1}V=\{0\}$.
Then, $\iota$ defines a bi-linear map:
$$
\iota\colon V\times \bigwedge^k V\to \bigwedge^{k-1} V.
$$
By identifying $V$ and $V^\ast$ using $\langle,\rangle$, we also use the notation $\iota$ for the bi-linear map:
$$
\iota\colon V^\ast \times \bigwedge^k V\to \bigwedge^{k-1} V.
$$

For any Riemannian manifold $(M,g)$, we define operators $\nabla^\ast \colon \Gamma(T^\ast M\otimes \bigwedge^k T^\ast M)\to \Gamma(\bigwedge^k T^\ast M)$ and $d^\ast \colon \Gamma(\bigwedge^k T^\ast M)\to \Gamma(\bigwedge^{k-1}T^\ast M)$ by
\begin{align*}
\nabla^\ast(\alpha\otimes \beta):&=-\tr_{T^\ast M} \nabla(\alpha\otimes \beta)
=-\sum_{i=1}^n \left(\nabla_{e_i}\alpha\right)(e_i)\cdot \beta-\sum_{i=1}^n\alpha(e_i)\cdot\nabla_{e_i}\beta.\\
d^\ast \omega:&=-\sum_{i=1}^n\iota(e_i)\nabla_{e_i}\omega
\end{align*}
for all $\alpha\otimes\beta\in \Gamma(T^\ast M\otimes\bigwedge^k T^\ast M)$ and $\omega\in\Gamma(\bigwedge^k T^\ast M)$, where $n=\dim M$ and $\{e_1,\ldots,e_n\}$ is an orthonormal basis of $TM$.
If $M$ is closed, then we have
\begin{align*}
\int_M \langle T,\nabla \alpha\rangle\,d\mu_g&=\int_M \langle \nabla^\ast T, \alpha\rangle\,d\mu_g,\\
\int_M \langle \omega,d\eta \rangle\,d\mu_g&=\int_M \langle d^\ast \omega, \eta \rangle\,d\mu_g
\end{align*}
for all $T\in\Gamma(T^\ast M\otimes\bigwedge^k T^\ast M)$, $\alpha\in\Gamma(\bigwedge^k T^\ast M)$, $\omega\in\Gamma(\bigwedge^k T^\ast M)$ and $\eta\in\Gamma(\bigwedge^{k-1} T^\ast M)$ by the divergence theorem.
The Hodge Laplacian $\Delta\colon \Gamma(\bigwedge^k T^\ast M)\to\Gamma(\bigwedge^k T^\ast M)$ is defined by
$$
\Delta:=d d^\ast +d^\ast d.
$$

\begin{notation}
For an $n$-dimensional Riemannian manifold $(M,g)$, we can take orthonormal basis of $TM$ only locally in general.
However, for example, the tensor
$$
\sum_{i=1}^n e^i\otimes \iota(\nabla_{e_i} \nabla f)\omega\in \Gamma(T^\ast M\otimes \bigwedge^{k-1} T^\ast M)\quad (f\in C^\infty(M),\,\omega\in \Gamma(\bigwedge^k T^\ast M))
$$
is defined independently of the choice of the orthonormal basis $\{e_1,\ldots,e_n\}$ of $TM$, where $\{e^1,\ldots,e^n\}$ denotes its dual.
Thus, we sometimes use such notation without taking a particular orthonormal basis.
\end{notation}
Finally, we list some important notation.
Let $(M,g)$ be a closed Riemannian manifold.
\begin{itemize}
\item $d$ denotes the Riemannian distance function.
\item $\Ric$ denotes the Ricci curvature tensor.
\item $\diam$ denotes the diameter.
\item $\Vol$ or $\mu_g$ denotes the Riemannian volume measure.
\item$\|\cdot\|_p$ denotes the normalized $L^p$-norm for each $p\geq 1$, which is defined by
\begin{equation*}
\|f\|_p^p:=\frac{1}{\Vol(M)}\int_M |f|^p\,d\mu_g
\end{equation*}
for any measurable function $f$ on $M$.
\item $\|f\|_{\infty}$ denotes the essential sup of $|f|$ for any measurable function $f$ on $M$.
\item $\nabla$ denotes the Levi-Civita connection.
\item $\nabla^2$ denotes the Hessian for functions.
\item $\Delta\colon \Gamma(\bigwedge^k T^\ast M)\to\Gamma(\bigwedge^k T^\ast M)$ denotes the Hodge Laplacian defined by $\Delta:=d d^\ast +d^\ast d$.
We frequently use the Laplacian acting on functions.
Note that $\Delta=-\tr_g \nabla^2$ holds for functions under our sign convention. 
\item $0=\lambda_0(g)< \lambda_1(g) \leq \lambda_2(g) \leq\cdots \to \infty$
denotes the eigenvalues of the Laplacian acting on functions.
\item $\gamma_{x,y}\colon [0,d(x,y)]\to M$ denotes one of minimal geodesics with unit speed such that $\gamma_{x,y}(0)=x$ and $\gamma_{x,y}(d(x,y))=y$ for any $x,y\in M$.
\item $\gamma_{u}\colon \mathbb{R}\to M$ denotes the geodesic with unit speed such that $\gamma_u(0)=x$ and $\dot{\gamma}(0)=u$ for any $x\in M$ and $u\in T_x M$ with $|u|=1$.
\item $I_x\subset M$ denotes the complement of the cut locus at $x\in M$. We have $\Vol(M\setminus I_x)=0$. We have that $\gamma_{x,y}$ is uniquely determined and $\nabla d(x,\cdot)=\dot{\gamma}_{x,y}(d(x,y))$ holds for any $y\in I_x\setminus\{x\}$.
\item $\Delta_{C,k}=\nabla^\ast \nabla\colon \Gamma(\bigwedge^k T^\ast M)\to \Gamma(\bigwedge^k T^\ast M)$ denotes the connection Laplacian acting on $k$-forms.
\item $0\leq \lambda_1(\Delta_{C,k}) \leq \lambda_2(\Delta_{C,k}) \leq\cdots \to \infty$ denotes the eigenvalues of the connection Laplacian $\Delta_{C,k}$ acting on $k$-forms.
\item $S^n(r)$ denotes the $n$-dimensional standard sphere of radius $r$.
\item $S^n:=S^n(1)$.
\end{itemize}
Note that the lowest eigenvalue of the Laplacian $\Delta$ acting on function is always equal to $0$, and so we start counting the eigenvalues of it from $i=0$.
This is not the case with the connection Laplacian $\Delta_{C,k}$ acting on $k$-forms, and so we start counting the eigenvalues of it from $i=1$.
For any $i\in\mathbb{Z}_{>0}$, we have
$$
\lambda_i(\Delta_{C,0})=\lambda_{i-1}(g).
$$

\subsection{Calculus of Differential Forms}
In this subsection, we recall some facts about differential forms, and do some calculations.

We first recall the decomposition:
\begin{equation*}
T^\ast M\otimes \bigwedge^k T^\ast M=T^{k,1}M\oplus\bigwedge^{k+1} T^\ast M\oplus \bigwedge^{k-1} T^\ast M.
\end{equation*}
See also \cite[Section 2]{Se}.

Let $V$ be an $n$-dimensional real vector space with an inner product $\langle,\rangle$.
We put
\begin{equation*}
\begin{split}
&P_1\colon V\otimes \bigwedge^k V\to  \bigwedge^{k+1} V,\quad P_1(\alpha\otimes \omega):=\left(\frac{1}{k+1}\right)^\frac{1}{2}\alpha\wedge\omega,\\
&P_2\colon V\otimes \bigwedge^k V\to  \bigwedge^{k-1} V,\quad P_2(\alpha\otimes \omega):=\left(\frac{1}{n-k+1}\right)^\frac{1}{2}\iota(\alpha)\omega,\\
&Q_1\colon \bigwedge^{k+1} V\to V\otimes \bigwedge^k V,\quad Q_1(\zeta):=\left(\frac{1}{k+1}\right)^\frac{1}{2}\sum_{i=1}^n e^i\otimes\iota(e^i)\zeta,\\
&Q_2\colon \bigwedge^{k-1} V\to V\otimes \bigwedge^k V,\quad Q_2(\eta):=\left(\frac{1}{n-k+1}\right)^\frac{1}{2}\sum_{i=1}^n e^i\otimes e^i\wedge\eta,
\end{split}
\end{equation*}
where $\{e^1,\ldots,e^n\}$ is orthonormal basis of $V$.
Then, we have 
\begin{itemize}
\item $\Imag Q_1\bot \Imag Q_2$,
\item $P_i\circ Q_i=\Id$ for each $i=1,2$,
\item $Q_1$ and $Q_2$ preserve the norms, 
\item $Q_i\circ P_i\colon V\otimes \bigwedge^k V\to V\otimes \bigwedge^k V$ is symmetric and $(Q_i\circ P_i)^2=Q_i\circ P_i$ for each $i=1,2$.
\end{itemize}
Therefore, $Q_i\circ P_i$ is the orthogonal projection $V\otimes \bigwedge^k V\to \Imag Q_i$.
Since $\bigwedge^{k+1} V\cong \Imag Q_1$ and $\bigwedge^{k-1} V \cong\Imag Q_2$, we can regard $\bigwedge^{k+1} V$ and $\bigwedge^{k-1} V$ as subspaces of $V\otimes \bigwedge^k V$.

Take an $n$-dimensional Riemannian manifold $(M,g)$ and consider the case when $V=T^\ast_x M$ ($x\in M$).
We can take a sub-bundle $T^{k,1}M$ of $T^\ast M\otimes \bigwedge^k T^\ast M$ such that  
\begin{equation*}
T^\ast M\otimes \bigwedge^k T^\ast M=T^{k,1}M\oplus\bigwedge^{k+1} T^\ast M\oplus \bigwedge^{k-1} T^\ast M
\end{equation*}
is an orthogonal decomposition.
Then, for $\omega\in\Gamma(\bigwedge^k T^\ast M)$, we can decompose $\nabla \omega\in \Gamma(T^\ast M\otimes\bigwedge^k T^\ast M)$, the $\bigwedge^{k+1} T^\ast M$-component is equal to $\left(1/(k+1)\right)^{1/2}d\omega$ and the $\bigwedge^{k-1} T^\ast M$-component is equal to $-\left(1/(n-k+1)\right)^{1/2} d^\ast \omega$.
Let  $T(\omega)$ denotes the remaining part ($T\colon \Gamma(\bigwedge^k T^\ast M)\to \Gamma(T^{k,1}M)$).
Then, we have 
\begin{equation*}
\nabla \omega=T(\omega)+ \left(\frac{1}{k+1}\right)^\frac{1}{2} Q_1(d\omega)-\left(\frac{1}{n-k+1}\right)^\frac{1}{2}Q_2(d^\ast w).
\end{equation*}
Therefore, we get 
\begin{equation}\label{2b}
|\nabla\omega|^2=|T(\omega)|^2+\frac{1}{k+1} |d\omega|^2+\frac{1}{n-k+1}|d^\ast \omega|^2.
\end{equation}
If $d^\ast \omega=0$ and $T(\omega)=0$, then $\omega$ is called a Killing k-form (see also \cite[Definition 2.1]{Se}).

We next recall the Bochner-Weitzenb\"ock formula.
\begin{Def}\label{p2a}
Let $(M,g)$ be an $n$-dimensional Riemannian manifold.
We define a homomorphism $\mathcal{R}_k\colon \bigwedge^k T^\ast M\to \bigwedge^k T^\ast M$ as
\begin{equation*}
\mathcal{R}_k \omega=-\sum_{i,j}e^i\wedge \iota(e_j)\left(R(e_i,e_j)\omega\right)
\end{equation*}
for any $\omega\in\bigwedge^k T^\ast M$, where $\{e_1,\ldots,e_n\}$ is an orthonormal basis of $TM$, $\{e^1,\ldots,e^n \}$ is its dual and $R(e_i,e_j)\omega$ is defined by
$$
R(e_i,e_j)\omega=\nabla_{e_i}\nabla_{e_j}\omega-\nabla_{e_j}\nabla_{e_i}\omega-\nabla_{[e_i,e_j]}\omega\in \Gamma(\bigwedge^k T^\ast M).
$$
\end{Def}
Note that if $k=1$, then we have $\mathcal{R}_1 \omega=\Ric (\omega,\cdot)$ for any $\omega\in\Gamma(T^\ast M)$.

The Bochner-Weitzenb\"ock formula is stated as follows:
\begin{Thm}[Bochner-Weitzenb\"ock formula]\label{p2b}
For any $\omega\in\Gamma (\bigwedge^k T^\ast M)$, we have
\begin{equation*}
\Delta\omega=\nabla^\ast \nabla \omega+\mathcal{R}_k \omega.
\end{equation*}
\end{Thm}
In particular, we have the following theorem when $k=1$:
\begin{Thm}[Bochner-Weitzenb\"ock formula for 1-forms]
For any $\omega\in\Gamma(T^\ast M)$, we have
\begin{equation*}
\Delta \omega =\nabla^\ast \nabla \omega + \Ric(\omega,\cdot).
\end{equation*}
\end{Thm}

Let us do some calculations of differential forms.
\begin{Lem}\label{p2c}
Let $(M,g)$ be an $n$-dimensional Riemannian manifold.
Take a vector field $X\in\Gamma(TM)$, a $p$-form $\omega\in\Gamma(\bigwedge^p T^\ast M)$ $(p\geq 1)$ and a local orthonormal bases $\{e_1,\ldots,e_n\}$ of $TM$.
\begin{itemize}
\item[(i)] We have $$\mathcal{R}_{p-1}(\iota(X)\omega)=\iota(X) \mathcal{R}_p \omega+\iota(\Ric(X))\omega+2\sum_{i=1}^n\iota(e_i)(R(X,e_i)\omega).$$
\item[(ii)] We have $$\Delta (\iota(X)\omega)=\iota(\Delta X)\omega+\iota(X)\Delta \omega
+2\sum_{i=1}^n\iota(e_i) (R(X,e_i)\omega)-2\sum_{i=1}^n\iota(\nabla_{e_i}X) (\nabla_{e_i}\omega).$$
\item[(iii)] We have
$$\sum_{i=1}^n\iota(e_i) (R(X,e_i)\omega)
=-\nabla_X d^\ast \omega +d^\ast \nabla_X \omega+\sum_{i,j=1}^n \langle \nabla_{e_j} X, e_i\rangle\iota(e_j)\nabla_{e_i}\omega.$$
\end{itemize}
\begin{proof}
Let $\{e^1,\ldots,e^n\}$ be the dual basis of $\{e_1,\ldots,e_n\}$.

We first show (i).
If $p=1$, both sides are equal to $0$.
Let us assume $p\geq 2$.
We have
\begin{equation}\label{2c}
\begin{split}
&\iota(\Ric(X))\omega\\
=&\frac{1}{(p-1)!}\sum_{i,i_1,\ldots,i_{p-1}}\omega(R(X,e_i)e_i,e_{i_1},\cdots,e_{i_{p-1}})e^{i_1}\wedge\cdots\wedge e^{i_{p-1}}\\
=&\frac{-1}{(p-1)!}\sum_{i,i_1,\ldots,i_{p-1}} (R(X,e_i)\omega)(e_i,e_{i_1},\ldots,e_{i_n})e^{i_1}\wedge\cdots\wedge e^{i_{p-1}}\\
&-\frac{1}{(p-1)!}\sum_{i,i_1,\ldots,i_{p-1}} \sum_{l=1}^{p-1} \omega(e_i,e_{i_1},\cdots,R(X,e_i)e_{i_l},\ldots,e_{i_{p-1}})e^{i_1}\wedge\cdots\wedge e^{i_{p-1}}\\
=&-\sum_{i=1}^n\iota(e_i)(R(X,e_i)\omega)\\
&-\frac{1}{(p-1)!}\sum_{i,i_1,\ldots,i_{p-1}} \sum_{l=1}^{p-1} \omega(e_i,e_{i_1},\cdots,R(X,e_i)e_{i_l},\ldots,e_{i_{p-1}})e^{i_1}\wedge\cdots\wedge e^{i_{p-1}}
\end{split}
\end{equation}
We calculate the second  term.
\begin{equation*}
\begin{split}
-&\frac{1}{(p-1)!}\sum_{i,i_1,\ldots,i_{p-1}} \sum_{l=1}^{p-1} \omega(e_i,e_{i_1},\cdots,R(X,e_i)e_{i_l},\ldots,e_{i_{p-1}})e^{i_1}\wedge\cdots\wedge e^{i_{p-1}}\\
=&\frac{1}{(p-1)!}\sum_{l=1}^{p-1} \sum_{i,j,i_1,\ldots,i_{p-1}}\langle R(e_j,e_{l_l})X,e_i\rangle\omega(e_i,e_j,e_{i_1},\cdots,\widehat{e_{i_l}},\ldots,e_{i_{p-1}})\\
&\qquad\qquad\qquad \qquad\qquad\qquad\qquad\qquad\qquad\qquad
e^{i_l}\wedge e^{i_1}\wedge\cdots\wedge\widehat{e^{i_l}}\wedge\cdots\wedge e^{i_{p-1}}\\
=&\sum_{j,k} e^k\wedge\iota(e_j)\iota(R(e_j,e_{k})X)\omega\\
=&\sum_{j,k} e^k\wedge\iota(e_j)R(e_j,e_{k})(\iota(X)\omega)-\sum_{j,k} e^k\wedge\iota(e_j)\iota(X)R(e_j,e_{k})\omega\\
=&\mathcal{R}_{p-1}(\iota(X)\omega)-\iota(X)\mathcal{R}_{p}\omega
-\sum_{j=1}^n \iota(e_j)(R(X,e_j)\omega)
\end{split}
\end{equation*}
Combining this and (\ref{2c}), we get (i).


Let us show (ii).
We have
\begin{equation*}
\nabla^\ast \nabla \iota(X)\omega
=\iota(\nabla^\ast \nabla X)\omega
-2\sum_{i} \iota(\nabla_{e_i}X)\nabla_{e_i}\omega+\iota(X)\nabla^\ast \nabla\omega.
\end{equation*}
Thus, by (i), we get
\begin{equation*}
\begin{split}
\Delta ( \iota(X)\omega)
=&\nabla^\ast \nabla \iota(X)\omega+\mathcal{R}_{p-1}\iota(X)\omega\\
=&\iota(\Delta X)\omega+\iota(X)\Delta \omega
+2\sum_{i=1}^n\iota(e_i) (R(X,e_i)\omega)-2\sum_{i=1}^n\iota(\nabla_{e_i}X) (\nabla_{e_i}\omega).
\end{split}
\end{equation*}
This gives (ii).

Finally, we show (iii).
We have
\begin{equation*}
\begin{split}
\sum_{i=1}^n \iota(e_i)(R(X,e_i)\omega)
=&\sum_{i=1}^n \iota(e_i)\left(\nabla_{X}\nabla_{e_i}\omega-\nabla_{e_i}\nabla_X\omega-\nabla_{\nabla_{X} e_i}\omega+\nabla_{\nabla_{e_i}X}\omega\right)\\
=&-\nabla_{X}d^\ast \omega+d^\ast \nabla_X\omega+\sum_{i,j=1}^n \langle \nabla_{e_j} X, e_i\rangle\iota(e_j)\nabla_{e_i}\omega.
\end{split}
\end{equation*}
This gives (iii).
\end{proof}

\end{Lem}
When $\omega$ is parallel, we have the following corollary.
\begin{Cor}\label{p2d}
Let $(M,g)$ be an $n$-dimensional Riemannian manifold.
Take a vector field $X\in\Gamma(TM)$ and a parallel $p$-form $\omega\in\Gamma(\bigwedge^p T^\ast M)$ $(p\geq 1)$.
\begin{itemize}
\item[(i)] We have $$\mathcal{R}_{p-1}(\iota(X)\omega)=\iota(\Ric(X))\omega.$$
\item[(ii)] We have $$\Delta (\iota(X)\omega)=\iota(\Delta X)\omega.$$
\end{itemize}
\end{Cor}

Finally, we give some easy equations for later use.
Let $(M,g)$ be an $n$-dimensional Riemannian manifold.
Take a local orthonormal basis $\{e_1,\ldots,e_n\}$ of $TM$.
Let $\{e^1,\ldots,e^n\}$ be its dual.
For any $\omega,\eta\in\Gamma(\bigwedge^k T^\ast M)$, we have
$$
\sum_{i=1}^n \langle e^i\wedge \omega, e^i\wedge \eta \rangle=(n-k)\langle\omega,\eta\rangle,
\quad \sum_{i=1}^n \langle \iota(e_i)\omega, \iota(e_i) \eta \rangle=k\langle\omega,\eta\rangle.
$$
For any $\alpha_1,\ldots,\alpha_k\in \Gamma(T^\ast M)$, we have
$$
Q_1(\alpha_1\wedge\cdots\wedge\alpha_k)=\left(\frac{1}{k}\right)^{1/2}\sum_{i=1}^k(-1)^{i-1}\alpha_i\otimes\alpha_1\wedge\cdots\wedge\widehat{\alpha_i}\wedge\cdots\wedge\alpha_k.
$$
Since $Q_1$ preserves the norms, we have
\begin{equation}\label{q1k}
\begin{split}
&k\left|\alpha_1\wedge\cdots\wedge\alpha_k\right|^2\\
=&\left|\sum_{i=1}^k(-1)^{i-1}\alpha_i\otimes\alpha_1\wedge\cdots\wedge\widehat{\alpha_i}\wedge\cdots\wedge\alpha_k\right|^2
\end{split}
\end{equation}
for any $\alpha_1,\ldots,\alpha_k\in \Gamma(T^\ast M)$.

Suppose that $M$ is oriented.
For any $k$, the Hodge star operator $\ast\colon \bigwedge^k T^\ast M\to \bigwedge^{n-k} T^\ast M$ is defined so that
$$
\langle\ast\omega,\eta \rangle V_g=\omega\wedge\eta
$$
for all $\omega\in\Gamma(\bigwedge^k T^\ast M)$ and $\eta\in\Gamma(\bigwedge^{n-k} T^\ast M)$, where $V_g$ denotes the volume form on $(M,g)$.
For any $\alpha\in\Gamma(T^\ast M)$, $\omega\in\Gamma(\bigwedge^k T^\ast M)$ and  $\eta\in\Gamma(\bigwedge^{k-1} T^\ast M)$, we have
\begin{align*}
\langle\ast(\omega \wedge \alpha),\eta\rangle V_g&=\omega \wedge \alpha \wedge \eta,\\
\langle\iota(\alpha)\ast \omega,\eta\rangle V_g=\langle\ast \omega,\alpha\wedge \eta\rangle V_g&=\omega\wedge \alpha \wedge \eta.
\end{align*}
Thus, we get 
\begin{equation}\label{hstar2}
\ast(\omega \wedge \alpha)=\iota(\alpha)\ast \omega.
\end{equation}
Therefore, for any $\alpha,\beta\in\Gamma(T^\ast M)$ and $\omega,\eta\in\Gamma(\bigwedge^k T^\ast M)$, we have
\begin{equation*}
\begin{split}
&\langle\iota (\alpha)\omega,\iota(\beta)\eta\rangle
=\langle\omega,\alpha \wedge \iota(\beta)\eta\rangle\\
=&-\langle\beta \wedge \omega,\alpha \wedge \eta\rangle+\langle\alpha,\beta\rangle\langle\omega,\eta\rangle
=-\langle\iota(\beta)\ast \omega,\iota(\alpha)\ast \eta\rangle+\langle\alpha,\beta\rangle\langle\omega,\eta\rangle,
\end{split}
\end{equation*}
and so
\begin{equation}\label{hstar}
\langle\iota (\alpha)\omega,\iota(\beta)\eta\rangle+\langle\iota(\beta)\ast \omega,\iota(\alpha)\ast \eta\rangle=\langle\alpha,\beta\rangle\langle\omega,\eta\rangle.
\end{equation}

\section{Almost Parallel $p$-form}
In this section, we show Main Theorems 1 and 3.
\subsection{Parallel $p$-form}
In this subsection, we show some easy results when the Riemannian manifold has a non-trivial parallel $p$-form.
We first give an easy proof of what Grosjean called a new Bochner-Reilly formula \cite[Proposition 3.1]{gr} for closed Riemannian manifolds with a non-trivial parallel $p$-form $\omega$.
Similarly, we also get the formula \cite[Proposition 3.1]{gr} for Riemannian manifold with boundary.
In the next subsection, we estimate the error terms when $\omega$ is not parallel.
\begin{Prop}[Bochner-Reilly-Grosjean formula \cite{gr}]\label{p3a}
Let $(M,g)$ be an $n$-dimensional closed Riemannian manifold.
For any $f\in C^\infty(M)$ and any parallel $p$-form $\omega$ $(1\leq p \leq n-1)$ on $M$, we have
\begin{equation*}
\begin{split}
&\int_M |T (\iota(\nabla f)\omega)|^2\,d\mu_g\\
=&\frac{p-1}{p}\int_M\langle\iota(\nabla f)\omega, \iota(\nabla\Delta f)\omega\rangle \,d\mu_g-\int_M \langle\iota(\Ric(\nabla f))\omega,\iota(\nabla f)\omega\rangle\,d\mu_g.
\end{split}
\end{equation*}
See subsection 2.2 for the definition of $T\colon \Gamma(\bigwedge^{p-1}T^\ast M)\to \Gamma(T^{p-1,1}M)$.
\end{Prop}
\begin{proof}
Since $d^\ast \iota(\nabla f) \omega=-d^\ast d^\ast(f\omega)=0$, we have
\begin{equation}\label{3a}
\begin{split}
&\int_M \langle\iota(\Ric(\nabla f))\omega,\iota(\nabla f)\omega\rangle\,d\mu_g\\
=&\int_M \langle\mathcal{R}_{p-1}(\iota(\nabla f)\omega),\iota(\nabla f)\omega\rangle\,d\mu_g\\
=&\int_M \langle d(\iota(\nabla f)\omega),d(\iota(\nabla f)\omega)\rangle\,d\mu_g
-\int_M \langle\nabla(\iota(\nabla f)\omega),\nabla(\iota(\nabla f)\omega)\rangle\,d\mu_g
\end{split}
\end{equation}
by Corollary \ref{p2d} (i), Bochner-Weitzenb\"{o}ck formula and the divergence theorem.
By (\ref{2b}) and Corollary \ref{p2d} (ii), we have
\begin{equation}\label{3b}
\begin{split}
&\int_M \langle d(\iota(\nabla f)\omega),d(\iota(\nabla f)\omega)\rangle\,d\mu_g
-\int_M \langle\nabla(\iota(\nabla f)\omega),\nabla(\iota(\nabla f)\omega)\rangle\,d\mu_g\\
=&\frac{p-1}{p}\int_M \langle \iota(\nabla\Delta f)\omega),\iota(\nabla f)\omega\rangle\,d\mu_g-\int_M |T(\iota(\nabla f)\omega)|^2\,d\mu_g
\end{split}
\end{equation}
By (\ref{3a}) and (\ref{3b}), we get the proposition.
\end{proof}

Based on Proposition \ref{p3a}, Grosjean showed Theorem \ref{grosjean}.
Assuming more strong condition on eigenvalues, we remove the assumption that the manifold is simply connected from Theorem \ref{grosjean}.
\begin{Cor}\label{p3d2}
Let $(M,g)$ be an $n$-dimensional closed Riemannian manifold.
Assume that $\Ric\geq (n-p-1)g$ and there exists a non-trivial parallel $p$-form on $M$ $(2\leq p< n/2)$.
If
$
\lambda_{n-p+1}(g)= n-p,
$
then $(M,g)$ is isometric to a product
$S^{n-p}\times (X,g')$,
where $(X,g')$ is some $p$-dimensional closed Riemannian manifold.
\end{Cor}
\begin{proof}
Let $f_k$ be the $k$-th eigenfunction of the Laplacian on $S^{n-p}$.
Note that the functions $f_1,\ldots,f_{n-p+1}$ are height functions.

By Theorem \ref{grosjean}, the universal cover $(\widetilde{M},\tilde{g})$ of $(M,g)$ is isometric to a product $S^{n-p}\times (X,g')$,
where $(X,g')$ is some $p$-dimensional closed Riemannian manifold.
We regard the function $f_i$ as a function on $\widetilde{M}$.
Since $\lambda_{n-p+1}(g)= n-p$, each
$f_i\in C^\infty(\widetilde{M})$ ($i=1,\ldots,n-p+1$) is a pull back of some function on $M$.
Thus, the covering transformation preserves $f_1,\ldots,f_{n-p+1}$.
Therefore, the covering transformation does not act on $S^{n-p}$, and so we get the corollary.
\end{proof}
The almost version of this corollary is Main Theorem 2.

Finally, we show that the assumption of Corollary \ref{p3d2} is optimal in some sense by giving an example.

Take a positive odd integer $p$ with $p\geq 3$ and a positive integer $n$ with $n> 2p$.
Put $a:=\sqrt{(p-1)/(n-p-1)}$. We define an equivalence relation $\sim$ on $S^{n-p}\times S^p(a)$ as follows:
\begin{equation*}
\begin{split}
&((x_0,\ldots,x_{n-p}),(y_0,\ldots,y_p))\sim ((x'_0,\ldots,x'_{n-p}),(y'_0,\ldots,y'_p))\\
\Leftrightarrow &\text{ there exists $k\in \mathbb{Z}$ such that}\\
&((x'_0,\ldots,x'_{n-p}),(y'_0,\ldots,y'_p))=(((-1)^k x_0, x_1,\ldots,x_{n-p}),(-1)^k(y_0,\ldots,y_p))
\end{split}
\end{equation*}
for any $((x_0,\ldots,x_{n-p}),(y_0,\ldots,y_p)), ((x'_0,\ldots,x'_{n-p}),(y'_0,\ldots,y'_p))\in S^{n-p}\times S^p(a)$.
Then, we have the following:
\begin{Prop}\label{p3e}
We have the following properties:
\begin{itemize}
\item $(M,g)=(S^{n-p}\times S^p(a))/\sim$ is an $n$-dimensional closed Riemannian manifold with a non-trivial parallel $p$-form.
\item $\Ric= (n-p-1)g$.
\item $\lambda_{n-p}(g)=n-p$.
\item $(M,g)$ is not isometric to any product Riemannian manifolds.
\end{itemize}
\end{Prop}
\begin{proof}

Let $\omega$ be the volume form on $S^p(a)$.
Since the action on $S^{n-p}\times S^p(a)$ preserves $\omega$, there exists a non-trivial parallel $p$-form on $(M,g)$.
We also denote it by $\omega$.
Since the action on $S^{n-p}\times S^p(a)$ preserves the function
$$
x_i \colon S^{n-p}\times S^p(a)\to \mathbb{R},\,((x_0,\ldots,x_{n-p}),(y_0,\ldots,y_p))\mapsto x_i
$$
for each $i=1,\ldots,n-p$, we have $\lambda_{n-p}(g)=n-p$.

Suppose that $(M,g)$ is isometric to a product $(M^{n-k}_1,g_1)\times (M^{k}_2,g_2)$ ($k\leq n-k$) for some $(n-k)$ and $k$-dimensional closed Riemannian manifolds $(M_1,g_1)$ and $(M_2,g_2)$.
Since we have the irreducible decomposition $T_{[(x,y)]} M\cong T_x S^{n-p}\oplus T_y S^p(a)$ of the restricted holonomy action, we get $k=p$.
Since $\lambda_1(g)=n-p$, we have that $(M_1,g_1)$ is isometric to $S^{n-p}$.
Thus, we get $\lambda_{n-p+1}(g)=n-p$.
However the action on $S^{n-p}\times S^p(a)$ does not preserve the function
$$
x_0\colon  S^{n-p}\times S^p(a)\to \mathbb{R},\,((x_0,\ldots,x_{n-p}),(y_0,\ldots,y_p))\mapsto x_0,
$$
and so $\lambda_{n-p+1}(g)\neq n-p$.
This is a contradiction.
\end{proof}

\subsection{Error Estimates}
In this subsection, we give error estimates about Proposition \ref{p3a}.
Lemma \ref{p4e} (vii) corresponds to Proposition \ref{p3a}.

We list the assumptions of this subsection.
\begin{Asu}
In this subsection, we assume the following:
\begin{itemize}
\item $(M,g)$ is an $n$-dimensional closed Riemannian manifold with $\Ric_g\geq -Kg$ and $\diam(M)\leq D$ for some positive real numbers $K>0$ and $D>0$.
\item $1\leq k \leq n-1$.
\item A $k$-form $\omega\in \Gamma(\bigwedge^k T^\ast M)$ satisfies $\|\omega\|_2=1$, $\|\omega\|_\infty\leq L_1$ and $\|\nabla \omega\|_2^2\leq \lambda$ for some $L_1>0$ and $0\leq \lambda\leq 1$.
\item A function $f\in C^\infty(M)$ satisfies $\|f\|_{\infty}\leq L_2\|f\|_2$, $\|\nabla f\|_{\infty}\leq L_2\|f\|_2$ and $\|\Delta f\|_2\leq L_2\|f\|_2$ for some $L_2>0$.
\end{itemize}
\end{Asu}

Note that we have
\begin{equation}\label{4a0}
\|\nabla^2 f\|_2^2=\|\Delta f\|_2^2-\frac{1}{\Vol(M)}\int_M \Ric(\nabla f,\nabla f)\,d\mu_g\leq (1+K)L^2_2\|f\|_2^2
\end{equation}
by the Bochner formula.

We first show the following:
\begin{Lem}\label{p4c}
There exists a positive constant $C(n,K,D)>0$ such that
$\||\omega|-1\|_2\leq C \lambda^{1/2}$ holds.
\end{Lem}
\begin{proof}

Put
$
\overline{\omega}:=\int_M |\omega| \,d\mu_g/\Vol(M).
$
Since we have $|\omega|\in W^{1,2}(M)$, we get
$$
\||\omega|-\overline{\omega}\|_2^2\leq \frac{1}{\lambda_1(g)}\|\nabla|\omega|\|_2^2\leq \frac{1}{\lambda_1(g)}\|\nabla\omega\|_2^2\leq\frac{\lambda}{\lambda_1(g)}\leq C\lambda
$$
by the Kato inequality and the Li-Yau estimate \cite[p.116]{SY}.
Therefore, we get
$$
|1-\overline{\omega}|=\left|\|\omega\|_2-\|\overline{\omega}\|_2\right|\leq \||\omega|-\overline{\omega}\|_2\leq C\lambda^{1/2},
$$
and so
$
\||\omega|-1\|_2\leq C\lambda^{1/2}.
$
\end{proof}

Let us give error estimates about Proposition \ref{p3a}.
\begin{Lem}\label{p4d}
There exists a positive constant $C=C(n,k,K,D,L_1,L_2)>0$ such that the following properties hold:
\begin{itemize}
\item[(i)] We have
$$
\frac{1}{\Vol(M)} \int_M |d^{\ast}(\iota(\nabla f)\omega)|^2\,d\mu_g
\leq C\|f\|_2^2\lambda.
$$
\item[(ii)] We have
$$
\left|\frac{1}{\Vol(M)}\int_M \Big(\langle \iota(\Ric(\nabla f))\omega,\iota(\nabla f)\omega\rangle -\langle \mathcal{R}_{k-1}(\iota(\nabla f)\omega),\iota(\nabla f)\omega\rangle \Big)\,d\mu_g\right|
\leq C\|f\|_2^2\lambda^{1/2}.
$$
\item[(iii)] We have
$$
\left|\frac{1}{\Vol(M)} \int_M \Big(\langle\Delta(\iota(\nabla f)\omega),\iota(\nabla f)\omega\rangle
-\langle
\iota(\nabla \Delta f)\omega,\iota(\nabla f)\omega\rangle\Big) \,d\mu_g\right|
\leq C\|f\|_2^2\lambda^{1/2}.
$$

\item[(iv)] We have
$$
\frac{1}{\Vol(M)}\int_M \left|\nabla (\iota(\nabla f)\omega)-\sum_{i=1}^n e^i\otimes \iota(\nabla_{e_i}\nabla f)\omega\right|^2 \,d\mu_g\leq  C\|f\|_2^2\lambda.
$$
\item[(v)] We have
$$
\frac{1}{\Vol(M)}\int_M \left|d (\iota(\nabla f)\omega)-\sum_{i=1}^n e^i\wedge \iota(\nabla_{e_i}\nabla f)\omega\right|^2 \,d\mu_g\leq C\|f\|_2^2\lambda.
$$
\item[(vi)] We have
$$
\frac{1}{\Vol(M)}\int_M |\nabla (\iota(\nabla f)\omega) |^2\,d\mu_g\leq C\|f\|_2^2.
$$
\item[(vii)] We have
\begin{align*}
&\Bigg|\frac{1}{\Vol(M)}\int_M \langle \iota(\Ric(\nabla f))\omega,\iota(\nabla f)\omega\rangle
\,d\mu_g\\
&\quad-
\frac{k-1}{k} \frac{1}{\Vol(M)}\int_M \langle \iota(\nabla\Delta f)\omega,\iota(\nabla f)\omega\rangle \,d\mu_g+\|T(\iota(\nabla f)\omega)\|_2^2\Bigg|
\leq C\|f\|_2^2\lambda^{1/2}.
\end{align*}
\item[(viii)] If $M$ is oriented and $1\leq k\leq n/2$, then we have
\begin{align*}
&\frac{1}{\Vol(M)}\int_M \Ric(\nabla f,\nabla f)|\omega|^2\,d\mu_g\\
\leq &
\frac{n-k-1}{n-k} \frac{1}{\Vol(M)}\int_M \langle \nabla\Delta f,\nabla f\rangle|\omega|^2 \,d\mu_g
-\|T(\iota(\nabla f)\omega)\|_2^2
-\|T(\iota(\nabla f)\ast\omega)\|_2^2\\
&\qquad\qquad -\left(\frac{n-k-1}{n-k} -\frac{k-1}{k} \right)\|d(\iota(\nabla f)\omega)\|^2_2
+C\|f\|_2^2\lambda^{1/2}.
\end{align*}
\end{itemize}
Although an orthonormal basis $\{e_1,\ldots,e_n\}$ of $TM$ is defined only locally,
$\sum_{i=1}^n e^i\otimes \iota(\nabla_{e_i}\nabla f)\omega$ and $\sum_{i=1}^n e^i\wedge \iota(\nabla_{e_i}\nabla f)\omega$ are well-defined as tensors.
\end{Lem}
\begin{proof}
We first prove (i).
Since $d^\ast (f\omega)=-\iota(\nabla f)\omega +f d^\ast \omega$
and $d^\ast\circ d^\ast=0$, we have
$
d^\ast (\iota(\nabla f)\omega)=-\iota(\nabla f)d^\ast \omega.
$
Thus, we get
$$
\frac{1}{\Vol(M)} \int_M |d^{\ast}(\iota(\nabla f)\omega)|^2\,d\mu_g
\leq C\|\nabla f\|_{\infty}^2 \|\nabla \omega\|_2^2
\leq C\|f\|_2^2 \lambda.
$$

To prove (ii) and (iii), we estimate following terms:
\begin{align*}
&\frac{1}{\Vol(M)} \int_M \langle\iota(\nabla f)\Delta \omega,\iota(\nabla f) \omega\rangle\,d\mu_g,\\
&\frac{1}{\Vol(M)} \int_M \langle\iota(\nabla f)\nabla^\ast \nabla \omega,\iota(\nabla f) \omega\rangle\,d\mu_g,\\
&\frac{1}{\Vol(M)} \int_M \langle \iota(\nabla f) \mathcal{R}_k \omega, \iota(\nabla f)\omega\rangle\,d\mu_g,\\
&\frac{1}{\Vol(M)} \int_M \langle \sum_{i=1}^n\iota(\nabla_{e_i}\nabla f) (\nabla_{e_i}\omega),\iota(\nabla f)\omega\rangle\,d\mu_g,\\
&\frac{1}{\Vol(M)} \int_M \langle\sum_{i=1}^n\iota(e_i)(R(\nabla f,e_i)\omega),\iota(\nabla f) \omega\rangle\,d\mu_g.
\end{align*}

We have
\begin{align*}
&\int_M \langle\iota(\nabla f)\Delta \omega,\iota(\nabla f) \omega\rangle\,d\mu_g\\
=&\int_M \langle d \omega,d (d f \wedge\iota(\nabla f) \omega)\rangle\,d\mu_g+\int_M \langle d^\ast\omega,d^\ast(d f \wedge\iota(\nabla f) \omega)\rangle\,d\mu_g
\end{align*}
and
\begin{align*}
&|\langle d \omega,d (d f \wedge\iota(\nabla f) \omega)\rangle|\\
=&|\langle d \omega, \sum_{i=1}^n d f\wedge e^i \wedge\left(\iota(\nabla_{e_i}\nabla f) \omega+\iota(\nabla f) \nabla_{e_i}\omega \right)\rangle|\\
\leq& C|\nabla \omega||\nabla f|(|\nabla^2 f||\omega|+|\nabla f||\nabla \omega|),\\
&|\langle d^\ast \omega,d^\ast (d f \wedge\iota(\nabla f) \omega)\rangle|\\
=&|\langle d^\ast \omega, \sum_{i=1}^n \iota(e_i)\left( \nabla_{e_i} d f\wedge  \iota(\nabla f) \omega+ d f\wedge  \iota(\nabla_{e_i} \nabla f) \omega+
d f\wedge\iota(\nabla f) \nabla_{e_i}\omega \right)\rangle|\\
\leq& C|\nabla \omega||\nabla f|(|\nabla^2 f||\omega|+|\nabla f||\nabla \omega|).
\end{align*}
Thus, we get
\begin{equation}\label{4a}
\left|\frac{1}{\Vol(M)} \int_M \langle\iota(\nabla f)\Delta \omega,\iota(\nabla f) \omega\rangle\,d\mu_g\right|
\leq C\|f\|_2^2\lambda^{1/2}.
\end{equation}

We have
\begin{align*}
\int_M \langle\iota(\nabla f)\nabla^\ast\nabla \omega,\iota(\nabla f) \omega\rangle\,d\mu_g
=\int_M \langle \nabla \omega,\nabla (d f \wedge\iota(\nabla f) \omega)\rangle\,d\mu_g
\end{align*}
and
$
|\langle \nabla \omega,\nabla (d f \wedge\iota(\nabla f) \omega)\rangle|
\leq C|\nabla \omega||\nabla f|(|\nabla^2 f||\omega|+|\nabla f||\nabla \omega|).
$
Thus, we get
\begin{equation}\label{4aa}
\left|\frac{1}{\Vol(M)} \int_M \langle\iota(\nabla f)\nabla^\ast\nabla \omega,\iota(\nabla f) \omega\rangle\,d\mu_g\right|
\leq C\|f\|_2^2\lambda^{1/2}.
\end{equation}

By Theorem \ref{p2b}, (\ref{4a}) and (\ref{4aa}), we have
\begin{equation}\label{4b}
\begin{split}
&\left|\frac{1}{\Vol(M)} \int_M \langle \iota(\nabla f) \mathcal{R}_k \omega, \iota(\nabla f)\omega\rangle\,d\mu_g\right|\\
\leq & \frac{1}{\Vol(M)}\left(\left|\int_M \langle \iota(\nabla f)\Delta \omega, \iota(\nabla f) \omega\rangle\,d\mu_g\right|+ \left|\int_M \langle \iota(\nabla f)\nabla^\ast\nabla \omega, \iota(\nabla f)\omega\rangle\,d\mu_g\right|\right)\\
\leq & C\|f\|_2^2\lambda^{1/2}.
\end{split}
\end{equation}

Since
$
|\langle \sum_{i=1}^n\iota(\nabla_{e_i}\nabla f) (\nabla_{e_i}\omega),\iota(\nabla f)\omega\rangle|
\leq C|\omega||\nabla f| |\nabla \omega||\nabla^2 f|,
$
we have
\begin{equation}\label{4c}
\left|\frac{1}{\Vol(M)} \int_M \langle \sum_{i=1}^n\iota(\nabla_{e_i}\nabla f) (\nabla_{e_i}\omega),\iota(\nabla f)\omega\rangle\,d\mu_g\right|
\leq C \|f\|_2^2\lambda^{1/2}.
\end{equation}

Let us estimate
\begin{equation*}
\frac{1}{\Vol(M)} \int_M \langle\sum_{i=1}^n\iota(e_i)(R(\nabla f,e_i)\omega),\iota(\nabla f) \omega\rangle\,d\mu_g.
\end{equation*}
We have
\begin{equation*}
\begin{split}
\left|\frac{1}{\Vol(M)} \int_M \langle \nabla_{\nabla f} d^\ast \omega,\iota(\nabla f) \omega\rangle\,d\mu_g\right|&
=\left|\frac{1}{\Vol(M)} \int_M \langle  d^\ast \omega,\nabla^\ast (d f\otimes\iota(\nabla f) \omega)\rangle\,d\mu_g\right|\\
&\leq C\|f\|_2^2\lambda^{1/2},\\
\left|\frac{1}{\Vol(M)} \int_M \langle  d^\ast \nabla_{\nabla f}\omega,\iota(\nabla f) \omega\rangle\,d\mu_g\right|&
=\left|\frac{1}{\Vol(M)} \int_M \langle   \nabla \omega, d f \otimes d (\iota(\nabla f) \omega)\rangle\,d\mu_g\right|\\
&\leq C\|f\|_2^2\lambda^{1/2}
\end{split}
\end{equation*}
and
$$\left|\frac{1}{\Vol(M)} \int_M \langle \sum_{i,j=1}^n \langle \nabla_{e_j} \nabla f, e_i\rangle\iota(e_j)\nabla_{e_i}\omega,\iota(\nabla f)\omega\rangle\,d\mu_g\right|\\
\leq C\|f\|_2^2\lambda^{1/2}.$$
Thus, by Lemma \ref{p2c} (iii), we get
\begin{equation}\label{4d}
\left|\frac{1}{\Vol(M)} \int_M \langle\sum_{i=1}^n\iota(e_i)(R(\nabla f,e_i)\omega),\iota(\nabla f) \omega\rangle\,d\mu_g\right|
\leq C\|f\|_2^2\lambda^{1/2}.
\end{equation}

By (\ref{4a}), (\ref{4b}), (\ref{4c}), (\ref{4d}) and Lemma \ref{p2c}, we get (ii) and (iii).

Since
$
\nabla (\iota(\nabla f)\omega)-\sum_{i=1}^n e^i\otimes \iota(\nabla_{e_i}\nabla f)\omega
=\sum_{i=1}^n e^i\otimes\iota(\nabla f)\nabla_{e_i}\omega,
$
we get (iv) and (vi).

Since
$
d (\iota(\nabla f)\omega)-\sum_{i=1}^n e^i\wedge \iota(\nabla_{e_i}\nabla f)\omega
=\sum_{i=1}^n e^i\wedge\iota(\nabla f)\nabla_{e_i}\omega,
$
we get (v).


By Theorem \ref{p2b} and (\ref{2b}), we have
\begin{equation*}
\begin{split}
&\frac{1}{\Vol(M)}\int_M \langle \mathcal{R}_{k-1}(\iota(\nabla f)\omega),\iota(\nabla f)\omega\rangle \,d\mu_g\\
=&\frac{1}{\Vol(M)}\int_M \langle (\Delta-\nabla^\ast\nabla)(\iota(\nabla f)\omega),\iota(\nabla f)\omega\rangle \,d\mu_g\\
=& \frac{k-1}{k} \frac{1}{\Vol(M)}\int_M \langle \Delta (\iota(\nabla f)\omega), \iota(\nabla f)\omega\rangle\,d\mu_g\\
&+\left(\frac{n-k+1}{n-k+2}-\frac{k-1}{k} \right)\|d^\ast(\iota(\nabla f)\omega)\|_2^2-\|T(\iota(\nabla f)\omega)\|_2^2.
\end{split}
\end{equation*}
Thus, by (i), (ii) and (iii), we get (vii)

Finally, we prove (viii). Suppose that $M$ is oriented and $1\leq k\leq n/2$.
Since $\nabla (\ast \omega)=\ast\nabla \omega$, we have
\begin{align*}
&\frac{1}{\Vol(M)}\int_M \langle \iota(\Ric(\nabla f))\ast\omega,\iota(\nabla f)\ast\omega\rangle
\,d\mu_g\\
\leq &
\frac{n-k-1}{n-k} \frac{1}{\Vol(M)}\int_M \langle \iota(\nabla\Delta f)\ast\omega,\iota(\nabla f)\ast\omega\rangle \,d\mu_g-\|T(\iota(\nabla f)\ast\omega)\|_2^2+C\|f\|_2^2\lambda^{1/2}
\end{align*}
by (vii). Thus, by (\ref{hstar}), (i), (iii) and (vii), we get
\begin{align*}
&\frac{1}{\Vol(M)}\int_M \Ric(\nabla f,\nabla f)|\omega|^2\,d\mu_g\\
\leq &
\frac{n-k-1}{n-k} \frac{1}{\Vol(M)}\int_M \langle \nabla\Delta f,\nabla f\rangle|\omega|^2 \,d\mu_g
-\|T(\iota(\nabla f)\omega)\|_2^2
-\|T(\iota(\nabla f)\ast\omega)\|_2^2\\
&\qquad -\left(\frac{n-k-1}{n-k} -\frac{k-1}{k} \right)
\frac{1}{\Vol(M)}\int_M \langle \iota(\nabla\Delta f)\omega,\iota(\nabla f)\omega\rangle \,d\mu_g
+C\|f\|_2^2\lambda^{1/2}\\
\leq &
\frac{n-k-1}{n-k} \frac{1}{\Vol(M)}\int_M \langle \nabla\Delta f,\nabla f\rangle|\omega|^2 \,d\mu_g
-\|T(\iota(\nabla f)\omega)\|_2^2
-\|T(\iota(\nabla f)\ast\omega)\|_2^2\\
&\qquad\qquad -\left(\frac{n-k-1}{n-k} -\frac{k-1}{k} \right)\|d(\iota(\nabla f)\omega)\|^2_2
+C\|f\|_2^2\lambda^{1/2}.
\end{align*}
This gives (viii).
\end{proof}

\subsection{Eigenvalue Estimate}
In this subsection, we complete the proofs of Main Theorems 1 and 3.
Recall that $\lambda_1(\Delta_{C,p})$ denotes the first eigenvalue of the connection Laplacian $\Delta_{C,p}$ acting on $p$-forms:
$$\Delta_{C,p}:=\nabla^\ast\nabla \colon \Gamma(\bigwedge^p T^\ast M)\to\Gamma(\bigwedge^p T^\ast M).$$
It is enough to show Main Theorem 1 when $\lambda_1(\Delta_{C,p})\leq 1$.
Note that we always have
$
\lambda_1(\Delta_{C,1})\geq 1
$
if $\Ric_g\geq (n-1)g$.

We need the following $L^\infty$ estimates.
\begin{Lem}\label{Linfes}
Take an integer $n\geq 2$ and positive real numbers $K>0$, $D>0$, $\Lambda>0$. 
Let $(M,g)$ be an $n$-dimensional closed Riemannian manifold with $\Ric\geq-Kg$ and $\diam(M)\leq D$. Then, we have the following:
\begin{itemize}
\item[(i)] For any function $f\in C^\infty(M)$ and any $\lambda\geq 0$ with $\Delta f=\lambda f$ and $\lambda\leq \Lambda$, then we have $\|\nabla f\|_\infty\leq C(n,K,D,\Lambda)\|f\|_2$ and $\|f\|_\infty\leq C(n,K,D,\Lambda)\|f\|_2$.
\item[(ii)] For any $p$-form $\omega\in \Gamma\left(\bigwedge^p T^\ast M\right)$ and any $\lambda\geq 0$ with $\Delta_{C,p} \omega=\lambda \omega$ and $\lambda\leq \Lambda$, then we have $\|\omega\|_\infty\leq C(n,K,D,\Lambda)\|\omega\|_2$.
\end{itemize}
\end{Lem}
\begin{proof}
By the gradient estimate for eigenfunctions  \cite[Theorem 7.3]{Pe1},
we get (i).

Let us show (ii).
Since we have
\begin{equation*}
\Delta |\omega|^2=2\langle \Delta_{C,p} \omega, \omega \rangle-2|\nabla \omega|^2\leq 2 \Lambda |\omega|^2,
\end{equation*}
we get $\|\omega\|_\infty\leq C$ by  \cite[Proposition 9.2.7]{Pe3} (see also Propositions 7.1.13 and 7.1.17 in \cite{Pe3}).
Note that our sign convention of the Laplacian is different from \cite{Pe3}.
\end{proof}

We use the following proposition not only for the proofs of Main Theorems 1 and 3 but also for other main theorems.
\begin{Prop}\label{p4e}
For given integers $n\geq 4$ and $2\leq p \leq n/2$, there exists a constant $C(n,p)>0$ such that the following property holds.
Let $(M,g)$ be an $n$-dimensional closed oriented Riemannian manifold with $\Ric_g\geq (n-p-1)g$.
Suppose that an integer $i\in\mathbb{Z}_{>0}$ satisfies $\lambda_i(g)\leq n-p+1$, and there exists an eigenform $\omega$ of the connection Laplacian $\Delta_{C,p}$ acting on $p$-forms with $\|\omega\|_2=1$ corresponding to the eigenvalue $\lambda$ with $0\leq \lambda\leq 1$.
Then, we have
\begin{align*}
&\frac{n-p-1}{n-p}\lambda_i(g)\left(\lambda_i(g)-(n-p)\right)\|f_i\|^2\\
\geq&\|T(\iota(\nabla f_i)\omega)\|_2^2
+\|T(\iota(\nabla f_i)\ast\omega)\|_2^2\\
&+\left(\frac{n-p-1}{n-p} -\frac{p-1}{p} \right)\|d(\iota(\nabla f_i)\omega)\|^2_2
-C\lambda^{1/2}\|f_i\|_2^2,
\end{align*}
where $f_i$ denotes the $i$-th eigenfunction of the Laplacian acting on functions.
\end{Prop}
\begin{proof}
By Lemma \ref{p4d} (viii), we have
\begin{align*}
&\frac{n-p-1}{\Vol(M)}\int_M \langle\nabla f_i,\nabla f_i\rangle|\omega|^2\,d\mu_g\\
\leq&\frac{1}{\Vol(M)}\int_M \Ric(\nabla f_i,\nabla f_i)|\omega|^2\,d\mu_g\\
\leq&\frac{n-p-1}{n-p}\frac{\lambda_i(g)}{\Vol(M)}\int_M \langle\nabla f_i,\nabla f_i\rangle|\omega|^2\,d\mu_g
-\|T(\iota(\nabla f_i)\omega)\|_2^2
-\|T(\iota(\nabla f_i)\ast\omega)\|_2^2\\
&\qquad\qquad -\left(\frac{n-p-1}{n-p} -\frac{p-1}{p} \right)\|d(\iota(\nabla f_i)\omega)\|^2_2
+C\lambda^{1/2}\|f_i\|_2^2.
\end{align*}
Thus, we get the proposition by Lemma \ref{p4c}.
\end{proof}

\begin{proof}[Proof of Main Theorem 1]
If $M$ is orientable, we get the theorem immediately by Proposition \ref{p4e}.
If $M$ is not orientable, we get the theorem by considering the two-sheeted orientable Riemannian covering $\pi\colon (\widetilde{M},\tilde{g})\to (M,g)$ because
we have
$
\lambda_1(g)\geq\lambda_1(\tilde{g})
$
and
$
\lambda_1(\Delta_{C,p},g)\geq \lambda_1(\Delta_{C,p},\tilde{g}).
$
\end{proof}
Similarly, we get Main Theorem 3 because $\lambda_1(\Delta_{C,p},g)=\lambda_1(\Delta_{C,n-p},g)$ holds if the manifold is orientable.

\section{Pinching}
In this section, we show the remaining main theorems.
Main Theorem 2 is proved in subsection 4.5 except for the orientability, and the orientability is proved in subsection 4.7.
Main Theorem 4 is proved in subsection 4.8.

We list assumptions of this section.
\begin{Asu}\label{asu1}
Throughout in this section, we assume the following:
\begin{itemize}
\item $n\geq 5$, $2\leq p < n/2$ and $1\leq k\leq n-p+1$.
\item $(M,g)$ is an $n$-dimensional closed Riemannian manifold with $\Ric_g\geq (n-p-1)g$.
\item $C=C(n,p)>0$ denotes a positive constant depending only on $n$ and $p$.
\item $\delta>0$ satisfies $\delta\leq \delta_0$ for sufficiently small $\delta_0=\delta_0(n,p)>0$.
\item $f_i\in C^\infty(M)$ ($i\in\{1,\ldots,k\}$) is an eigenfunction of the Laplacian acting on functions with $\|f_i\|_2^2=1/(n-p+1)$ corresponding to the eigenvalue $\lambda_i$ with $0<\lambda_i\leq n-p+\delta$ such that
$$
\int_M f_i f_j\,d\mu_g=0
$$
holds for any $i\neq j$.
\end{itemize}
\end{Asu}
Note that, for given real numbers $a,b$ with $0<b<a$ and a positive constant $C>0$, we can assume that
$
C \delta^a\leq\delta^b.
$
At the beginning of each subsections, we add either one of the following assumptions if necessary.
\begin{Asu}\label{aspform}
There exists an eigenform $\omega\in\Gamma(\bigwedge^p T^\ast M)$ of the connection Laplacian $\Delta_{C,p}$ with $\|\omega\|_2=1$ corresponding to the eigenvalue $\lambda$ with $0\leq \lambda \leq \delta$.
\end{Asu}
\begin{Asu}\label{asn-pform}
There exists an eigenform $\xi\in\Gamma(\bigwedge^{n-p} T^\ast M)$ of the connection Laplacian $\Delta_{C,n-p}$ with $\|\xi\|_2=1$ corresponding to the eigenvalue $\lambda$ with $0\leq \lambda \leq \delta$.
\end{Asu}

Under our assumptions, we have $\|\omega\|_\infty\leq C$, $\|\xi\|_{\infty} \leq C$, $\|f_i\|_\infty \leq C $ and $\|\nabla f_i\|_\infty \leq C$ for all $i$
by Lemma \ref{Linfes}.
By Main Theorems 1 and 3, we have
$\lambda_i\geq n-p-C(n,p)\delta^{1/2}$
for all $i$.
Note that we do not assume that $\lambda_i=\lambda_i(g)$.

\subsection{Useful Techniques}
In this subsection, we list some useful techniques for our pinching problems.
Although we suppose that Assumption \ref{asu1} holds, most assertions hold under weaker assumptions.

The following lemma is a variation of the Cheng-Yau estimate.
See \cite[Lemma 2.10]{Ai2} for the proof (see also \cite[Theorem 7.1]{Ch}).
\begin{Lem}\label{chya}
Take a positive real number $0<\epsilon_1 \leq1$. For any function
$f\in \Span_{\mathbb{R}}\{f_1,\ldots,f_k\}$ and any point $x\in M$,
we have 
\begin{equation*}
|\nabla f|^2(x)\leq \frac{C}{\epsilon_1}\left(f(p)-f(x)+\epsilon_1\|f\|_2\right)^2,
\end{equation*} 
where $p\in M$ denotes a maximum point of $f$.
\end{Lem}
The following theorem is an easy consequence of the Bishop-Gromov inequality.
\begin{Thm}\label{bigr}
For any $p\in M$ and $0<r\leq \diam(M)+1$, we have $r^n \Vol(M)\leq C\Vol(B_r(p))$.
\end{Thm}

The following theorem is due to Cheeger-Colding \cite{CC2} (see also \cite[Theorem 7.1.10]{Pe3}).
By this theorem, we get integral pinching conditions along the geodesics under the integral pinching condition for a function on $M$.
\begin{Thm}[segment inequality]\label{seg}
For any non-negative measurable function $h\colon M\to \mathbb{R}_{\geq 0}$, we have
\begin{equation*}
\frac{1}{\Vol(M)^2}\int_{M\times M} \frac{1}{d(y_1,y_2)}\int_0^{d(y_1,y_2)} h\circ \gamma_{y_1,y_2}(s) \,dsdy_1dy_2\leq  \frac{C}{\Vol(M)}\int_M h\,d\mu_g.
\end{equation*}
\end{Thm}
\begin{Rem}
The book \cite{Pe3} deals with the segment $c_{y_1,y_2}\colon[0,1]\to M$ for each $y_1,y_2\in M$, defined to be
$c_{y_1,y_2}(0)=y_1$, $c_{x,y}(1)=y_2$ and $\nabla_{\partial /\partial t} \dot{c}=0$.
We have $c_{x,y}(t)=\gamma_{x,y}(t d(x,y))$ for all $t\in[0,1]$ and
$$
d(y_1,y_2)\int_0^1 h\circ c_{y_1,y_2}(t) \,d t=\int_0^{d(y_1,y_2)} h\circ \gamma_{y_1,y_2}(s) \,d s.
$$
\end{Rem}

After getting integral pinching conditions along the geodesics, we use the following lemma to get $L^\infty$ error estimate along them.
The proof is standard (c.f. \cite[Lemma 2.41]{CC2}).
\begin{Lem}\label{trif}
Take positive real numbers $l,\epsilon>0$  and a non-negative real number $r\geq 0$.
Suppose that a smooth function $u\colon [0,l]\to \mathbb{R}$ satisfies
$$\int_0^l |u''(t)+r^2 u(t)| \,dt\leq\epsilon.$$
Then, we have
\begin{equation*}
\begin{split}
\left|u(t)-u(0) \cos r t- \frac{u'(0)}{r} \sin r t\right|&\leq \epsilon\frac{\sinh rt}{r},\\
\left|u'(t)+ r u(0)\sin r t- u'(0)\cos r t\right|&\leq \epsilon+\int_0^t\left|u(s)-u(0)\cos r s-\frac{u'(0)}{r}\sin r s\right|\,ds,
\end{split}
\end{equation*}
for all $t\in [0,l]$, where we defined
$
\frac{1}{r}\sin r t:=t,$
$\frac{1}{r}\sinh r t:=t
$
if $r=0$.
\end{Lem}

The following lemma is standard.
\begin{Lem}\label{cosi}
For all $t\in \mathbb{R}$, we have
\begin{equation*}
1-\frac{1}{2}t^2\leq \cos t\leq 1-\frac{1}{2}t^2+\frac{1}{24}t^4.
\end{equation*}
For any $t\in [-\pi,\pi]$, we have $\cos t\leq 1-\frac{1}{9}t^2$, and so $|t|\leq3(1-\cos t)^{1/2}$.
For any $t_1,t_2 \in [0,\pi]$, we have $|t_1-t_2|\leq3|\cos t_1-\cos t_2|^{1/2}$.
\end{Lem}

Finally, we recall some facts about the geodesic flow.
Let $U M$ denotes the sphere bundle defined by
$$
U M:=\{u\in TM:|u|=1\}.
$$
There exists a natural Riemannian metric $G$ on $UM$, which is the restriction of the Sasaki metric on $TM$ (see \cite[p.55]{Sa}).
The Riemannian volume measure $\mu_G$ satisfies
$$
\int_{UM} F\,d\mu_G=\int_M \int_{U_p M} F(u)\, d\mu_0(u) \,d\mu_g(p)
$$
for any $F\in C^\infty(U M)$, where $\mu_0$ denotes the standard measure on $U_p M\cong S^{n-1}$.
The geodesic flow $\phi_t\colon U M\to U M$ ($t\in\mathbb{R}$) is defined by
$$
\phi_t(u):=\left.\frac{\partial}{\partial s}\right|_{s=t}\gamma_u (s)\in U_{\gamma_u(t)} M
$$
for any $u\in U M$.
Though $\phi_t$ does not preserve the metric $G$ in general, it preserves the measure $\mu_G$.
This is an easy consequence of \cite[Lemma 4.4]{Sa}, which asserts that the geodesic flow on $T M$ preserve the natural symplectic structure on $T M$.
We can easily show the following lemma.
\begin{Lem}\label{geofl}
For any $f\in C^\infty (M)$ and $l>0$, we have
$$
\frac{1}{\Vol(M)}\int_M f \,d\mu_g=\frac{1}{l\Vol(UM)}\int_{UM}\int_0^l f\circ\gamma_u(t)\,d t\,d\mu_G(u).
$$
\end{Lem}
This kind of lemma was used by Colding \cite{Co1} to prove that the almost equality of the Bishop comparison theorem implies the Gromov-Hausdorff closeness to the standard sphere.
\subsection{Estimates for the Segments}
In this subsection, we suppose that Assumption \ref{aspform} holds.
The goal is to give error estimates along the geodesics.
We first list some basic consequences of our pinching condition.
\begin{Lem}\label{p5c}
For any $f\in \Span_{\mathbb{R}}\{f_1,\ldots,f_{k}\}$, we have
\begin{itemize}
\item[(i)] $\|\iota(\nabla f)\omega\|_2^2\leq C\delta^{1/2}\|f\|_2^2$,
\item[(ii)] $\|\nabla(\iota(\nabla f)\omega)\|_2^2\leq C\delta^{1/2}\|f\|_2^2$,
\item[(iii)] $\|(|\nabla^2 f|^2-\frac{1}{n-p}|\Delta f|^2)|\omega|^2\|_1\leq C\delta^{1/4}\|f\|_2^2$.
\end{itemize}
\end{Lem}
\begin{proof}
It is enough to consider the case when $M$ is orientable.

We first assume that $f=f_i$ for some $i=1,\ldots,k$.
Then, we have
\begin{equation}\label{5a0}
\begin{split}
&\|d(\iota(\nabla f)\omega)\|^2_2\leq C\delta^{1/2}\|f\|_2^2,\\
&\|d^\ast (\iota(\nabla f)\omega)\|^2_2 \leq C\delta^{1/2}\|f\|_2^2,\quad
\|T(\iota(\nabla f)\omega)\|_2^2\leq C\delta^{1/2}\|f\|_2^2,\\
&\|d^\ast (\iota(\nabla f)\ast \omega)\|^2_2 \leq C\delta^{1/2}\|f\|_2^2,\quad
\|T(\iota(\nabla f)\ast\omega)\|_2^2 \leq C\delta^{1/2}\|f\|_2^2
\end{split}
\end{equation}
by Lemma \ref{p4d} (i) and Proposition \ref{p4e}. 
Thus, by (\ref{2b}), we get
\begin{equation}\label{5a}
\|\nabla (\iota(\nabla f)\omega)\|^2_2\leq C\delta^{1/2}\|f\|_2^2
\end{equation}
and
\begin{equation}\label{5b}
\|\nabla (\iota(\nabla f)\ast\omega)\|^2_2\leq \frac{1}{n-p} \|d (\iota(\nabla f)\ast \omega)\|^2_2+C\delta^{1/2}\|f\|_2^2.
\end{equation}
Moreover, by Lemma \ref{p4d} (iii), we have
\begin{equation}\label{5c}
\begin{split}
\|\iota(\nabla f)\omega\|_2^2
=&\frac{1}{\lambda_i}\frac{1}{\Vol(M)}\int_M \langle \iota(\nabla \Delta f)\omega, \iota(\nabla f)\omega\rangle\,d\mu_g\\
\leq& C\|d(\iota(\nabla f)\omega)\|^2_2+C\|d^\ast (\iota(\nabla f)\omega)\|^2_2+C\delta^{1/2}\|f\|_2^2\\
\leq& C\delta^{1/2}\|f\|_2^2.
\end{split}
\end{equation}

For any $f=a_1 f_1+\cdots + a_k f_k\in \Span_{\mathbb{R}}\{f_1,\ldots,f_{k}\}$,
we have (\ref{5a0}), (\ref{5a}), (\ref{5b}), (\ref{5c}).
For example,
we have
\begin{equation*}
\|\nabla (\iota(\nabla f)\omega)\|_2\leq\sum_{i=1}^k |a_k|\|\nabla (\iota(\nabla f_i)\omega)\|_2\leq C\delta^{1/4}\sum_{i=1}^k |a_k|\|f_i\|_2\leq C\delta^{1/4}\|f\|_2.
\end{equation*}
Thus, we get (i) and (ii) by (\ref{5a}) and (\ref{5c}).

Finally, we prove (iii).
Take arbitrary $f\in \Span_{\mathbb{R}}\{f_1,\ldots,f_{k}\}$.
We have
\begin{equation}\label{5ca}
\begin{split}
&\left|\sum_{i=1}^n e^i\otimes \iota(\nabla_{e_i}\nabla f)\ast\omega\right|^2\\
=&\sum_{i=1}^n \langle \iota(\nabla_{e_i} \nabla f)\ast\omega,\iota(\nabla_{e_i} \nabla f)\ast\omega\rangle
=|\nabla^2 f|^2|\omega|^2-\left|\sum_{i=1}^n e^i\otimes \iota(\nabla_{e_i}\nabla f)\omega\right|^2.
\end{split}
\end{equation}
Thus, we have
\begin{equation*}
\begin{split}
&\frac{1}{\Vol(M)}\int_M \left||\nabla(\iota(\nabla f)\ast \omega)|^2-|\nabla^2 f|^2|\omega|^2\right|\,d\mu_g\\
\leq &\frac{1}{\Vol(M)}\int_M \left||\nabla(\iota(\nabla f)\ast \omega)|^2-\left|\sum_{i=1}^n e^i\otimes \iota(\nabla_{e_i}\nabla f)\ast\omega\right|^2\right| \,d\mu_g\\
&\qquad+\frac{1}{\Vol(M)}\int_M \left|\sum_{i=1}^n e^i\otimes \iota(\nabla_{e_i}\nabla f)\omega\right|^2\,d\mu_g,
\end{split}
\end{equation*}
and so we get
\begin{equation}\label{5d}
\frac{1}{\Vol(M)}\int_M \left||\nabla(\iota(\nabla f)\ast \omega)|^2-|\nabla^2 f|^2|\omega|^2\right|\,d\mu_g
\leq C\delta^{1/2}\|f\|_2^2
\end{equation}
by (ii) and Lemma \ref{p4d} (iv) and (vi).
We have
\begin{equation}\label{5e}
\begin{split}
&\left|\sum_{i=1}^n e^i\wedge \iota(\nabla_{e_i}\nabla f)\ast\omega\right|^2\\
=&\sum_{i=1}^n |\iota(\nabla_{e_i} \nabla f)\ast\omega|^2-\sum_{i,j=1}^n \langle \iota(e_i)\iota(\nabla_{e_j}\nabla f)\ast \omega, \iota(e_j)\iota(\nabla_{e_i}\nabla f)\ast \omega \rangle\\
=&|\nabla^2 f|^2|\omega|^2-\left|\sum_{i=1}^n e^i\otimes \iota(\nabla_{e_i}\nabla f)\omega\right|^2\\
&\qquad -\sum_{i,j,k,l=1}^n \nabla^2 f(e_i,e_k)\nabla^2 f(e_j,e_l)\langle e^i\wedge e^l \wedge \omega,
e^j\wedge e^k \wedge \omega \rangle
\end{split}
\end{equation}
by (\ref{5ca}) and (\ref{hstar2}). Since
\begin{align*}
\langle e^i\wedge e^l \wedge \omega,
e^j\wedge e^k \wedge \omega \rangle
=&(\delta_{i j}\delta_{k l}-\delta_{i k}\delta_{j l})|\omega|^2
-\delta_{i j}\langle \iota(e_k)\omega,\iota(e_l)\omega\rangle\\
&+\delta_{i k}\langle \iota(e_j)\omega,\iota(e_l)\omega\rangle
+\langle e^l\wedge \omega,e^j\wedge e^k\wedge\iota(e_i)\omega\rangle,
\end{align*}
we have
\begin{equation}\label{5f}
\begin{split}
&\sum_{i,j,k,l=1}^n \nabla^2 f(e_i,e_k)\nabla^2 f(e_j,e_l)\langle e^i\wedge e^l \wedge \omega,
e^j\wedge e^k \wedge \omega \rangle\\
=&|\nabla^2 f|^2|\omega|^2-(\Delta f)^2|\omega|^2
-\sum_{i=1}^n | \iota(\nabla_{e_i}\nabla f)\omega|^2
-\sum_{i=1}^n\Delta f \langle \iota(\nabla_{e_i}\nabla f)\omega,\iota(e_i)\omega\rangle\\
&\qquad+\sum_{j,k,l=1}^n \nabla^2 f(e_j,e_l)\langle e^l\wedge \omega,e^j\wedge e^k\wedge\iota(\nabla_{e_k} \nabla f)\omega\rangle.
\end{split}
\end{equation}
By (\ref{5e}) and (\ref{5f}), we get
\begin{equation*}
\begin{split}
\left|\sum_{i=1}^n e^i\wedge \iota(\nabla_{e_i}\nabla f)\ast\omega\right|^2
=&(\Delta f)^2|\omega|^2+\sum_{i=1}^n\Delta f \langle \iota(\nabla_{e_i}\nabla f)\omega,\iota(e_i)\omega\rangle\\
-&\sum_{j,k,l=1}^n \nabla^2 f(e_j,e_l)\langle e^l\wedge \omega,e^j\wedge e^k\wedge\iota(\nabla_{e_k} \nabla f)\omega\rangle,
\end{split}
\end{equation*}
and so 
\begin{equation}\label{5g}
\left|\left|\sum_{i=1}^n e^i\wedge \iota(\nabla_{e_i}\nabla f)\ast\omega\right|^2-(\Delta f)^2|\omega|^2\right|
\leq C|\nabla^2 f| |\omega|\left|\sum_{i=1}^n e^i\otimes \iota(\nabla_{e_i}\nabla f)\omega\right|
\end{equation}
By (\ref{5g}), (ii) and Lemma \ref{p4d}, we get
\begin{equation}\label{5h}
\frac{1}{\Vol(M)}\int_M \left|
|d (\iota(\nabla f)\ast\omega)|^2-(\Delta f)^2|\omega|^2\right|\,d\mu_g
\leq C\delta^{1/4}\|f\|_2^2.
\end{equation}
Since we have
$
|\nabla (\iota(\nabla f)\ast\omega)|^2\geq |d (\iota(\nabla f)\ast \omega)|^2/(n-p)
$
at each point by (\ref{2b}), we get (iii) by (\ref{5b}), (\ref{5d}) and (\ref{5h}).
\end{proof}
We use the following notation.
\begin{notation}\label{np5d}
Take $f\in \Span_{\mathbb{R}}\{f_1,\ldots,f_{k}\}$ with $\|f\|_2^2=1/(n-p+1)$ and put
\begin{align*}
h_0&:=|\nabla^2 f|^2, \quad h_1:=||\omega|^2-1|, \quad h_2:=|\nabla \omega|^2,\\
h_3&:=|\iota(\nabla f)\omega |^2, \quad h_4:=|\nabla (\iota(\nabla f)\omega)|^2,\quad h_5:=\left|\sum_{i=1}^n e^i\otimes\iota(\nabla_{e_i}\nabla f)\omega\right|^2\\
h_6&:=\left||\nabla^2 f|^2-\frac{1}{n-p}(\Delta f)^2\right||\omega|^2.
\end{align*}
For each $y_1\in M$, we define
\begin{align*}
D_f(y_1):=&\Big\{y_2\in I_{y_1}\setminus\{y_1\}:\frac{1}{d(y_1,y_2)}\int_{0}^{d(y_1,y_2)} h_0\circ \gamma_{y_1,y_2}(s)\,d s\leq \delta^{-1/50} \text{ and}\\
&\qquad \quad\frac{1}{d(y_1,y_2)}\int_{0}^{d(y_1,y_2)} h_i\circ \gamma_{y_1,y_2}(s)\,d s\leq \delta^{1/5} \text{ for all $i=1,\ldots,6$}
\Big\},\\
Q_f:=&\{y_1\in M: \Vol(M\setminus D_f(y_1))\leq\delta^{1/100}\Vol(M)\},\\
E_f(y_1):=&\Big\{u\in U_{y_1} M: \frac{1}{\pi }\int_{0}^{\pi} h_0\circ \gamma_u(s)\,d s\leq \delta^{-1/50} \text{ and }\frac{1}{\pi}\int_{0}^{\pi} h_i\circ \gamma_u (s)\,d s\leq \delta^{1/5} \\
&\qquad \quad\qquad \quad\qquad \quad\qquad \quad\qquad \quad\qquad \quad\qquad \quad\text{ for all $i=1,\ldots,6$}
\Big\},\\
R_f:=&\{y_1\in M: \Vol(U_{y_1} M\setminus E_f(y_1))\leq\delta^{1/100}\Vol(U_{y_1}M)\}.
\end{align*}
\end{notation}

Now, we use the segment inequality and Lemma \ref{geofl}.
We show that we have the integral pinching condition along most geodesics.
\begin{Lem}\label{p5d}
Take $f\in \Span_{\mathbb{R}}\{f_1,\ldots,f_{k}\}$ with $\|f\|_2^2=1/(n-p+1)$.
Then, we have the following properties:
\begin{itemize}
\item[(i)] $\Vol(M\setminus Q_f)\leq C\delta^{1/100}\Vol(M).$
\item[(ii)] $\Vol(M\setminus R_f)\leq C\delta^{1/100}\Vol(M).$
\end{itemize}
\end{Lem}
\begin{proof}
We have
$\|h_i\|_1\leq C\delta^{1/4}$ for all $i=1,\ldots,6$ by the assumption, Lemmas \ref{p4c}, \ref{p4d} (iv) and \ref{p5c}, and we have
$\|h_0\|_1\leq C$ by (\ref{4a0}).

For any $y_1\in M\setminus Q_f$, we have $\Vol(M\setminus D_f(y_1))>\delta^{1/100}\Vol(M)$, and so we have either
\begin{equation*}
\frac{1}{\Vol(M)}\int_M\frac{1}{d(y_1,y_2)}\int_0^{d(y_1,y_2)}h_0\circ \gamma_{y_1,y_2}(s)\,d s \,d y_2\geq \frac{1}{7}\delta^{-1/100}
\end{equation*}
or
\begin{equation*}
\frac{1}{\Vol(M)}\int_M\frac{1}{d(y_1,y_2)}\int_0^{d(y_1,y_2)}h_i\circ \gamma_{y_1,y_2}(s)\,d s \,d y_2\geq\frac{1}{7}\delta^{21/100}
\end{equation*}
for some $i=1,\ldots,6$.
Thus, we get either
\begin{equation*}
\frac{1}{\Vol(M)}\int_M \int_M\frac{1}{d(y_1,y_2)}\int_0^{d(y_1,y_2)}h_0\circ \gamma_{y_1,y_2}(s)\,d s \,d y_1\,d y_2\geq \frac{1}{49}\delta^{-1/100}\Vol(M\setminus Q_f)
\end{equation*}
or
\begin{equation*}
\frac{1}{\Vol(M)}\int_M \int_M\frac{1}{d(y_1,y_2)}\int_0^{d(y_1,y_2)}h_i\circ \gamma_{y_1,y_2}(s)\,d s \,d y_1\,d y_2\geq \frac{1}{49}\delta^{21/100}\Vol(M\setminus Q_f)
\end{equation*}
for some $i=1,\ldots,6$.
Therefore, we get (i) by the segment inequality (Theorem \ref{seg}).

Similarly, we get (ii) by Lemma \ref{geofl}.
\end{proof}

Under the pinching condition along the geodesic, we get the following:
\begin{Lem}\label{p5e}
Take $f\in \Span_{\mathbb{R}}\{f_1,\ldots,f_{k}\}$ with $\|f\|_2^2=1/(n-p+1)$.
Suppose that a geodesic $\gamma\colon [0,l]\to M$ satisfies one of the following:
\begin{itemize}
\item There exist $x\in M$ and $y\in D_f(x)$ such that $l=d(x,y)$ and $\gamma=\gamma_{x,y}$,
\item There exist $x\in M$ and $u\in E_f(x)$ such that $l=\pi$ and $\gamma=\gamma_u$.
\end{itemize}
Then, we have
$$
||\omega|^2(s)-1|\leq C\delta^{1/10},\quad |\iota(\nabla f)\omega|(s)\leq C\delta^{1/10}
$$
for all $s\in [0,l]$, and at least one of the following:
\begin{itemize}
\item[(i)] $\frac{1}{l}\int_0^l|\nabla^2 f|\circ \gamma(s)\,d s\leq C\delta^{1/250}$,
\item[(ii)] There exists a parallel orthonormal basis $\{E^1(s),\ldots,E^n(s)\}$ of $T_{\gamma(s)}^\ast M$ along $\gamma$ such that
$$
|\omega-E^{n-p+1}\wedge\cdots\wedge E^n|(s)\leq C\delta^{1/25}
$$
for all $s\in[0,l]$, and
$$
\frac{1}{l}\int_0^l|\nabla^2 f+f\sum_{i=1}^{n-p}E^i\otimes E^i|(s)\, d s\leq C\delta^{1/200},
$$
where we write
$|\cdot|(s)$ instead of $|\cdot|\circ\gamma(s)$.
\end{itemize}
In particular, for both cases,
there exists a parallel orthonormal basis $\{E^1(s),\ldots,E^n(s)\}$ of $T_{\gamma(s)}^\ast M$ along $\gamma$ such that
$$
\frac{1}{l}\int_0^l|\nabla^2 f+f\sum_{i=1}^{n-p}E^i\otimes E^i|(s)\, d s\leq C\delta^{1/250}.
$$

Moreover, if we put
$\dot{\gamma}^E:=\sum_{i=1}^{n-p} \langle\dot{\gamma},E_i\rangle E_i,$
where $\{E_1,\ldots,E_n\}$ denotes the dual basis of $\{E^1,\ldots,E^n\}$,
then $|\dot{\gamma}^E|$ is constant along $\gamma$, and
\begin{equation*}
\begin{split}
\left|f\circ \gamma(s)-f(\gamma(s_0))\cos (|\dot{\gamma}^E|(s-s_0))-\frac{1}{|\dot{\gamma}^E|}\langle\nabla f,\dot{\gamma}(s_0)\rangle\sin (|\dot{\gamma}^{E}|(s-s_0))\right|&\leq C\delta^{1/250},\\
\left| \langle \nabla f, \dot{\gamma}(s)\rangle+f(\gamma(s_0))|\dot{\gamma}^{E}|\sin (|\dot{\gamma}^{E}|(s-s_0))-\langle\nabla f,\dot{\gamma}(s_0)\rangle\cos (|\dot{\gamma}^{E}|(s-s_0))\right|&\leq C\delta^{1/250}
\end{split}
\end{equation*}
for all $s,s_0\in[0,l]$.
\end{Lem}
\begin{proof}
Let us show the first assertion.
Since
$\frac{d}{d s}|\omega|^2(s)=2\langle\nabla_{\dot{\gamma}}\omega,\omega\rangle$,
we have
\begin{align*}
\left||\omega|^2(s)-|\omega|^2(0)\right|
=&\left|\int_0^s \frac{d}{d s}|\omega|^2(t)\,d t\right|\\
\leq& 2 \left(\int_0^s |\nabla \omega|^2 (t)\,d t\right)^{1/2} \left(\int_0^s |\omega|^2 (t)\, d t\right)^{1/2}
\leq C\delta^{1/10}
\end{align*}
for all $s\in[0,l]$.
Since we have $\int_0^l||\omega|^2-1|\, d t \leq \delta^{1/5}$, we get $||\omega|^2(s)-1|\leq C\delta^{1/10}$.
In particular,
$|\omega|(s)\geq 1/2$, and so
\begin{equation}\label{5i}
\frac{1}{l}\int_0^l\left||\nabla^2 f|^2-\frac{1}{n-p}(\Delta f)^2\right|(s)\,d s\leq 2\delta^{1/5}.
\end{equation}
Similarly, we have $|\iota(\nabla f)\omega|(s)\leq C\delta^{1/10}$ for all $s\in [0,l]$.

We show the remaining assertions.
Put
\begin{align*}
A_1:=&\left\{s\in [0,l]:\left|\sum_{i=1}^n e^i\otimes\iota(\nabla_{e_i}\nabla f)\omega\right|^2(s)>\delta^{1/10}\right\},\\
A_2:=&\left\{s\in [0,l]:\left||\nabla^2 f|^2-\frac{1}{n-p}(\Delta f)^2\right|(s)>\delta^{1/10}\right\},\\
A_3:=&\left\{s\in [0,l]:|\nabla^2 f|(s)<\delta^{1/250}\right\}.
\end{align*}
Then, we have
$H^1(A_1)\leq \delta^{1/10}l$ and
$H^1(A_2)\leq 2\delta^{1/10} l$, where $H^1$ denotes the one dimensional Hausdorff measure.
We consider the following two cases:
\begin{itemize}
\item[(a)] $[0,l]=A_1\cup A_2\cup A_3$,
\item[(b)] $[0,l]\neq A_1\cup A_2\cup A_3$.
\end{itemize}

We first consider the case (a).
Since
$H^1([0,l]\setminus A_3)\leq 3 \delta^{1/10} l,$
we have
\begin{align*}
\int_{[0,l]\setminus A_3}|\nabla^2 f|(s)\,d s
\leq&
\left(\int_{[0,l]\setminus A_3}|\nabla^2 f|^2(s)\,d s\right)^{1/2}H^1 ([0,l]\setminus A_3)^{1/2}\\
\leq &C \delta^{-1/100}\delta^{1/20} l=C\delta^{1/25}l.
\end{align*}
On the other hand, we have
$
\int_{A_3} |\nabla^2 f|(s)\,d s\leq \delta^{1/250} l.
$
Therefore, we get (i).
Moreover, 
since $|\Delta f|\leq \sqrt{n}|\nabla^2 f|$ and $\left\|\Delta f-(n-p)f\right\|_{\infty}\leq C\delta^{1/2}$, we get
$$
\frac{1}{l}\int_0^l|\nabla^2 f+f\sum_{i=1}^{n-p}E^i\otimes E^i|(s)\, d s\leq C\delta^{1/250},
$$
where $\{E^1(s),\ldots,E^n(s)\}$ is any parallel orthonormal basis of $T_{\gamma(s)}^\ast M$ along $\gamma$.

We next consider the case (b).
There exists $t\in[0,l]$ such that
\begin{align*}
\left|\sum_{i=1}^n e^i\otimes\iota(\nabla_{e_i}\nabla f)\omega\right|^2(t)&\leq\delta^{1/10},\\
\left||\nabla^2 f|^2-\frac{1}{n-p}(\Delta f)^2\right|(t)&\leq\delta^{1/10},\quad |\nabla^2 f|(t)\geq\delta^{1/250}.
\end{align*}
Take an orthonormal basis $\{e_1,\ldots,e_n\}$ of $T_{\gamma(t)}M$ such that
$\nabla^2 f(e_i,e_j)=\mu_i\delta_{i j}\, (\mu_i\in\mathbb{R})$
for all $i,j=1,\ldots,n$.
Let $\{e^1,\ldots,e^n\}$ be the dual basis of $T_{\gamma(t)}^\ast M$.
Then, we have
$$
\delta^{1/10}\geq \left|\sum_{i=1}^n e^i\otimes\iota(\nabla_{e_i}\nabla f)\omega\right|^2(t)
=\sum_{i=1}^n\mu_i^2 |\iota(e_i)\omega|^2(t).
$$
Thus, for each $i=1,\ldots,n$, we have at least one of the following:
\begin{itemize}
\item[(1)] $|\mu_i|\leq \delta^{1/100}$,
\item[(2)] $|\iota(e_i)\omega|(t)\leq \delta^{1/25}$.
\end{itemize}
Since $|\omega|(t)\geq 1/2$, we have
$\Card \{i: |\iota(e_i)\omega|(t)\leq \delta^{1/25}\}\leq n-p,$
and so
$\Card \{i: |\mu_i|\leq \delta^{1/100}\}\geq p.$
Therefore, we can assume $|\mu_i|\leq \delta^{1/100}$ for all $i=n-p+1,\ldots, n$.
Then, we get
\begin{align*}
\left|
\nabla^2 f+\frac{\Delta f}{n-p}\sum_{i=1}^{n-p} e^i\otimes e^i
\right|^2(t)
=&|\nabla^2 f|^2(t)+\frac{2}{n-p}(\Delta f)(t)\sum_{i=1}^{n-p}\mu_i
+\frac{(\Delta f)^2(t)}{n-p}\\
=&|\nabla^2 f|^2(t)-\frac{(\Delta f)^2(t)}{n-p}-\frac{2}{n-p}(\Delta f)(t)\sum_{i=n-p+1}^{n}\mu_i\\
\leq& C\delta^{1/100}. 
\end{align*}
Putting $e_i\otimes e_i$ into the inside of the left hand side,
we get
$\left|\mu_i+\Delta f(t)/(n-p)\right|^2\leq C\delta^{1/100}$
for all $i=1,\ldots, n-p$, and so
\begin{align*}
|\mu_i|\geq \frac{|\Delta f(t)|}{n-p}-C\delta^{1/200}
\geq &\left(\frac{|\nabla^2 f|^2(t)-\delta^{1/10}}{n-p}\right)^{1/2}-C\delta^{1/200}\\
\geq &\left(\frac{\delta^{1/125}-\delta^{1/10}}{n-p}\right)^{1/2}-C\delta^{1/200}
>\delta^{1/100}.
\end{align*}
Thus, we have
$|\iota(e_i)\omega|(t)\leq \delta^{1/25}$
for all $i=1,\ldots,n-p$.
Therefore, we get either
$|\omega(t)-e^{n-p+1}\wedge\cdots\wedge e^n|\leq C\delta^{1/25}$ or
$|\omega(t)+e^{n-p+1}\wedge\cdots\wedge e^n|\leq C\delta^{1/25}$
by $||\omega|^2(t)-1|\leq C\delta^{1/10}$.
We can assume that $|\omega(t)-e^{n-p+1}\wedge\cdots\wedge e^n|\leq C\delta^{1/25}$.

Let $\{E_1,\ldots,E_n\}$ be the parallel orthonormal basis of $TM$ along $\gamma$ such that
$E_i(t)=e_i$, and let $\{E^1,\ldots,E^n\}$ be its dual.
Because
\begin{align*}
\int_0^l \left|\frac{d}{d s}|\omega-E^{n-p+1}\wedge\cdots \wedge E^n|^2(s)\right|\,d s
\leq C\delta^{1/10},
\end{align*}
we get
$|\omega-E^{n-p+1}\wedge\cdots\wedge E^n|(s)\leq C\delta^{1/25}$
for all $s\in [0,l]$.
Thus, we get
$|\langle\iota(E_i)\omega,\iota(E_j)\omega\rangle|\leq C\delta^{1/25}$
for all $i=1,\cdots,n$ and $j=1,\ldots,n-p$, and
$|\langle\iota(E_i)\omega,\iota(E_j)\omega\rangle-\delta_{i j}|\leq C\delta^{1/25}$
for all $i,j=n-p+1,\cdots,n$.
Therefore, we get
\begin{align*}
&\left|
\left|\sum_{i=1}^n E^i\otimes\iota(\nabla_{E_i}\nabla f)\omega\right|^2-\sum_{i=1}^n\sum_{j=n-p+1}^n(\nabla^2 f(E_i,E_j))^2
\right|\\
=&\left|
\sum_{i,j,k=1}^n \nabla^2 f(E_i,E_j)\nabla^2 f(E_i,E_k)\langle\iota(E_j)\omega,\iota(E_k)\omega\rangle-\sum_{i=1}^n\sum_{j=n-p+1}^n(\nabla^2 f (E_i,E_j))^2
\right|\\
\leq&C |\nabla^2 f|^2 \delta^{1/25}.
\end{align*}
Thus, for all $i=1,\cdots,n$ and $j=1,\ldots,n-p$, we get
$$
|\nabla^2 f(E_i,E_j)|^2\leq \left|\sum_{k=1}^n E^k\otimes\iota(\nabla_{E_k}\nabla f)\omega\right|^2+C |\nabla^2 f|^2 \delta^{1/25},
$$
and so
$$
\frac{1}{l}\int_0^l|\nabla^2 f (E_i,E_j)|^2(s)\,d s
\leq \delta^{1/5}+C\delta^{-1/50}\delta^{1/25}
\leq C\delta^{1/50}.
$$
This gives
$$
\frac{1}{l}\int_0^l|\nabla^2 f (E_i,E_j)|(s)\,d s
\leq C\delta^{1/100}
$$
for all $i=1,\cdots,n$ and $j=1,\ldots,n-p$.
Because
\begin{align*}
\left|
\nabla^2 f+\frac{\Delta f}{n-p}\sum_{i=1}^{n-p}E^i\otimes E^i
\right|^2
=|\nabla^2 f|^2-\frac{(\Delta f)^2}{n-p}-2\frac{\Delta f}{n-p}\sum_{i=n-p+1}^{n}\nabla^2 f (E_i,E_i),
\end{align*}
we have
$$
\frac{1}{l}\int_0^l\left|
\nabla^2 f+\frac{\Delta f}{n-p}\sum_{i=1}^{n-p}E^i\otimes E^i
\right|^2\,d s
\leq 2\delta^{1/5}+C\delta^{1/100}\leq C\delta^{1/100}
$$
by (\ref{5i}). Since 
$\left\|f-\Delta f/(n-p)\right\|_{\infty}\leq C\delta^{1/2}$,
we get (ii).

Let us show the final assertion.
It is trivial that $|\dot{\gamma}^E|$ is constant along $\gamma$.
Since we have
$$
\left(\nabla^2 f+f \sum_{i=1}^{n-p}E^i\otimes E^i\right)(\dot{\gamma},\dot{\gamma})
=\frac{d^2}{d s^2} f\circ \gamma + |\dot{\gamma}^E|^2 f\circ \gamma,
$$
we get
\begin{equation*}
\int_0^l\left|\frac{d^2}{d s^2} f\circ \gamma(s) + |\dot{\gamma}^E|^2 f\circ \gamma(s)\right|\,d s\leq C\delta^{1/250}.
\end{equation*}
Thus, we get the lemma by Lemma \ref{trif}.
\end{proof}

\subsection{Almost Parallel $(n-p)$-form I}
In this subsection, we suppose that Assumption \ref{asn-pform} holds instead of \ref{aspform}.
If $M$ is orientable, then Assumption \ref{asn-pform} implies \ref{aspform}, and so we assume that $M$ is not orientable. We use the following notation.
\begin{notation}\label{np5f}
Take $f\in\Span_{\mathbb{R}}\{f_1,\ldots, f_k\}$ with $\|f\|_2^2=1/(n-p+1)$.
Let $\pi\colon (\widetilde{M},\tilde{g})\to (M,g)$ be the two-sheeted oriented Riemannian covering.
Put
$
\tilde{f}:=f\circ \pi\in C^\infty(\widetilde{M})$,
$\widetilde{\xi}:=\pi^\ast \xi\in\Gamma(\bigwedge^{n-p}T^\ast \widetilde{M})$
and $\omega:=\ast \widetilde{\xi}\in\Gamma(\bigwedge^{p}T^\ast \widetilde{M})$.
Define $h_0,\ldots,h_6$, $Q_{\tilde{f}}$, $D_{\tilde{f}}(\tilde{y}_1)$, $R_{\tilde{f}}$ and $E_{\tilde{f}}(\tilde{y_1})$
as Notation \ref{np5d} for $\tilde{f}$, $\omega$ and $\tilde{y}_1\in \widetilde{M}$.
Put
\begin{align*}
Q_f:=&M\setminus \pi\left(\widetilde{M}\setminus Q_{\tilde{f}}\right),\quad D_f(y_1):=&&M\setminus \pi\left(\widetilde{M}\setminus\bigcap_{\tilde{y}\in\pi^{-1}(y_1)} D_{\tilde{f}}(\tilde{y})\right),\\
R_f:=&M\setminus \pi\left(\widetilde{M}\setminus R_{\tilde{f}}\right),\quad E_f(y_1):=&&U_{y_1} M\setminus \bigcup_{\tilde{y}\in\pi^{-1}(y_1)}\pi_\ast\left(U_{\tilde{y}}\widetilde{M}\setminus E_{\tilde{f}}(\tilde{y})\right)
\end{align*}
for each $y_1\in M$.
\end{notation}
We immediately have the following lemmas by Lemmas \ref{p5d} and \ref{p5e}.
\begin{Lem}\label{p5f}
We have the following:
\begin{itemize}
\item[(i)] $\Vol(M\setminus Q_f)\leq C\delta^{1/100}\Vol(M)$, and $\Vol(M\setminus D_f(y_1))\leq2\delta^{1/100}\Vol(\widetilde{M})=4\delta^{1/100}\Vol(M)$ for each $y_1\in Q_f$.
\item[(ii)] $\Vol(M\setminus R_f)\leq C\delta^{1/100}\Vol(M)$, and $\Vol(U_{y_1} M\setminus E_f(y_1))\leq2\delta^{1/100}\Vol(U_{y_1}M)$ for each $y_1\in R_f$.
\item[(iii)] Take $y_1\in M$ and $y_2\in D_f(y_1)$ and one of the lift of $\gamma_{y_1,y_2}$:
$$
\tilde{\gamma}_{y_1,y_2}\colon[0,d(y_1,y_2)]\to \widetilde{M}.
$$
Put $\tilde{y}_1:=\tilde{\gamma}_{y_1,y_2}(0)\in \widetilde{M}$ and $\tilde{y}_2:=\tilde{\gamma}_{y_1,y_2}(d(y_1,y_2))\in \widetilde{M}$.
Then, we have $\tilde{y}_2\in D_{\tilde{f}}(\tilde{y}_1)$.
\item[(iv)] Take $y_1\in M$ and $u\in E_f(y_1)$ and one of the lift of $\gamma_u$:
$$
\tilde{\gamma}_{u}\colon[0,\pi]\to \widetilde{M}.
$$
Put $\tilde{y}_1:=\tilde{\gamma}_{u}(0)\in \widetilde{M}$ and $\tilde{u}:=\dot{\tilde{\gamma}}_{u}(0)\in U_{\tilde{y}_1}\widetilde{M}$.
Then, we have $\tilde{u}\in E_{\tilde{f}}(\tilde{y}_1)$.
\end{itemize}
\end{Lem}
\begin{Lem}\label{p5g}
Suppose that a geodesic $\gamma\colon [0,l]\to M$ satisfies one of the following:
\begin{itemize}
\item There exist $x\in M$ and $y\in D_f(x)$ such that $l=d(x,y)$ and $\gamma=\gamma_{x,y}$,
\item There exist $x\in M$ and $u\in E_f(x)$ such that $l=\pi$ and $\gamma=\gamma_u$.
\end{itemize}
Let $\tilde{\gamma}\colon [0,l]\to\widetilde{M}$ be one of the lift of $\gamma$.
Then, we have
$$
||\omega|^2(\tilde{\gamma}(s))-1|\leq C\delta^{1/10},\quad |\iota(\nabla \tilde{f})(\omega)|\circ\tilde{\gamma}(s)\leq C\delta^{1/10}
$$
for all $s\in [0,l]$, and at least one of the following:
\begin{itemize}
\item[(i)] $\frac{1}{l}\int_0^l|\nabla^2 f|\circ \gamma(s)\,d s\leq C\delta^{1/250}$,
\item[(ii)] There exists a parallel orthonormal basis $\{E^1(s),\ldots,E^n(s)\}$ of $T_{\gamma(s)}^\ast M$ along $\gamma$ such that
$$
|\xi-E^{1}\wedge\cdots\wedge E^{n-p}|(s)\leq C\delta^{1/25}
$$
for all $s\in[0,s]$, and
$$
\frac{1}{l}\int_0^l|\nabla^2 f+f\sum_{i=1}^{n-p}E^i\otimes E^i|(s)\, d s\leq C\delta^{1/200}.
$$
\end{itemize}
\end{Lem}

\subsection{Eigenfunction and Distance}
In this subsection, we suppose that either Assumption \ref{aspform} or \ref{asn-pform} holds.
In the following, Lemma \ref{p5d} (resp. \ref{p5e}) shall be replaced by Lemma \ref{p5f} (resp. \ref{p5g}) under Assumption \ref{asn-pform}.
The following proposition, which asserts that our function is an almost cosine function in some sense, is the goal of this subsection.
See Notation \ref{np5d} (under Assumption \ref{aspform}) and Notation \ref{np5f} (under Assumption \ref{asn-pform}) for the definitions of $D_f$, $Q_f$, $E_f$ and $R_f$.

\begin{Prop}\label{p53a}
Take $f\in \Span_{\mathbb{R}}\{f_1,\ldots,f_{k}\}$ with $\|f\|_2^2=1/(n-p+1)$.
There exists a point $p_f\in Q_f$ such that the following properties hold:
\begin{itemize}
\item[(i)] $\sup_M f\leq f(p_f)+C\delta^{1/100n}$ and $|f(p_f)-1|\leq C\delta^{1/800n}$,
\item[(ii)] For any $x\in D_f(p_f)$ with $|\nabla f|(x)\leq \delta^{1/800n}$, we have
$
||f(x)|-1|\leq C\delta^{1/800n}.
$
\item[(iii)] For any $x\in D_f(p_f)\cap Q_f\cap R_f$,
we have
$
|f(x)^2+|\nabla f|^2(x)-1|\leq C \delta^{1/800n}.
$
\item[(iv)] Put
$
A_f:=\{x\in M: |f(x)-1|\leq \delta^{1/900n}\}.
$
Then, we have
$$
|f(x)-\cos d(x,A_f)|\leq C\delta^{1/2000n}
$$
for all $x\in M$,
and
$
\sup_{x\in M}d(x,A_f)\leq \pi+ C\delta^{1/100n}.
$
\end{itemize}
\end{Prop}
\begin{proof}

Take a maximum point $\tilde{p}\in M$ of $f$.
Then, by the Bishop-Gromov theorem and Lemma \ref{p5d}, there exists a point $p_f\in Q_f$ with
$d(\tilde{p},p_f)\leq C \delta^{1/100n}$.
By Lemmas \ref{chya} and \ref{Linfes}, we have
\begin{equation}\label{54b}
|\nabla f|(p_f)\leq C\delta^{1/200n}.
\end{equation}
\begin{Clm}\label{c0}
For any $x\in D_f(p_f)$ with $|\nabla f|(x)\leq C\delta^{1/800n}$, we have
$$
||f(x)|-|f(p_f)||\leq C\delta^{1/800n}.
$$
\end{Clm}
\begin{proof}[Proof of Claim \ref{c0}]
Since
$|\nabla f|(p_f)\leq C\delta^{1/200n}$ and $|\nabla f|(x)\leq C\delta^{1/800n},$
we get
\begin{align*}
|f\circ \gamma_{p_f,x}(s)-f(p_f)\cos ( |\dot{\gamma}_{p_f,x}^E| s)|&\leq C\delta^{1/200n},\\
|f\circ \gamma_{p_f,x}(d(p_f,x)-s)-f(x)\cos ( |\dot{\gamma}_{p_f,x}^E|s)|&\leq C\delta^{1/800n}
\end{align*}
for all $s\in[0,d(p_f,x)]$ by Lemma \ref{p5e}.
Thus, we have
\begin{align*}
|f(x)-f(p_f)\cos ( |\dot{\gamma}_{p_f,x}^E| d(p_f,x))|&\leq C\delta^{1/200n},\\
|f(p_f)-f(x)\cos ( |\dot{\gamma}_{p_f,x}^E|d(p_f,x))|&\leq C\delta^{1/800n},
\end{align*}
and so we get $||f(x)|-|f(p_f)||\leq C\delta^{1/800n}$.
\end{proof}
Similarly to $p_f$, we take a point $q_f\in Q_{f}(x)$ with $d(\tilde{q},q_f)\leq C\delta^{1/100n}$, where $\tilde{q}\in M$ is minimum point of $f$.
By $\|f\|_{\infty}\geq\|f\|_2=1/\sqrt{n-p+1}$, we have 
$\max\{|f(p_f)|,|f(q_f)|\}\geq 1/\sqrt{n-p+1}-C\delta^{1/100n}$.
Since $|\nabla f|(q_f)\leq C\delta^{1/200n}$, we have
$|f(p_f)|\geq |f(q_f)|-C\delta^{1/800n}$ by Claim \ref{c0}.
Therefore, we get
\begin{equation}\label{54c0}
f(p_f)\geq \frac{1}{\sqrt{n-p+1}}-C\delta^{1/800n}\geq\frac{1}{2\sqrt{n-p+1}}.
\end{equation}

\begin{Clm}\label{c1}
Take $x\in M$ and $y\in D_f(x)$.
Let $\{E^1,\ldots,E^n\}$ be a parallel orthonormal basis along $\gamma_{x,y}$ in Lemma \ref{p5e}.
If $(i)$ holds in the lemma, we can assume that $E_1=\dot{\gamma}_{x,y}$.
Then, we have
\begin{align}
\label{54ba}|\langle\nabla f,\dot{\gamma}_{x,y}(s)\rangle-\langle
\nabla f,\dot{\gamma}_{x,y}^{E}(s)\rangle|&\leq C\delta^{1/25},\\
\label{54c} |\langle\nabla f,\dot{\gamma}_{x,y}(s)\rangle|&\leq |\nabla f(\gamma_{x,y}(s))||\dot{\gamma}_{x,y}^{E}|+C\delta^{1/25}
\end{align}
and
\begin{equation*}
\begin{split}
\left|f\circ \gamma_{x,y}(s)-f(x)\cos (|\dot{\gamma}_{x,y}^{E}|s)-\frac{1}{|\dot{\gamma}_{x,y}^{E}|}\langle\nabla f(x),\dot{\gamma}_{x,y}(0)\rangle\sin (|\dot{\gamma}_{x,y}^{E}|s)\right|&\leq C\delta^{1/250},\\
\left| \langle \nabla f, \dot{\gamma}_{x,y}(s)\rangle+f(x)|\dot{\gamma}_{x,y}^{E}|\sin (|\dot{\gamma}_{x,y}^{E}|s)-\langle\nabla f(x),\dot{\gamma}_{x,y}(0)\rangle\cos (|\dot{\gamma}_{x,y}^{E}|s)\right|&\leq C\delta^{1/250}
\end{split}
\end{equation*}
for all $s\in[0,d(x,y)]$.
\end{Clm}
\begin{proof}[Proof of Claim \ref{c1}]
If (i) holds in the lemma, $\dot{\gamma}_{x,y}=\dot{\gamma}_{x,y}^E$, and so (\ref{54ba}) and (\ref{54c}) are trivial.
If (ii) in the lemma holds, we have
$|\iota(\nabla f)(E^{n-p+1}\wedge\cdots\wedge E^n)|\leq C\delta^{1/25}$, and so
$|\langle\nabla f(x),E_i\rangle|\leq C\delta^{1/25}$
for all $i=n-p+1,\ldots,n$.
This gives (\ref{54ba}) and (\ref{54c}).
We get the remaining part of the claim by Lemma \ref{p5e} putting $s_0=0$.
\end{proof}
\begin{Clm}\label{c2}
For any $x\in Q_f\cap R_f$ with $|\nabla f|(x)\geq \delta^{1/800n}$,
we have 
\begin{equation*}
|f(x)^2+|\nabla f|^2(x)-f(p_f)^2|\leq C\delta^{1/800n}.
\end{equation*}
Moreover, there exists a point $y\in D_f(p_f)\cap D_f(x)$ such that the following properties hold.
\begin{itemize}
\item[(a)] $d(x,y)< \pi$,
\item[(b)] $|f(p_f)-f(y)|\leq C \delta^{1/800n}$,
\item[(c)] $|f(x)-f(p_f)\cos d(x,y)|\leq C \delta^{1/800n},$
\item[(d)] For any $z\in M$ with $d(x,z)\leq d(x,y)-\delta^{1/2000n}$,
we have $f(p_f)-f(z)\geq \frac{1}{C}\delta^{1/1000n}$.
\end{itemize}
\end{Clm}
\begin{proof}[Proof of Claim \ref{c2}]
Take $x\in Q_f\cap R_f$ with $|\nabla f|(x)\geq \delta^{1/800n}$.
By the definition of $R_f$, there exists a vector $u\in E_f(x)$ with
\begin{equation*}
\left|
\frac{\nabla f}{|\nabla f|}(x)-u
\right|\leq C \delta^{1/100n}.
\end{equation*}
Thus, we have
\begin{equation}\label{54d}
\Big|\langle\nabla f(x),\dot{\gamma}_u(0)\rangle-|\nabla f|(x)\Big|=|\nabla f|(x)-\langle\nabla f(x), u\rangle\leq  C\delta^{1/100n}.
\end{equation}
Let $\{E^1,\ldots,E^n\}$ be a parallel orthonormal basis along $\gamma_{u}$ in Lemma \ref{p5e}.
We first suppose that (ii) holds in the lemma.
Then, for all $i=n-p+1,\ldots, n$, we have $|\langle\nabla f, E_i\rangle|\leq C\delta^{1/25}$, and so
$$
|\langle u,E_i\rangle|\leq \left| u-\frac{\nabla f}{|\nabla f|}(x)\right|+\left|\langle \frac{\nabla f}{|\nabla f|}(x), E_i\rangle\right|\leq C\delta^{1/100n}+C\delta^{1/25}\delta^{-1/800n}\leq C\delta^{1/100n}.
$$
Thus, we get
$|\dot{\gamma}_u^E|^2=|u^E|^2=1-\sum_{i=n-p+1}^n\langle u, E_i\rangle^2\geq 1-C\delta^{1/100n}.$
If (i) holds in the lemma, we can assume $u=E_1$, and so $|\dot{\gamma}_u^E|=|u^E|=1$.
For both cases, we get
\begin{equation}\label{54e}
\begin{split}
|f\circ \gamma_u(s)-f(x)\cos s-|\nabla f|(x)\sin s|\leq& C\delta^{1/100n}\\
|\langle\nabla f,\dot{\gamma}_u(s)\rangle+f(x)\sin s-|\nabla f|(x)\cos s|\leq& C\delta^{1/100n}
\end{split}
\end{equation}
for all $s\in [0,\pi]$ by (\ref{54d}).
Take $s_0\in[0,\pi]$ such that
\begin{align*}
\frac{f(x)}{(f(x)^2+|\nabla f|^2(x))^{1/2}}=&\cos s_0,\\
\frac{|\nabla f|(x)}{(f(x)^2+|\nabla f|^2(x))^{1/2}}=&\sin s_0.
\end{align*}
Since
$\sin s_0\geq \frac{1}{C}\delta^{1/800n}$ by the assumption,
we have
\begin{equation}\label{54ea}
\frac{1}{C}\delta^{1/800n}\leq s_0\leq \pi-\frac{1}{C}\delta^{1/800n}.
\end{equation}
By the definition of $s_0$ and the formulas for $\cos (s-s_0)$ and $\sin(s-s_0)$, we have
\begin{equation*}
\begin{split}
(f(x)^2+|\nabla f|^2(x))^{1/2}\cos (s-s_0)=&f(x)\cos s+|\nabla f|(x)\sin s,\\
(f(x)^2+|\nabla f|^2(x))^{1/2}\sin (s-s_0)=&f(x)\sin s-|\nabla f|(x)\cos s,
\end{split}
\end{equation*}
and so we get
\begin{equation}\label{54f}
\begin{split}
|f\circ \gamma_u(s_0)-(f(x)^2+|\nabla f|^2(x))^{1/2}|\leq& C\delta^{1/100n},\\
|\langle\nabla f,\dot{\gamma}_u(s_0)\rangle|\leq& C\delta^{1/100n}
\end{split}
\end{equation}
by (\ref{54e}).
Take $y\in D_f(p_f)\cap D_f(x)$ with
$d(\gamma_u(s_0),y)\leq C\delta^{1/100n}$.
We have
\begin{equation}\label{54fa}
d(x,y)\leq d(x,\gamma_u(s_0))+d(\gamma_u(s_0),y)\leq s_0+C\delta^{1/100n}.
\end{equation}
By (\ref{54f}), we get
\begin{equation}\label{54fb}
|f(y)-(f(x)^2+|\nabla f|^2(x))^{1/2}|\leq C\delta^{1/100n}
\end{equation}
Take a parallel orthonormal basis $\{\widetilde{E^1},\ldots,\widetilde{E^n}\}$ of $T^\ast M$ along $\gamma_{x,y}$ in Lemma \ref{p5e}.
By (\ref{54ea}) and (\ref{54fa}), we get (a) and
$$
\frac{1}{C}\delta^{1/800n}\leq
|\dot{\gamma}_{x,y}^{\widetilde{E}}|d(x,y)+s_0
\leq 2\pi-\frac{1}{C}\delta^{1/800n},
$$
and so
\begin{equation}\label{54g}
\cos (|\dot{\gamma}_{x,y}^{\widetilde{E}}|d(x,y)+s_0)\leq 1-\frac{1}{C}\delta^{1/400n}.
\end{equation}
If $|\dot{\gamma}_{x,y}^{\widetilde{E}}|\leq \delta^{1/100}$, we have $|f(y)-f(x)|\leq C\delta^{1/250}$ by Claim \ref{c1},
and so
$
(f(x)^2+|\nabla f|^2(x))^{1/2}-f(x)\leq C\delta^{1/100n}
$
by (\ref{54fb}).
This contradicts to
$
|\nabla f|(x)\geq \delta^{1/800n}.
$
Thus, we get $|\dot{\gamma}_{x,y}^{\widetilde{E}}|\geq \delta^{1/100}$.
Then, we have
\begin{equation}\label{54g1}
\frac{1}{|\dot{\gamma}_{x,y}^{\widetilde{E}}|}|\langle\nabla f(x),\dot{\gamma}_{x,y}(0)\rangle|\leq |\nabla f|(x)+C\delta^{3/100}
\end{equation}
and
\begin{align*}
&(f(x)^2+|\nabla f|^2(x))^{1/2}\\
\leq& f(y)+C\delta^{1/100n}\\
\leq& f(x)\cos (|\dot{\gamma}_{x,y}^{\widetilde{E}}|d(x,y))+\frac{1}{|\dot{\gamma}_{x,y}^{\widetilde{E}}|}\langle\nabla f(x),\dot{\gamma}_{x,y}(0)\rangle\sin (|\dot{\gamma}_{x,y}^{\widetilde{E}}|d(x,y))+C\delta^{1/100n}\\
\leq &\left(f(x)^2+\frac{1}{|\dot{\gamma}_{x,y}^{\widetilde{E}}|^2}\langle\nabla f(x),\dot{\gamma}_{x,y}(0)\rangle^2\right)^{1/2}+C\delta^{1/100n}
\end{align*}
by Claim \ref{c1} and (\ref{54fb}).
Thus,
\begin{equation}\label{54g2}
|\nabla f|^2(x)
\leq
\frac{1}{|\dot{\gamma}_{x,y}^{\widetilde{E}}|^2}\langle\nabla f(x),\dot{\gamma}_{x,y}(0)\rangle^2 +C\delta^{1/100n}.
\end{equation}
By (\ref{54g1}) and (\ref{54g2}), we get
\begin{equation}\label{54h0}
\left|\frac{1}{|\dot{\gamma}_{x,y}^{\widetilde{E}}|^2}\langle\nabla f(x),\dot{\gamma}_{x,y}(0)\rangle^2-|\nabla f|^2(x)\right|\leq C\delta^{1/100n}.
\end{equation}
This gives
\begin{equation}\label{54h}
\begin{split}
&\left|\frac{1}{|\dot{\gamma}_{x,y}^{\widetilde{E}}|}|\langle\nabla f(x),\dot{\gamma}_{x,y}(0)\rangle|-|\nabla f|(x)\right|\\
\leq& \left|\frac{1}{|\dot{\gamma}_{x,y}^{\widetilde{E}}|^2}\langle\nabla f(x),\dot{\gamma}_{x,y}(0)\rangle^2-|\nabla f|^2(x)\right|\delta^{-1/800n}\leq
C\delta^{7/800n}.
\end{split}
\end{equation}
We show that $\langle\nabla f(x),\dot{\gamma}_{x,y}(0)\rangle> 0$.
If $\langle\nabla f(x),\dot{\gamma}_{x,y}(0)\rangle\leq 0$,
we get
\begin{equation*}
\left|f(y)-f(x)\cos (|\dot{\gamma}_{x,y}^{\widetilde{E}}|d(x,y))+ |\nabla f|\sin (|\dot{\gamma}_{x,y}^{\widetilde{E}}|d(x,y))\right|\leq C\delta^{7/800n}
\end{equation*}
by (\ref{54h}) and Claim \ref{c1},
and so
\begin{equation*}
\left|f(y)-(f(x)^2+|\nabla f|^2(x))^{1/2}\cos (|\dot{\gamma}_{x,y}^{\widetilde{E}}|d(x,y)+s_0)\right|\leq C\delta^{7/800n}.
\end{equation*}
Thus, we get
\begin{equation*}
\begin{split}
&(f(x)^2+|\nabla f|^2(x))^{1/2}\\
\leq &f(y)+C\delta^{1/100n}\\
\leq &(f(x)^2+|\nabla f|^2(x))^{1/2} \cos (|\dot{\gamma}_{x,y}^{\widetilde{E}}|d(x,y)+s_0)+C\delta^{7/800n}\\
\leq &(f(x)^2+|\nabla f|^2(x))^{1/2} -\frac{1}{C}\delta^{3/800n}
\end{split}
\end{equation*}
by (\ref{54fb}), (\ref{54g}) and $|\nabla f|(x)\geq \delta^{1/800n}$.
This is a contradiction.
Therefore, we get $\langle\nabla f(x),\dot{\gamma}_{x,y}(0)\rangle>0$.
Thus,
\begin{equation}\label{54ha}
\begin{split}
\left|f(y)-(f(x)^2+|\nabla f|^2(x))^{1/2}\cos (|\dot{\gamma}_{x,y}^{\widetilde{E}}|d(x,y)-s_0)\right|&\leq C\delta^{7/800n},\\
\left| \langle \nabla f(y), \dot{\gamma}_{x,y}\rangle+|\dot{\gamma}_{x,y}^{\widetilde{E}}|(f(x)^2+|\nabla f|^2(x))^{1/2}\sin (|\dot{\gamma}_{x,y}^{\widetilde{E}}|d(x,y)-s_0)\right|&\leq C\delta^{7/800n}
\end{split}
\end{equation}
by (\ref{54h}) and Claim \ref{c1}.
Then, we have
\begin{align*}
(f(x)^2+|\nabla f|^2(x))^{1/2} (1-\cos(|\dot{\gamma}_{x,y}^{\widetilde{E}}|d(x,y)-s_0))
\leq C\delta^{7/800n}
\end{align*}
by (\ref{54fb}), and so
$$
1-\cos(|\dot{\gamma}_{x,y}^{\widetilde{E}}|d(x,y)-s_0)\leq C\delta^{3/400n}.
$$
by $|\nabla f|(x)\geq\delta^{1/800n}$.
Since
$-\pi<|\dot{\gamma}_{x,y}^{\widetilde{E}}|d(x,y)-s_0<\pi,$
we get
\begin{equation}\label{54i}
\left||\dot{\gamma}_{x,y}^{\widetilde{E}}|d(x,y)-s_0\right|\leq C\delta^{3/800n}.
\end{equation}
Thus, we have
$s_0\leq |\dot{\gamma}_{x,y}^{\widetilde{E}}|s_0+ C\delta^{3/800n}$
by (\ref{54fa}), and so
\begin{equation}\label{54j}
1-|\dot{\gamma}_{x,y}^{\widetilde{E}}| \leq C\delta^{1/400n}
\end{equation}
by (\ref{54ea}). Thus, we get
\begin{equation}\label{54k}
|d(x,y)-s_0|\leq C\delta^{1/400n}.
\end{equation}
By (\ref{54ha}) and (\ref{54i}), we have
\begin{equation}\label{54l}
|\langle\nabla f(y), \dot{\gamma}_{x,y}(d(x,y))\rangle|\leq C\delta^{3/800n}.
\end{equation}

We have
\begin{equation}\label{54m}
\begin{split}
&\frac{d}{d s}\left(|\nabla f|^2(s)-\langle\nabla f,\dot{\gamma}_{x,y}(s)\rangle^2\right)\\
=&2\left(\langle\nabla_{\dot{\gamma}_{x,y}}\nabla f,\nabla f\rangle(s)-\langle\nabla_{\dot{\gamma}_{x,y}}\nabla f,\dot{\gamma}_{x,y}(s)\rangle\langle\nabla f,\dot{\gamma}_{x,y}(s)\rangle\right)\\
=&2\langle \nabla^2 f+ f\sum_{i=1}^{n-p}\widetilde{E}^i\otimes \widetilde{E}^i,\dot{\gamma}_{x,y}\otimes\nabla f\rangle(s)-2f\langle\nabla f,\dot{\gamma}_{x,y}^{\widetilde{E}}\rangle\\
&-2\langle \nabla^2 f+ f\sum_{i=1}^{n-p}\widetilde{E}^i\otimes \widetilde{E}^i,\dot{\gamma}_{x,y}\otimes\dot{\gamma}_{x,y}
\rangle(s)\langle\nabla f,\dot{\gamma}_{x,y}(s)\rangle\\
&+2f|\dot{\gamma}_{x,y}^{\widetilde{E}}|^2\langle\nabla f,\dot{\gamma}_{x,y}(s)\rangle.
\end{split}
\end{equation}
Thus, we get
\begin{equation}\label{54n}
\begin{split}
&\left|\frac{d}{d s}\left(|\nabla f|^2(s)-\langle\nabla f,\dot{\gamma}_{x,y}(s)\rangle^2\right)\right|\\
\leq&C\left|\nabla^2 f+ f\sum_{i=1}^{n-p}\widetilde{E}^i\otimes \widetilde{E}^i\right|
+C\left|\langle\nabla f,\dot{\gamma}_{x,y}^{\widetilde{E}}\rangle-|\dot{\gamma}_{x,y}^{\widetilde{E}}|^2\langle\nabla f,\dot{\gamma}_{x,y}(s)\rangle\right|.\\
\leq &C\left|\nabla^2 f+ f\sum_{i=1}^{n-p}\widetilde{E}^i\otimes \widetilde{E}^i\right|+
C\delta^{1/400n}
\end{split}
\end{equation}
by (\ref{54ba}) and (\ref{54j}). By integration, we get
$$
\int_0^{d(x,y)} \left|\frac{d}{d s}\left(|\nabla f|^2(s)-\langle\nabla f,\dot{\gamma}_{x,y}(s)\rangle^2\right)\right|\,d s
\leq C\delta^{1/400n},
$$
and so
\begin{equation*}
\Big||\nabla f|^2(y)-\langle\nabla f(y),\dot{\gamma}_{x,y}(d(x,y))\rangle^2
-|\nabla f|^2(x)+\langle\nabla f(x),\dot{\gamma}_{x,y}(0)\rangle^2
\Big|\leq C\delta^{1/400n}.
\end{equation*}
Thus, we get
\begin{equation*}
|\nabla f|(y)\leq C\delta^{1/800n}.
\end{equation*}
by (\ref{54h0}), (\ref{54j}) and (\ref{54l}).
By Claim \ref{c0} and (\ref{54c0}), we get
$$
\left||f(y)|-f(p_f)\right|\leq C\delta^{1/800n}.
$$
Since
$$
f(y)\geq (f(x)^2+|\nabla f|^2(x))^{1/2}-C\delta^{1/100n}\geq \delta^{1/800n}-C\delta^{1/100n}>0
$$
by (\ref{54fb}),
we get (b).
We get 
\begin{equation}\label{54o}
|(f(x)^2+|\nabla f|^2(x))^{1/2}-f(p_f)|
\leq  C\delta^{1/800n}
\end{equation}
by (\ref{54ha}), (\ref{54i}) and (b),
and so we get (c) by the definition of $s_0$ and (\ref{54k}).
(\ref{54o}) implies the first assertion.

Finally, we show (d).
Suppose that a point $z\in M$ satisfies $d(x,z)\leq d(x,y)-\delta^{1/2000n}$.
Then, 
$d(x,y)\geq \delta^{1/2000n}$, and so
$$f(x)\leq f(p_f)\cos d(x,y)+C\delta^{1/800n}\leq f(p_f)-\frac{1}{C}\delta^{1/1000n}$$
by (\ref{54c0}).
There exists $w\in D_f(x)$ with
$d(z,w)\leq C\delta^{1/100n}$.
Let $\{\overline{E}^1,\ldots,\overline{E}^n\}$ be a parallel orthonormal basis along $\gamma_{x,w}$ in Lemma \ref{p5e}.
If (i) holds in the lemma, we assume that $\overline{E}_1=\dot{\gamma}_{x,w}$.
If $|\dot{\gamma}_{x,w}^{\overline{E}}|\leq \delta^{1/100}$, we have
\begin{equation*}
f(z)\leq f(w)+C\delta^{1/100n}
\leq f(x)+ C\delta^{1/100n}
\leq f(p_f)-\frac{1}{C}\delta^{1/1000n}
\end{equation*}
by Claim \ref{c1}.
If $|\dot{\gamma}_{x,w}^{\overline{E}}|\geq \delta^{1/100}$, we have
\begin{equation*}
\begin{split}
f(z)\leq& f(w)+C\delta^{1/100n}\\
\leq& f(x)\cos (|\dot{\gamma}_{x,w}^{\overline{E}}|d(x,z))+|\nabla f|(x)\sin (|\dot{\gamma}_{x,w}^{\overline{E}}|d(x,z))+C\delta^{1/100n}\\
\leq& f(p_f)\cos (|\dot{\gamma}_{x,w}^{\overline{E}}|d(x,z)-d(x,y))+\delta^{1/800n}
\leq f(p_f)-\frac{1}{C}\delta^{1/1000n}
\end{split}
\end{equation*}
by Claim \ref{c1}, (\ref{54k}), (\ref{54o}) and $-\pi\leq|\dot{\gamma}_{x,w}^{\overline{E}}|d(x,z)-d(x,y)\leq -\delta^{1/2000n}$.
For both cases, we get (d).
\end{proof}
By Claims \ref{c0} and \ref{c2},
we get
\begin{equation}\label{54p}
|f(x)^2+|\nabla f|^2(x)-f(p_f)^2|\leq C\delta^{1/800n}
\end{equation}
for all $x\in D_f(p_f)\cap Q_f\cap R_f$.
\begin{Clm}\label{c3}
We have
$
|f(p_f)-1|\leq C\delta^{1/800n}.
$
\end{Clm}
\begin{proof}[Proof of Claim \ref{c3}]
Since
$
\|f^2+|\nabla f|^2-f(p_f)^2\|_{\infty}\leq C
$
and
$
\Vol(M\setminus (D_f(p_f)\cap Q_f\cap R_f) )\leq C\delta^{1/100},
$
we get
$$
\frac{1}{\Vol(M)}\int_M|f(x)^2+|\nabla f|^2(x)-f(p_f)^2| \,d\mu_g\leq C \delta^{1/800n}
$$
by (\ref{54p}).
By the assumption,
we have
$$
\frac{1}{\Vol(M)}\left|\int_M (f(x)^2+|\nabla f|^2(x)-1) \,d\mu_g\right|\leq C \delta^{1/2}
$$
Thus, we get
$
|f(p_f)^2-1|\leq C\delta^{1/800n}.
$
Since $f(p_f)>0$, we get the claim.
\end{proof}
By Claims \ref{c0}, \ref{c3} and (\ref{54p}), we get (i), (ii) and (iii).

Finally, we prove (iv).
Put
$
A_f:=\{x\in M: |f(x)-1|\leq \delta^{1/900n}\}.
$
Since we have
$
|f(w)-\cos d(w,A_f)|\leq \delta^{1/900n}
$
for all $w\in A_f$, we get (iv) on $A_f$.

Let us show (iv) on $M\setminus A_f$.
Take $w\notin A_f$ and $x\in D_f(p_f)\cap Q_f\cap R_f $ with
$d(w,x)\leq C\delta^{1/100n}$.

We first suppose that $|\nabla f|(x)\geq \delta^{1/800n}$.
Take $y\in D_f(p_f)\cap D_f(x)$ of Claim \ref{c2}.
Then,
$|f(y)-1|\leq C\delta^{1/800n}$, and so $y\in A_f$.
Thus,
\begin{equation}\label{54q}
d(x, A_f)\leq d(x,y)<\pi.
\end{equation}
For all $z\in A_f$,
we have
$|f(p_f)-f(z)|\leq C\delta^{1/900n}$,
and so
$d(x,z)> d(x,y)-\delta^{1/2000n}$
by Claim \ref{c2} (d).
Thus,
\begin{equation}\label{54r}
d(x,A_f)\geq d(x,y)-\delta^{1/2000n}.
\end{equation}
By (\ref{54q}) and (\ref{54r}), we get
$
|d(x,A_f)- d(x,y)|\leq \delta^{1/2000n}.
$
Therefore, we have
$|f(x)-\cos d(x,A_f)|\leq C\delta^{1/2000n}$
by Claim \ref{c2} (c), and so
$|f(w)-\cos d(w,A_f)|\leq C\delta^{1/2000n}$.
By (\ref{54q}), we have $d(w,A_f)\leq \pi+C\delta^{1/100n}$.

We next suppose that $|\nabla f|(x)\leq \delta^{1/800n}$.
Then,
$||f|(x)-1|\leq C\delta^{1/800n}$
by Claim \ref{c0}.
If $f(x)\geq 0$, then $w \in A_f$.
This contradicts to $w\notin A_f$.
Thus, we have $|f(x)+1|\leq C\delta^{1/800n}$.
We see that (i) in Lemma \ref{p5e} cannot occur for $\gamma_{p_f,x}$
because we have
$$
|\nabla^2 f|\geq \frac{1}{\sqrt{n}}|\Delta f|\geq\frac{n-p}{\sqrt{n}}|f|-C\delta^{1/2}.
$$
Thus, there exists an orthonormal basis $\{e^1,\ldots,e^n\}$ of $T_x^\ast M$ such that 
$|\omega(x)-e^{n-p+1}\wedge\cdots\wedge e^n|\leq C\delta^{1/25}$ if Assumption \ref{aspform} holds,
and
$|\xi(x)-e^{1}\wedge\cdots\wedge e^{n-p}|\leq C\delta^{1/25}$ if Assumption \ref{asn-pform} holds.
Take $u\in E_f(x)$ with $|u-e_1|\leq C\delta^{1/100n}$.
Then, we get
$
|f\circ\gamma_u(s)+\cos s|\leq C\delta^{1/800n}
$
for all $s\in [0,\pi]$ by Lemma \ref{p5e}.
Thus, we get $\gamma_u(\pi)\in A_f$, and so 
\begin{equation}\label{54s}
d(w,A_f)\leq \pi+C\delta^{1/100n}.
\end{equation}
For any $y\in A_f$, there exists $z\in D_f(x)$ with $d(y,z)\leq C\delta^{1/100n}$.
Let $\{E^1,\ldots,E^n\}$ be a parallel orthonormal basis of $T^\ast M$ along $\gamma_{x,z}$ of Claim\ref{c1}.
Then,
\begin{equation*}
|1+\cos ( |\dot{\gamma}_{x,z}^{E}| d(x,z))|\leq C\delta^{1/900n}
\end{equation*}
by Claim \ref{c1}.
Thus, we get $d(x,z)\geq\pi-C\delta^{1/1800n}$,
and so
\begin{equation}\label{54t}
d(w,A_f)\geq \pi- C\delta^{1/1800n}.
\end{equation}
By (\ref{54s}) and (\ref{54t}), we get
$|d(w,A_f)-\pi|\leq C\delta^{1/1800n}$, and so
$|f(w)-\cos d(w,A_f)|\leq C\delta^{1/1800n}$.

For both cases, we get (iv).
\end{proof}
\subsection{Gromov-Hausdorff Approximation}
In this subsection, we suppose that Assumption \ref{asu1} for $k=n-p+1$ and either Assumption \ref{aspform} or \ref{asn-pform} hold. We construct a Gromov-Hausdorff approximation map, and show that the Riemannian manifold is close to the product metric space $S^{n-p}\times X$ in the Gromov-Hausdorff topology.
The following proposition is based on \cite[Lemma 5.2]{Pe1}.
\begin{Lem}\label{p54a}
Define
$\widetilde{\Psi}:=(f_1,\dots,f_{n-p+1})\colon M\to \mathbb{R}^{n-p+1}$.
Then, we have
$$
\||\widetilde{\Psi}|^2-1\|_{\infty}\leq C\delta^{1/1000n^2}.
$$
\end{Lem}
\begin{proof}
We first prove the following claim:
\begin{Clm}\label{p54b}
For any $x\in M$, we have $|\widetilde{\Psi}|(x)\leq 1+C\delta^{1/800n}$
\end{Clm}
\begin{proof}[Proof of Claim \ref{p54b}]
If $|\widetilde{\Psi}|(x)=0$, the claim is trivial.
Thus, we assume that $|\widetilde{\Psi}|(x)\neq 0$.
Put
$$f_x:=\frac{1}{|\widetilde{\Psi}|(x)}\sum_{i=1}^{n-p+1} f_i(x)f_i.$$
Then, we have
$\|f_x\|_2^2=1/(n-p+1).$
Thus, we get
$
|\widetilde{\Psi}|(x)=f_x(x)\leq 1+ C\delta^{1/800n}$
by Proposition \ref{p53a} (i).
\end{proof}
For $x\in M$ with $|\widetilde{\Psi}(x)|^2-1< 0$,
we have $||\widetilde{\Psi}(x)|^2-1|=1-|\widetilde{\Psi}(x)|^2$.
For $x\in M$ with $|\widetilde{\Psi}(x)|^2-1\geq 0$,
we have $||\widetilde{\Psi}(x)|^2-1|=|\widetilde{\Psi}(x)|^2-1 \leq 1-|\widetilde{\Psi}(x)|^2+C\delta^{1/800n}$ by Claim \ref{p54b}.
For both cases, we have $||\widetilde{\Psi}(x)|^2-1|\leq 1-|\widetilde{\Psi}(x)|^2+C\delta^{1/800n}$. Combining this and $\|\widetilde{\Psi}\|_2=1$, we get
$
\||\widetilde{\Psi}|^2-1\|_1 \leq C \delta^{1/800n}.$
Therefore, we have
$$
\Vol(\{x\in M:||\widetilde{\Psi}(x)|^2-1|\geq \delta^{1/1000n^2}\})\leq C\delta^{1/800n}\delta^{-1/1000n^2}\leq C\delta^{1/1000n}
$$
(note that we assumed $n\geq 5$).
This and the Bishop-Gromov inequality imply that,
for any $x\in M$, there exists $y\in\{x\in M:||\widetilde{\Psi}(x)|^2-1|< \delta^{1/1000n^2}\}$ with $d(x,y)\leq C\delta^{1/1000n^2}$, and so
$||\widetilde{\Psi}(x)|^2-1|\leq C\delta^{1/1000n^2}$ by $\|\nabla|\widetilde{\Psi}|^2\|_\infty\leq C$.
Thus, we get the lemma.
\end{proof}
\begin{notation} In the remaining part of this subsection, we use the following notation.
\begin{itemize}
\item Let $d_S$ denotes the intrinsic distance function on $S^{n-p}(1)$.
Note that we have $\cos d_S(x,y)=x\cdot y$
and 
$$d_{\mathbb{R}^{n-p+1}}(x,y)\leq d_{S}(x,y)\leq 3 d_{\mathbb{R}^{n-p+1}}(x,y)$$
for all $x,y\in S^{n-p}\subset\mathbb{R}^{n-p+1}$.
\item For each $f\in\Span_{\mathbb{R}}\{f_1,\ldots, f_{n-p+1}\}$ with $\|f\|_2^2=1/(n-p+1)$, we use the notation $p_f$ and $A_f$ of Proposition \ref{p53a}.
Recall that we defined
$
A_f:=\{x\in M: |f(x)-1|\leq \delta^{1/900n}\}.
$
\item Define $\widetilde{\Psi}:=(f_1,\dots,f_{n-p+1})\colon M\to \mathbb{R}^{n-p+1}$ and
$$
\Psi:=\frac{\widetilde{\Psi}}{|\widetilde{\Psi}|}\colon M\to S^{n-p}.
$$
\item For each $x\in M$, put
$$
f_x:=\frac{1}{|\widetilde{\Psi}|(x)}\sum_{i=1}^{n-p+1} f_i(x)f_i=\sum_{i=1}^{n-p+1} \Psi_i(x)f_i,
$$
$p_x:=p_{f_x}$ and $A_x:=A_{f_x}$.
\item For each $x\in M$ and $f\in\Span_{\mathbb{R}}\{f_1,\ldots, f_{n-p+1}\}$ with $\|f\|_2^2=1/(n-p+1)$, choose $a_f(x)\in A_f$ such that
$d(x,A_f)=d(x,a_f(x)).$
\end{itemize}
\end{notation}
The goal of this subsection is to show that
$$
\Phi_f \colon M\to S^{n-p}\times A_f,\,x\mapsto (\Psi(x),a_f(x))
$$
is a Gromov-Hausdorff approximation map.

\begin{Lem}\label{p54c00}
For all $x,y\in M$,
we have $|\Psi(x)-\Psi(y)|\leq Cd(x,y).$
\end{Lem}
\begin{proof}
Since we have $\|\nabla f_i\|_{\infty}\leq C$ for all $i\in\{1,\ldots,n-p+1\}$,
we get
$|\widetilde{\Psi}(x)-\widetilde{\Psi}(y)|\leq Cd(x,y)$
for all $x,y\in M$.
Thus, we get the lemma by Lemma \ref{p54a} ($|\widetilde{\Psi}|\geq1/2$).
\end{proof}

\begin{Lem}\label{p54c0}
Take $u\in S^{n-p}$ and put $f=\sum_{i=1}^{n-p+1}u_i f_i$.
Then, we have
$$
|d_S(\Psi(y),u)-d(y,A_{f})|\leq C\delta^{1/2000n^2}
$$
for all $y\in M$.
\end{Lem}
\begin{proof}
Since
$
f(y)=u\cdot\widetilde{\Psi}(y),
$
we have
$
|u \cdot\widetilde{\Psi}(y)-\cos d(y,A_{f})|\leq C\delta^{1/2000n}
$
by Proposition \ref{p53a},
and so
$$
|u\cdot \Psi(y)-\cos d(y,A_{f})|\leq C\delta^{1/1000n^2}
$$
by Lemma \ref{p54a}.
Since $\cos d_S(\Psi(y),u)=u\cdot \Psi(y)$, this and $d(y,A_{f})\leq \pi+C\delta^{1/100n}$ imply the lemma.
\end{proof}
By the definition of $A_{y}$, we immediately get the following corollaries:
\begin{Cor}\label{p54c01}
Take $u\in S^{n-p}$ and put $f=\sum_{i=1}^{n-p+1}u_i f_i$.
Then, we have
$
d_S(\Psi(p_f),u)\leq C\delta^{1/2000n^2}.
$
\end{Cor}
\begin{Cor}\label{p54c}
For each $y_1,y_2\in M$,
we have
$$
|d_S(\Psi(y_1),\Psi(y_2))-d(y_2,A_{y_1})|\leq C\delta^{1/2000n^2}.
$$
\end{Cor}
\begin{Cor}\label{p54c1}
For each $y\in M$,
we have
$
d(y,A_{y})\leq C\delta^{1/2000n^2}.
$
\end{Cor}
We need to show the almost Pythagorean theorem for our purpose.
To do this, we regard $|\dot{\gamma}^E| s$ in  Lemma \ref{p5e} as a moving distance in $S^{n-p}$.
We first approximate their cosine.
\begin{Lem}\label{p54d}
Take $y_1\in M$, $\tilde{y}_1\in D_{f_{y_1}}(p_{y_1})\cap R_{f_{y_1}}\cap Q_{f_{y_1}}$ with $d(y_1,\tilde{y}_1)\leq C\delta^{1/100n}$ and $y_2\in D_{f_{y_1}}(\tilde{y}_1)$ $($note that we can take such $\tilde{y}_1$ for any $y_1$ by the Bishop-Gromov theorem$)$.
Let $\{E^1,\ldots,E^n\}$ be a parallel orthonormal basis of $T^\ast M$ along $\gamma_{\tilde{y}_1,y_2}$ in Lemma \ref{p5e} for $f_{y_1}$.
Then, $(ii)$ holds in the lemma, and
$$
|\cos(|\dot{\gamma}_{\tilde{y}_1,y_2}^E|s)-\cos d_S(\Psi(y_1),\Psi(\gamma_{\tilde{y}_1,y_2}(s)))|\leq C\delta^{1/2000n^2}
$$
for all $s\in[0,d(\tilde{y}_1,y_2)]$.
In particular, we have
$$
|\cos(|\dot{\gamma}_{\tilde{y}_1,y_2}^E|d(\tilde{y}_1,y_2))-\cos d_S(\Psi(y_1),\Psi(y_2))|\leq C\delta^{1/2000n^2}.
$$
\end{Lem}
\begin{proof}
By Corollary \ref{p54c1}, we have
$
d(\tilde{y}_1,A_{y_1})\leq C\delta^{1/2000n^2},
$
and so
we get
\begin{equation*}
f_{y_1}\circ \gamma_{\tilde{y}_1,y_2}(s)
\geq \cos d(\gamma_{\tilde{y}_1,y_2}(s),A_{y_1})- C\delta^{1/2000n}
\geq \cos s- C\delta^{1/2000n^2}
\geq \frac{1}{\sqrt{2}}- C\delta^{1/2000n^2}
\end{equation*}
for all $s\leq\min\{\pi/4,d(\tilde{y}_1,y_2)\}$.
Therefore, we have
\begin{equation*}
|\nabla^2 f_{y_1}|(\gamma_{\tilde{y}_1,y_2}(s))
\geq \frac{1}{\sqrt{n}}|\Delta f_{y_1}|(\gamma_{\tilde{y}_1,y_2}(s))
\geq  \frac{n-p}{\sqrt{2n}}- C\delta^{1/2000n^2}
\end{equation*}
for all $s\leq\min\{\pi/4,d(\tilde{y}_1,y_2)\}$.
Thus, (i) in Lemma \ref{p5e} cannot occur, and so (ii) holds in the lemma.

Since we have $f_{y_1}(y_1)=|\widetilde{\Psi}(y_1)|$, we get
\begin{equation}\label{ad1}
|f_{y_1}(\tilde{y}_1)-1|\leq C\delta^{1/1000n^2}
\end{equation}
by Lemma \ref{p54a} and $d(y_1,\tilde{y}_1)\leq C\delta^{1/100n}$.
By (\ref{ad1}) and Proposition \ref{p53a} (iii),
we have
$
|\nabla f_{y_1}|(\tilde{y}_1)\leq C\delta^{1/2000n^2}.
$
Thus, we get
$$
|f_{y_1}(\gamma_{\tilde{y}_1,y_2}(s))-\cos(|\dot{\gamma}_{\tilde{y}_1,y_2}^E|s)|\leq C\delta^{1/2000n^2}
$$
for all $s\in[0,d(\tilde{y}_1,y_2)]$
by Lemma \ref{p5e}.
On the other hand, we have
$$
|f_{y_1}(\gamma_{\tilde{y}_1,y_2}(s))-\cos  d_S(\Psi(y_1),\Psi(\gamma_{\tilde{y}_1,y_2}(s)))|\leq C\delta^{1/2000n^2}
$$
for all $s\in[0,d(\tilde{y}_1,y_2)]$
by Proposition \ref{p53a} (iv) and Corollary \ref{p54c}.
Thus, we get the lemma.
\end{proof}
\begin{notation} We use the following notation:
\begin{itemize}
\item For any $y_1,y_2\in M$ and $f\in \Span_{\mathbb{R}}\{f_1,\ldots, f_{n-p+1}\}$ with $\|f\|_2^2=1/(n-p+1)$, define
\begin{align*}
&G_f^{y_1}(y_2)\\
:=&\langle\dot{\gamma}_{y_2,y_1}(0),\nabla f(y_2)\rangle d(y_1,y_2)\sin d_S(\Psi(y_1),\Psi(y_2))\\
&\quad +\Big(\cos d(y_2, A_f)\cos d_S(\Psi(y_1),\Psi(y_2))-\cos d(y_1,A_f)\Big)
d_S(\Psi(y_1),\Psi(y_2)).
\end{align*}
\item For any $y_1,y_2\in M$, define
\begin{empheq}[left={H^{y_1}(y_2):=\empheqlbrace}]{align*}
&1 \qquad d(y_1,y_2)\leq \pi,\\
&0 \qquad d(y_1,y_2)>\pi.
\end{empheq}
\item For any $y_1,y_2\in M$ and $f\in \Span_{\mathbb{R}}\{f_1,\ldots, f_{n-p+1}\}$ with $\|f\|_2^2=1/(n-p+1)$, define
\begin{align*}
C_f^{y_1}(y_2):=&\Big\{y_3\in M : \gamma_{y_2,y_3}(s)\in I_{y_1}\setminus\{y_1\} \text{ for almost all $s\in[0,d(y_2,y_3)]$, and}\\
&\qquad \qquad\qquad \qquad \int_{0}^{d(y_2,y_3)} |G_f^{y_1}H^{y_1}|(\gamma_{y_2,y_3}(s))\,d s\leq \delta^{1/12000n^2}\Big\},\\
P_f^{y_1}:=&\{y_2\in M: \Vol(M\setminus C_f^{y_1}(y_2))\leq\delta^{1/12000n^2}\Vol(M)\}.
\end{align*}
\end{itemize}
\end{notation}

Pinching condition on $G_f^{y_1}$ plays a crucial role for our purpose.
Let us estimate $G_f^{y_1}$.
\begin{Lem}\label{p54e}
Take $\eta>0$ with $\eta\geq \delta^{1/2000n}$, $f\in\Span_{\mathbb{R}}\{f_1,\ldots, f_{n-p+1}\}$ with $\|f\|_2^2=1/(n-p+1)$, $y_1\in Q_f$ and $y_2\in D_f(y_1)$.
Let $\{E^1,\ldots,E^n\}$ be a parallel orthonormal basis of $T^\ast M$ along $\gamma_{y_1,y_2}$ in Lemma \ref{p5e} for $f$.
If
$$
||\dot{\gamma}_{y_1,y_2}^E|d(y_1,y_2)-d_S(\Psi(y_1),\Psi(y_2))|\leq \eta,
$$
then
$
|G_f^{y_1}(y_2)|\leq C\eta.
$
\end{Lem}
\begin{proof}
We have
\begin{align*}
\Big|f(y_1)-f(y_2)\cos &(|\dot{\gamma}_{y_1,y_2}^E|d(y_1,y_2))\\
&-\frac{1}{|\dot{\gamma}_{y_1,y_2}^E|}\langle\nabla f(y_2),\dot{\gamma}_{y_2,y_1}(0)\rangle\sin (|\dot{\gamma}_{y_1,y_2}^E|d(y_1,y_2))
\Big|
\leq C\delta^{1/250}
\end{align*}
by Lemma \ref{p5e}.
Thus, by Proposition \ref{p53a} (iv), we get
\begin{align*}
\Big||\dot{\gamma}_{y_1,y_2}^E|\cos d(y_1,A_f)&-|\dot{\gamma}_{y_1,y_2}^E|\cos d(y_2, A_f)\cos (|\dot{\gamma}_{y_1,y_2}^E|d(y_1,y_2))\\
&-\langle\nabla f(y_2),\dot{\gamma}_{y_2,y_1}(0)\rangle\sin (|\dot{\gamma}_{y_1,y_2}^E|d(y_1,y_2))
\Big|
\leq C\delta^{1/2000n},
\end{align*}
and so we get the lemma.
\end{proof}

The quantity $|\dot{\gamma}_{y_1,y_2}^E|$ in the above lemma is slightly different from that of Lemma \ref{p54d}.
Comparing these two quantity, we get the following:
\begin{Cor}\label{p54f0}
Take $\eta>0$ with $\eta\geq \delta^{1/2000n}$, $f\in\Span_{\mathbb{R}}\{f_1,\ldots, f_{n-p+1}\}$ with $\|f\|_2^2=1/(n-p+1)$, $y_1\in M$, $\tilde{y}_1\in D_{f_{y_1}}(p_{y_1})\cap R_{f_{y_1}}\cap Q_{f_{y_1}}\cap Q_f$ with $d(y_1,\tilde{y}_1)\leq C\delta^{1/100n}$ and $y_2\in D_{f_{y_1}}(\tilde{y}_1)\cap D_f(\tilde{y}_1)$.
Let $\{E^1,\ldots,E^n\}$ be a parallel orthonormal basis of $T^\ast M$ along $\gamma_{\tilde{y}_1,y_2}$ in Lemma \ref{p5e} for $f_{y_1}$.
If
$$
||\dot{\gamma}_{\tilde{y}_1,y_2}^E|d(\tilde{y}_1,y_2)-d_S(\Psi(\tilde{y}_1),\Psi(y_2))|\leq \eta,
$$
then
$
|G_f^{\tilde{y}_1}(y_2)|\leq C\eta.
$
\end{Cor}
\begin{proof}
Let $\{\widetilde{E}^1,\ldots,\widetilde{E}^n\}$ be a parallel orthonormal basis of $T^\ast M$ along $\gamma_{\tilde{y}_1,y_2}$ in Lemma \ref{p5e} for $f$ (if (i) holds, then we can assume that $\widetilde{E}^i=E^i$ for all $i$).
We show that
$
\left||\dot{\gamma}_{\tilde{y}_1,y_2}^E|-|\dot{\gamma}_{\tilde{y}_1,y_2}^{\widetilde{E}}|\right|\leq C\delta^{1/50}.
$
Then, we immediately get the corollary by Lemma \ref{p54e}.

We first suppose that Assumption \ref{aspform} holds.
We have $|\omega(y_2)-E^{n-p+1}\wedge \cdots\wedge E^n|\leq C\delta^{1/25}$
by Lemmas \ref{p5e} and \ref{p54d}.
Since $|\dot{\gamma}_{\tilde{y}_1,y_2}^E|^2=1-|\iota(\dot{\gamma}_{\tilde{y}_1,y_2})(E^{n-p+1}\wedge \cdots\wedge E^n)|^2$, we get
\begin{equation}\label{55e}
\left||\dot{\gamma}_{\tilde{y}_1,y_2}^E|^2-\left(1-|\iota(\dot{\gamma}_{\tilde{y}_1,y_2})\omega|^2(y_2)\right)\right|\leq C\delta^{1/25}.
\end{equation}
Similarly, we get
\begin{equation}\label{55f}
\left||\dot{\gamma}_{\tilde{y}_1,y_2}^{\widetilde{E}}|^2-\left(1-|\iota(\dot{\gamma}_{\tilde{y}_1,y_2})\omega|^2(y_2)\right)\right|\leq C\delta^{1/25}.
\end{equation}
By (\ref{55e}) and (\ref{55f}),
we get
$
\left||\dot{\gamma}_{\tilde{y}_1,y_2}^E|-|\dot{\gamma}_{\tilde{y}_1,y_2}^{\widetilde{E}}|\right|\leq C\delta^{1/50}.
$

We next suppose that Assumption \ref{asn-pform} holds.
Similarly, we have
\begin{align*}
\left||\dot{\gamma}_{\tilde{y}_1,y_2}^E|^2-|\iota(\dot{\gamma}_{\tilde{y}_1,y_2})\xi|^2(y_2)\right|\leq& C\delta^{1/25},\\
\left||\dot{\gamma}_{\tilde{y}_1,y_2}^{\widetilde{E}}|^2-|\iota(\dot{\gamma}_{\tilde{y}_1,y_2})\xi|^2(y_2)\right|\leq& C\delta^{1/25},
\end{align*}
and so
$
\left||\dot{\gamma}_{\tilde{y}_1,y_2}^E|-|\dot{\gamma}_{\tilde{y}_1,y_2}^{\widetilde{E}}|\right|\leq C\delta^{1/50}.
$

By the above two cases, we get the corollary.
\end{proof}

Let us show the integral pinching condition.
\begin{Lem}\label{p54f}
Take $f\in\Span_{\mathbb{R}}\{f_1,\ldots, f_{n-p+1}\}$ with $\|f\|_2^2=1/(n-p+1)$, $y_1\in M$ and $\tilde{y}_1\in D_{f_{y_1}}(p_{y_1})\cap R_{f_{y_1}}\cap Q_{f_{y_1}}\cap Q_f$ with $d(y_1,\tilde{y}_1)\leq C\delta^{1/100n}$.
Then,
$\|G_f^{\tilde{y}_1} H_{\tilde{y}_1}\|_1\leq C\delta^{1/4000n^2}$
and
$
\Vol(M\setminus P_f^{\tilde{y}_1})\leq C\delta^{1/12000n^2}.
$
\end{Lem}
\begin{proof}
Take arbitrary $y_2\in D_f(\tilde{y}_1)\cap D_{f_{y_1}}(\tilde{y}_1)$.
Let $\{E^1,\ldots,E^n\}$ be a parallel orthonormal basis of $T^\ast M$ along $\gamma_{\tilde{y}_1,y_2}$ in Lemma \ref{p5e} for $f_{y_1}$.
Then, we have
$
||\dot{\gamma}_{\tilde{y}_1,y_2}^E|d(\tilde{y}_1,y_2)-d_S(\Psi(\tilde{y}_1),\Psi(y_2))|\leq C\delta^{1/4000n^2},
$
if $d(\tilde{y}_1,y_2)\leq \pi$ by Lemmas \ref{p54c00} and \ref{p54d}.
Thus, by Corollary \ref{p54f0}, we have
$$
\sup_{D_f(\tilde{y}_1)\cap D_{f_{y_1}}(\tilde{y}_1)}|G_f^{\tilde{y}_1} H^{\tilde{y}_1}|\leq C\delta^{1/4000n^2}.
$$
Since $\Vol(M\setminus (D_f(\tilde{y}_1)\cap D_{f_{y_1}}(\tilde{y}_1)))\leq C\delta^{1/100}\Vol(M)$ and $\|G_f^{\tilde{y}_1} H^{\tilde{y}_1}\|_\infty\leq C$,
we get $\|G_f^{\tilde{y}_1} H^{\tilde{y}_1}\|_1\leq C\delta^{1/4000n^2}.$
By the segment inequality (Theorem \ref{seg}), we get the remaining part of the lemma.
\end{proof}
\begin{notation}\label{order}
We use the following notation.
$$\eta_0=\delta^{1/12000n^3},\, \eta_1=\eta_0^{1/26}, \, \eta_2=\eta_1^{1/78}\text{ and } L=\eta_2^{1/150}.
$$
\end{notation}
We use Lemma \ref{p54f} to give the almost Pythagorean theorem for the special case (see Lemma \ref{p54l}).
For the general case, we need to estimate $\|G_f^{\tilde{y}_1}\|_1$.
To do this, we show that $|\dot{\gamma}_{\tilde{y}_1,y_2}^E|d(\tilde{y}_1,y_2)\leq \pi+L$ under the assumption of Lemma \ref{p54d} in Lemma \ref{p54n}. Then, we can estimate $\|G_f^{\tilde{y}_1}\|_1$ similarly to Lemma \ref{p54f}.
After proving that, we use Lemma \ref{p54i} again to give the almost Pythagorean theorem for the general case.
The following lemma, which guarantees that an almost shortest pass from a point in $M$ to $A_f$ almost corresponds to a geodesic in $S^{n-p}$ through $\Psi$ under some assumptions, is the first step to achieve these objectives.
\begin{Lem}\label{p54g}
Take \begin{itemize}
\item $f\in\Span_{\mathbb{R}}\{f_1,\ldots, f_{n-p+1}\}$ with $\|f\|_2^2=1/(n-p+1)$,
\item $u\in S^{n-p}$ with $f=\sum_{i=1}^{n-p+1}u_i f_i$,
\item $x,y\in M$,
\item $\eta>0$ with $\eta_0\leq\eta\leq L^{1/3n}$.
\end{itemize}
Suppose
\begin{itemize}
\item $d(y,A_f)\leq C \eta$,
\item $|d(x,A_f)-d(x,y)|\leq C\eta$.
\end{itemize}
Then, we have the following for all $s,s'\in[0,d(x,y)]$:
\begin{itemize}
\item[(i)] $|d(\gamma_{y,x}(s),A_f)-s|\leq C\eta$,
\item[(ii)] $\left||s-s'|-d_S\left(\Psi(\gamma_{y,x}(s)),\Psi(\gamma_{y,x}(s'))\right)\right|\leq C\eta$,
\item[(iii)] If in addition $d(x,A_f)\geq \frac{1}{C}\eta^{1/26}$, there exists $v\in S^{n-p}$ such that $u\cdot v=0$ and
$$
d_S(\Psi(\gamma_{y,x}(s)),\gamma_v(s))\leq C\eta^{3/13}
$$
for all $s\in[0,d(x,y)]$, where we define $\gamma_v(s):=(\cos s) u+(\sin s) v\in S^{n-p}$.
\end{itemize}
\end{Lem}
\begin{proof}
We first prove (i).
We have
$
d(\gamma_{y,x}(s),A_f)\leq s+ C\eta
$
and
\begin{align*}
d(x,y)-C\eta\leq d(x,A_f)\leq 
d(\gamma_{y,x}(s),A_f)+d(x,y)-s.
\end{align*}
Thus, we get (i).

We next prove (ii).
By Lemma \ref{p54c0}, we have $d_S(\Psi(y),u)\leq C\eta$ and
$|d_S(\Psi(\gamma_{y,x}(s)),u)-d(\Psi(\gamma_{y,x}(s)),A_f)|\leq C\delta^{1/2000n^2}$, and so we get
\begin{equation}\label{55i}
|s-d_S(\Psi(\gamma_{y,x}(s)),\Psi(y))|\leq C \eta
\end{equation}
for all $s\in[0,d(x,y)]$
by (i).
Take arbitrary $s,s'\in[0,d(x,y)]$ with $s<s'$.
Then,
\begin{equation}\label{55j}
\begin{split}
s'-s=d(\gamma_{y,x}(s),\gamma_{y,x}(s'))&\geq d(\gamma_{y,x}(s),A_{\gamma_{y,x}(s')})-d(\gamma_{y,x}(s'),A_{\gamma_{y,x}(s')})\\
&\geq d_S(\Psi(\gamma_{y,x}(s)),\Psi(\gamma_{y,x}(s')))-C\delta^{1/2000n^2}
\end{split}
\end{equation}
by Corollaries \ref{p54c} and \ref{p54c1}.
On the other hand,
we have
\begin{equation*}
\begin{split}
s'-C\eta\leq& d_S(\Psi(\gamma_{y,x}(s')),\Psi(y))\\
\leq &d_S(\Psi(\gamma_{y,x}(s)),\Psi(\gamma_{y,x}(s')))+d_S(\Psi(\gamma_{y,x}(s)),\Psi(y))\\
\leq &d_S(\Psi(\gamma_{y,x}(s)),\Psi(\gamma_{y,x}(s'))) +s+C\eta
\end{split}
\end{equation*}
by (\ref{55i}), and so
\begin{equation}\label{55k}
s'-s\leq d_S(\Psi(\gamma_{y,x}(s)),\Psi(\gamma_{y,x}(s'))) +C\eta.
\end{equation}
By (\ref{55j}) and (\ref{55k}), we get (ii).

Finally, we prove (iii).
Since $d(x,A_f)\geq\frac{1}{C}\eta^{1/26}$, there exists $s_0\in[0,d(x,y)]$ such that
$\frac{1}{C}\eta^{1/26}\leq d(z,y)\leq \pi- \frac{1}{C}\eta^{1/26}$, where we put $z=\gamma_{y,x}(s_0)$.
Then, there exists $v\in S^{n-p}$ with $u\cdot v=0$ and $t_1\in[0,\pi]$ such that
$
\Psi(z)=(\cos t_1) u+(\sin t_1) v.
$
We have
\begin{align*}
|\cos t_1-\cos d(z,y)|=&|\cos d_S(\Psi(z),u)-\cos s_0|\\
\leq& |\cos d(z,A_f)-\cos s_0|+C\delta^{1/2000n^2}
\leq C\eta
\end{align*}
by Lemma \ref{p54c0} and (i).
This gives
\begin{equation}\label{55l}
|t_1-d(z,y)|\leq C\eta^{1/2}.
\end{equation}
Take arbitrary $s\in [0,d(x,y)]$.
Then, there exist $w\in S^{n-p}$ and $x_1,x_2,x_3\in \mathbb{R}$ such that
$w\perp \Span_{\mathbb{R}}\{u,v\}$, $x_1^2+x_2^2+x_3^2=1$ and
$
\Psi(\gamma_{y,x}(s))=x_1 u+x_2 v+ x_3 w.
$
Since we have
$
|s-d_S(\Psi(\gamma_{y,x}(s)),u)|\leq C\eta
$
by (i) and Lemma \ref{p54c0},
and
$\cos d_S(\Psi(\gamma_{y,x}(s)),u)=x_1$,
we get
\begin{equation}\label{55m}
|\cos s- x_1|\leq C\eta.
\end{equation}
We have
$$
\left||d(z,y)-s|-d_S(\Psi(\gamma_{y,x}(s)),\Psi(z))\right|\leq C\eta
$$
by (ii).
Since
$\cos d_S(\Psi(\gamma_{y,x}(s)),\Psi(z))=x_1 \cos t_1+x_2\sin t_1$,
we get
\begin{equation}\label{55n}
|\cos(d(z,y)-s)- x_1 \cos d(z,y)-x_2\sin d(z,y)|\leq C\eta^{1/2}
\end{equation}
by (\ref{55l}).
By (\ref{55m}) and (\ref{55n}), we have
$
\sin d(z,y)|\sin s- x_2|\leq C\eta^{1/2}.
$
By the assumption, we have
$
\sin d(z,y)\geq \frac{1}{C}\eta^{1/26},
$
and so we get
\begin{equation}\label{55o}
|\sin s- x_2|\leq C\eta^{6/13}.
\end{equation}
By (\ref{55m}) and (\ref{55o}), 
we get
\begin{equation*}
|\cos d_S(\Psi(\gamma_{y,x}(s)),\gamma_v(s))-1|
=|x_1 \cos s+x_2\sin s-1|\leq C\eta^{6/13}.
\end{equation*}
Thus, we get (iii).
\end{proof}
The following lemma asserts that the differential of  an almost shortest pass from a point in $M$ to $A_f$ is in the direction of $\nabla f$ under some assumptions.
\begin{Lem}\label{p54h}
Take \begin{itemize}
\item $f\in\Span_{\mathbb{R}}\{f_1,\ldots, f_{n-p+1}\}$ with $\|f\|_2^2=1/(n-p+1)$,
\item $x\in D_f(p_f)\cap Q_f \cap R_f$,
\item $y\in D_f(x)\cap D_f(p_f)\cap Q_f\cap R_f$,
\item $\eta>0$ with $\eta_0\leq\eta\leq L^{1/3n}$.
\end{itemize}
Suppose
\begin{itemize}
\item $d(x,A_f)\geq\frac{1}{C}\eta^{1/26}$,
\item $d(y,A_f)\leq C \eta$,
\item $|d(x,A_f)-d(x,y)|\leq C\eta$.
\end{itemize}
Let $\{E^1,\ldots,E^n\}$ be a parallel orthonormal basis of $T^\ast M$ along $\gamma_{x,y}$ in Lemma \ref{p5e} for $f$.
Then, we have the following for all $s\in[0,d(x,y)]$:
\begin{itemize}
\item[(i)] $||\dot{\gamma}^E_{x,y}|-1|\leq C \eta^{6/13}$,
\item[(ii)] $|\nabla f (\gamma_{y,x}(s))+(\sin s) \dot{\gamma}_{y,x}(s)|\leq C\eta^{3/26}$.
\end{itemize}
\end{Lem}
\begin{proof}
We first note that we have
\begin{equation}\label{55p}
d(x,y)\leq \pi+C\eta
\end{equation}
by the assumption and Proposition \ref{p53a} (iv).

Let us prove (i).
By $d(y,A_f)\leq C \eta$, we have $\cos d(y,A_f)\geq 1- C\eta^2$.
Thus, we have
\begin{equation}\label{55q}
|1-f(y)|\leq C\eta^2
\end{equation}
by Proposition \ref{p53a} (iv).
By Proposition \ref{p53a} (iii), we get
$
|\nabla f|(y)\leq C\eta.
$
Thus, we have
\begin{equation}\label{55r}
|f(x)-\cos(|\dot{\gamma}_{x,y}^E|d(x,y))|\leq C\eta
\end{equation}
by Lemma \ref{p5e},
and so
$
||\dot{\gamma}_{x,y}^E|d(x,y)-d(x,A_f)|\leq C\eta^{1/2}
$
by Proposition \ref{p53a} (iv) and (\ref{55p}).
By the assumptions,
we get (i).

We next prove (ii).
By Proposition \ref{p53a}, we have
$
||\nabla f|^2(x)-\sin^2 d(x,A_f)|\leq C\delta^{1/2000n},
$
and so
$
||\nabla f|(x)-|\sin d(x,A_f)||\leq C\delta^{1/4000n}.
$
Since $\sin d(x, A_f)\geq -C\delta^{1/100n}$ by Proposition \ref{p53a} (iv), we have
$
||\nabla f|(x)-\sin d(x,A_f)|\leq C\delta^{1/4000n}.
$
Thus, we get
\begin{equation}\label{55s}
||\nabla f|(x)-\sin d(x,y)|\leq C\eta
\end{equation}
by the assumption.
On the other hand, by (i) and Lemma \ref{p5e}, we have
$
|f(y)-f(x)\cos d(x,y)-\langle\nabla f(x),\dot{\gamma}_{x,y}(0)\rangle\sin d(x,y)|\leq C\eta^{6/13},
$
and so
\begin{equation}\label{55t}
|\sin^2 d(x,y)-\langle\nabla f(x),\dot{\gamma}_{x,y}(0)\rangle\sin d(x,y)|\leq C\eta^{6/13}
\end{equation}
by (\ref{55q}) and (\ref{55r}).

We consider the following two cases:
\begin{itemize}
\item $d(x,y)\leq \pi-\eta^{3/13}$,
\item $d(x,y)> \pi-\eta^{3/13}$.
\end{itemize}

We first suppose that $d(x,y)\leq \pi-\eta^{3/13}$.
We get $
|\sin d(x,y)-\langle\nabla f(x),\dot{\gamma}_{x,y}(0)\rangle|\leq C\eta^{3/13}$
by the assumption and (\ref{55t}).
By (\ref{55s}), we get
\begin{equation}\label{55u}
|\nabla f|(x)-\langle\nabla f(x),\dot{\gamma}_{x,y}(0)\rangle \leq C\eta^{3/13}.
\end{equation}

We next suppose that $d(x,y)> \pi-\eta^{3/13}$.
Then, we have
$\cos d(x,A_f)\leq -1+C\eta^{6/13}$, and so
$|\nabla f|(x)\leq C\eta^{3/13}$
by Proposition \ref{p53a} (iii) and (iv).
Thus, we also get (\ref{55u}) for this case.

By (i), (\ref{54m}) and Lemma \ref{p5e}, we have
$$
\int_0^{d(x,y)} \left|\frac{d}{d s}\left(|\nabla f|^2(\gamma_{x,y}(s))-\langle\nabla f(\gamma_{x,y}(s)), \dot{\gamma}_{x,y}(s)\rangle^2\right)\right|\,d s\leq C\eta^{6/13}.
$$
Thus, we get
\begin{equation}\label{55v}
|\nabla f|^2(\gamma_{x,y}(s))-\langle\nabla f(\gamma_{x,y}(s)), \dot{\gamma}_{x,y}(s)\rangle^2\leq C\eta^{3/13}
\end{equation}
for all $s\in[0,d(x,y)]$ by (\ref{55u}).
Since
$$
|\nabla f (\gamma_{x,y}(s))-\langle\nabla f(\gamma_{x,y}(s)), \dot{\gamma}_{x,y}(s)\rangle\dot{\gamma}_{x,y}(s)|^2=|\nabla f|^2(\gamma_{x,y}(s))-\langle\nabla f(\gamma_{x,y}(s)), \dot{\gamma}_{x,y}(s)\rangle^2,
$$
we get
$
|\nabla f (\gamma_{x,y}(s))-\langle\nabla f(\gamma_{x,y}(s)), \dot{\gamma}_{x,y}(s)\rangle\dot{\gamma}_{x,y}(s)|\leq 
 C\eta^{3/26}
$
by (\ref{55v}).
Since we have
\begin{equation*}
|\langle\nabla f(\gamma_{x,y}(s)), \dot{\gamma}_{x,y}(s)\rangle+\cos d(x,y)\sin s-\sin d(x,y) \cos s|\leq C\eta^{3/13}
\end{equation*}
by (\ref{55r}), (\ref{55s}), (\ref{55u}), (i) and Lemma \ref{p5e},
we get
\begin{equation*}
|\nabla f (\gamma_{x,y}(s))-\sin (d(x,y)-s)\dot{\gamma}_{x,y}(s)|\leq 
 C\eta^{3/26}
\end{equation*}
This gives (ii).
\end{proof}
The following lemma is crucial to show the almost Pythagorean theorem.
\begin{Lem}\label{p54i}
Take \begin{itemize}
\item $f\in\Span_{\mathbb{R}}\{f_1,\ldots, f_{n-p+1}\}$ with $\|f\|_2^2=1/(n-p+1)$,
\item $x\in D_f(p_f)\cap Q_f \cap R_f$,
\item $y\in D_f(x)\cap D_f(p_f)\cap Q_f\cap R_f$,
\item $z\in M$,
\item $\eta>0$ with $\eta_0\leq\eta\leq L^{1/3n}$ and $T\in [0, d(x,y)]$.
\end{itemize}
Suppose
\begin{itemize}
\item $d(y,A_f)\leq C\eta$,
\item $|d(x,A_f)-d(x,y)|\leq C\eta$,
\item $\gamma_{y,x}(s)\in I_z\setminus\{z\}$ for almost all $s\in [T,d(x,y)]$,
\item $\int_T^{d(x,y)} |G_f^z(\gamma_{y,x}(s))|\,d s\leq C\eta^{3/26}$.
\end{itemize}
Then, we have
$$
\left| d(z,x)^2-d_S(\Psi(z),\Psi(x))^2- d(z,\gamma_{y,x}(T))^2+d_S(\Psi(z),\Psi(\gamma_{y,x}(T)))^2
\right|\leq C\eta^{1/26}.
$$
\end{Lem}
\begin{proof}
If $d(x,A_f)\leq\eta^{1/26}$, then $d(x,y)\leq C\eta^{1/26}$, and so
$d(x,\gamma_{y,x}(T))\leq C\eta^{1/26}$.
Thus, we immediately get the lemma by Lemma \ref{p54c00} if $d(x,A_f)\leq\eta^{1/26}$.
In the following, we assume that $d(x,A_f)\geq\eta^{1/26}$.
Take $u\in S^{n-p}$ with $f=\sum_{i=1}^{n-p+1}u_i f_i$, and $v\in S^{n-p}$ of Lemma \ref{p54g} (iii).
Define
$$
r(s):=d_S(\Psi (z),\gamma_v(s)).
$$
Then, by the triangle inequality and Lemma \ref{p54g} (iii), we have
\begin{equation}\label{55w}
|r(s)-d_S(\Psi (z),\Psi(\gamma_{y,x}(s)))|\leq C\eta^{3/13}.
\end{equation}

There exist $w\in S^{n-p}$ and $x_1,x_2,x_3\in \mathbb{R}$ such that
$w\perp \Span_{\mathbb{R}}\{u,v\}$, $x_1^2+x_2^2+x_3^2=1$ and
$
\Psi(z)=x_1 u+x_2 v+ x_3 w.
$
Then,
\begin{equation}\label{55x}
\cos r(s)=x_1\cos s+x_2\sin s
\end{equation}
by the definition of $\gamma_v$ in Lemma \ref{p54g} (iii),
and so
\begin{equation*}
-x_1\sin s+x_2\cos s
=\frac{d}{d s} \cos r(s)
=-r'(s)\sin r(s).
\end{equation*}
Thus, we get
\begin{equation}\label{55y}
\begin{split}
-r'(s)\sin r(s) \sin s=-x_1\sin^2 s+x_2\sin s\cos s=\cos r(s)\cos s-x_1
\end{split}
\end{equation}
by (\ref{55x}).
Since $x_1=\Psi(z)\cdot u$ and $f(z)=\widetilde{\Psi}(z)\cdot u$, we have
\begin{equation}\label{55z}
|x_1-\cos d(z,A_f)|\leq C\delta^{1/1000n^2}
\end{equation}
by Proposition \ref{p53a} (iv) and Lemma \ref{p54a}.
By Lemma \ref{p54g}, (\ref{55w}), (\ref{55y}) and (\ref{55z}), we get
\begin{equation}\label{56a}
\begin{split}
&\Big|\Big(\cos d(\gamma_{y,x}(s),A_f)\cos d_S(\Psi(z),\Psi(\gamma_{y,x}(s)))-\cos d(z,A_f)\Big)d_S(\Psi(z),\Psi(\gamma_{y,x}(s)))\\
&\qquad\qquad\qquad\qquad\qquad\qquad\qquad  +r'(s)r(s)\sin r(s) \sin s\Big|\leq C\eta^{3/13}.
\end{split}
\end{equation}

Define
$$
l(s):=d(z,\gamma_{y,x}(s)).
$$
Then, we have $
l'(s)=\langle\dot{\gamma}_{z,\gamma_{y,x}(s)}(l(s)),\dot{\gamma}_{y,x}(s)\rangle
$
for all $s\in [0,d(x,y)]$ with $\gamma_{y,x}(s)\in I_z\setminus\{z\}$, and so
$
|l'(s)\sin s+\langle\dot{\gamma}_{z,\gamma_{y,x}(s)}(l(s)),\nabla f(\gamma_{y,x}(s))\rangle|\leq C\eta^{3/26}
$
by Lemma \ref{p54h} (ii).
Thus, for almost all $s\in [T,d(x,y)]$, we have
\begin{equation}\label{56b}
\begin{split}
\Big|\langle\dot{\gamma}_{\gamma_{y,x}(s),z}(0),&\nabla f(\gamma_{y,x}(s))\rangle l(s)\sin d_S(\Psi(z),\Psi(\gamma_{y,x}(s)))\\
&-l'(s)l(s)\sin r(s)\sin s \Big|\leq C\eta^{3/26}
\end{split}
\end{equation}
by (\ref{55w}).
By the definition of $G_f^z$, (\ref{56a}) and (\ref{56b}), for almost all $s\in [T,d(x,y)]$, we have
\begin{align*}
\Big|
G_f^z(\gamma_{y,x}(s))-l'(s)l(s)\sin r(s)\sin s+r'(s)r(s)\sin r(s) \sin s
\Big|\leq C\eta^{3/26}.
\end{align*}
Thus, by the assumption, we get
\begin{equation}\label{56c}
\int_T^{d(x,y)}\left|\left(\frac{d}{d s}(l(s)^2-r(s)^2)\right)\sin r(s)\sin s\right|\,d s
\leq C\eta^{3/26}.
\end{equation}
Define
\begin{align*}
I&:=\{s\in [T,d(x,y)]: \eta^{1/26}\leq s\leq \pi -\eta^{1/26}\text{ and }\eta^{1/26}\leq r(s) \leq\pi -\eta^{1/26}
\}\\
II&:=[T,d(x,y)]\setminus I.
\end{align*}
Then, we have
\begin{equation}\label{56ca}
\int_I \left|\frac{d}{d s}(l(s)^2-r(s)^2)\right|\,d s
\leq C\eta^{1/26}
\end{equation}
by (\ref{56c}).
Let us estimate $H^1(II)$, where $H^1$ denotes the $1$-dimensional Hausdorff measure.
Suppose that
$$
\{s\in [T,d(x,y)]: r(s)<\eta^{1/26} \text{ or } r(s)>\pi-\eta^{1/26}\}\neq \emptyset,
$$
and take arbitrary $s\in[T,d(x,y)]$ such that
$r(s)<\eta^{1/26}$ or $r(s)>\pi-\eta^{1/26}$.
Then, we have
\begin{equation}\label{56d}
||\cos r(s)|-1|\leq C\eta^{1/13}.
\end{equation}
Note that we have $r(s)\leq \pi$ by $\diam (S^{n-p})=\pi$.
By (\ref{55x}), we get
\begin{equation}\label{56e}
1-C\eta^{1/13}\leq (x_1^2+x_2^2)^{1/2}\leq 1.
\end{equation}
Take $s_1\in[0,2\pi]$ such that
\begin{align*}
\cos s_1=&\frac{x_1}{(x_1^2+x_2^2)^{1/2}},\\
\sin s_1=&\frac{x_2}{(x_1^2+x_2^2)^{1/2}}.
\end{align*}
Then, we get
$||\cos (s-s_1)|-1|\leq C\eta^{1/13}$ by (\ref{55x}), (\ref{56d}) and (\ref{56e}).
Thus, there exists $n\in \mathbb{Z}$ such that
$
|s-s_1-n\pi|\leq C\eta^{1/26}.
$
Then, we have $|n|\leq 2$, and so
$$
H^1\left(\{s\in [T,d(x,y)]: r(s)<\eta^{1/26} \text{ or } r(s)>\pi-\eta^{1/26}\}\right)\leq C\eta^{1/26}.
$$
Note that we have $d(x,y)\leq d(x,A_f)+C\eta\leq \pi+C\eta$ by the assumption and Proposition \ref{p53a} (iv).
Since we have
$$
H^1\left(\{s\in [T,d(x,y)]: s<\eta^{1/26} \text{ or } s>\pi-\eta^{1/26}\}\right)\leq C\eta^{1/26},
$$
we get
$H^1(II)\leq C\eta^{1/26}$.
Since $\left|\frac{d}{d s}(l(s)^2-r(s)^2)\right|\leq C$ for almost all $s\in[T,d(x,y)]$,
we get
\begin{equation}\label{56f}
\int_{II} \left|\frac{d}{d s}(l(s)^2-r(s)^2)\right|\,d s
\leq C\eta^{1/26}.
\end{equation}

By (\ref{56ca}) and (\ref{56f}), we get
\begin{equation*}
\int_T^{d(x,y)}\left|\frac{d}{d s}(l(s)^2-r(s)^2)\right|\,d s
\leq C\eta^{1/26}.
\end{equation*}
Thus, we have
$
|l(d(x,y))^2-r(d(x,y))^2-l(T)^2+r(T)^2|\leq C\eta^{1/26}.
$
By (\ref{55w}) and the definition of $l$, we get the lemma.
\end{proof}

\begin{Def}
Take $f\in\Span_{\mathbb{R}}\{f_1,\ldots, f_{n-p+1}\}$ with $\|f\|_2^2=1/(n-p+1)$.
By Lemma \ref{p54f} and the Bishop-Gromov inequality, for any triple $(x_1,x_2,x_3)\in M\times M\times M$, we can take points $\tilde{x}_1\in D_{f_{x_1}}(p_{x_1})\cap Q_{f_{x_1}} \cap R_{f_{x_1}}\cap Q_f$, $\tilde{x}_2\in D_f(p_f)\cap Q_f \cap R_f\cap P_f^{\tilde{x}_1}$ and $\tilde{x}_3\in D_f(\tilde{x}_2)\cap D_f(p_f)\cap Q_f\cap R_f\cap C_f^{\tilde{x}_1}(\tilde{x}_2)$ such that $d(x_1,\tilde{x}_1)\leq C\delta^{1/100n}$,
$d(x_2,\tilde{x}_2)\leq C\eta_0$, $d(x_3,\tilde{x}_3)\leq C\eta_0$.
We call the triple $(\tilde{x}_1,\tilde{x}_2,\tilde{x}_3)$ a ``{\it $\Pi$-triple for $(x_1,x_2,x_3,f)$}''.
\end{Def}
\begin{Lem}\label{ptrp}
Take 
\begin{itemize}
\item $f\in\Span_{\mathbb{R}}\{f_1,\ldots, f_{n-p+1}\}$ with $\|f\|_2^2=1/(n-p+1)$,
\item $x,y,z\in M$,
\item $\eta>0$ with $\eta_0\leq\eta\leq L^{1/3n}$ and $T\in [0, d(x,y)]$.
\end{itemize}
Take a $\Pi$-triple $(\tilde{z},\tilde{x},\tilde{y})$ for $(z,x,y,f)$.
Suppose
\begin{itemize}
\item $d(y,A_f)\leq C\eta$,
\item $|d(x,A_f)-d(x,y)|\leq C\eta$,
\item $d(\tilde{z},\gamma_{\tilde{y},\tilde{x}}(s))\leq \pi$ for all $s\in[T,d(\tilde{x},\tilde{y})]$.
\end{itemize}
Then, we have
$$
\left| d(\tilde{z},\tilde{x})^2-d_S(\Psi(\tilde{z}),\Psi(\tilde{x}))^2- d(\tilde{z},\gamma_{\tilde{y},\tilde{x}}(T))^2+d_S(\Psi(\tilde{z}),\Psi(\gamma_{\tilde{y},\tilde{x}}(T)))^2
\right|\leq C\eta^{1/26}.
$$
\end{Lem}
\begin{proof}
We have $(G^{\tilde{z}}_f H^{\tilde{z}})(\gamma_{\tilde{y},\tilde{x}}(s))=G^{\tilde{z}}_f(\gamma_{\tilde{y},\tilde{x}}(s))$ for all $s\in[T,d(\tilde{x},\tilde{y})]$.
Thus, we get the lemma immediately by the definition of $C_f^{\tilde{z}}(\tilde{x})$ and Lemma \ref{p54i}.
\end{proof}
The following lemma guarantees that if the images of two points in $M$ under $\Phi_f$ are close to each other in $S^{n-p}\times A_f$, then their distance in $M$ are close to each other under some assumptions.
\begin{Lem}\label{p54j}
Take \begin{itemize}
\item $f\in\Span_{\mathbb{R}}\{f_1,\ldots, f_{n-p+1}\}$ with $\|f\|_2^2=1/(n-p+1)$,
\item $x,y,z\in M$,
\item $\eta>0$ with $\eta_0\leq\eta\leq L^{1/3n}$.
\end{itemize}
Suppose
\begin{itemize}
\item $d(x,A_f)\leq \pi- \frac{1}{C}\eta^{1/78}$ and $d(z,A_f)\leq \pi- \frac{1}{C}\eta^{1/78}$, 
\item $d(y,A_f)\leq C\eta$,
\item $|d(x,A_f)-d(x,y)|\leq C\eta$ and $|d(z,A_f)-d(z,y)|\leq C\eta$
\item $d_S(\Psi(x),\Psi(z))\leq C\eta$.
\end{itemize}
Then, we have
$
d(x,z)\leq C\eta^{1/52}.
$
\end{Lem}
\begin{proof}
We first show the following claim.
\begin{Clm}\label{p54k}
If $x,y,z\in M$ satisfies:
\begin{itemize}
\item $d(x,A_f)\leq \frac{1}{2}\pi- \frac{1}{C}\eta^{1/2}$ and $d(z,A_f)\leq \frac{1}{2}\pi- \frac{1}{C}\eta^{1/2}$,
\item $d(y,A_f)\leq C\eta$,
\item $|d(x,A_f)-d(x,y)|\leq C\eta$ and $|d(z,A_f)-d(z,y)|\leq C\eta$,
\item $d_S(\Psi(x),\Psi(z))\leq C\eta^{1/52}$.
\end{itemize}
Then, we have
$
d(x,z)\leq C\eta^{1/52}.
$
\end{Clm}
\begin{proof}[Proof of Claim \ref{p54k}]
Take $u\in S^{n-p}$ with $f=\sum_{i=1}^{n-p+1} u_i f_i$.
By the assumptions and Lemma \ref{p54c0}, we have
\begin{align*}
d_S(u,\Psi(y))\leq& C\eta,\\
|d_S(\Psi(z),u)-d(z,A_f)|\leq &C\delta^{1/2000n^2}.
\end{align*}
Since we have
$|d(z,A_f)-d(z,y)|\leq C\eta$ by the assumptions,
we get
\begin{equation}\label{57a0}
|d_S(\Psi(z),\Psi(y))-d(z,y)|\leq C\eta.
\end{equation}
Take a $\Pi$-triple $(\tilde{z},\tilde{x},\tilde{y})$ for $(z,x,y,f)$.
Then, we have
\begin{align*}
d(\tilde{z},\gamma_{\tilde{y},\tilde{x}}(s))
\leq d(z,y)+d(y,x)+C\eta_0
\leq \pi-\frac{1}{C}\eta^{1/2}+C\eta\leq \pi
\end{align*}
for all $s\in[0,d(\tilde{x},\tilde{y})]$, and so
$$
\left| d(z,x)^2-d_S(\Psi(z),\Psi(x))^2- d(z,y)^2+d_S(\Psi(z),\Psi(y))^2
\right|\leq C\eta^{1/26}
$$
by Lemmas \ref{p54c00} and \ref{ptrp}.
Thus, we get
$d(x,z)\leq C\eta^{1/52}$ by (\ref{57a0}).
\end{proof}

Let us suppose that $x,y,z\in M$ satisfies the assumptions of the lemma.
Take $u\in S^{n-p}$ with
$f=\sum_{i=1}^{n-p+1}u_i f_i$.
By the assumptions and Lemma \ref{p54c0},
we have
\begin{equation}\label{57a1}
|d(x,A_f)-d(z,A_f)|
\leq |d_S(\Psi(x),u)-d(\Psi(z),u)|+C\delta^{1/2000n^2}
\leq C\eta
\end{equation}
Thus, if either $d(x,A_f)\leq \eta^{1/26}$ or $d(z,A_f)\leq \eta^{1/26}$ holds, then the lemma is trivial.
In the following, we assume $d(x,A_f)\geq \eta^{1/26}$ and $d(z,A_f)\geq \eta^{1/26}$.
Take a $\Pi$-triple $(\tilde{z},\tilde{x},\tilde{y})$ for $(z,x,y,f)$.
By Lemma \ref{p54g} (iii), we can take $v_1,v_2\in S^{n-p}$ such that $u\cdot v_i=0$ ($i=1,2$),
\begin{equation}\label{57a}
d_S(\Psi(\gamma_{\tilde{y},\tilde{x}}(s)),\gamma_{v_1}(s))\leq C\eta^{3/13}
\end{equation}
for all $s\in [0,d(\tilde{y},\tilde{x})]$
and
\begin{equation}\label{57b}
d_S(\Psi(\gamma_{\tilde{y},\tilde{z}}(s)),\gamma_{v_2}(s))\leq C\eta^{3/13}
\end{equation}
for all $s\in [0,d(\tilde{y},\tilde{z})]$,
where $\gamma_{v_i}(s):=(\cos s) u+(\sin s) v_i \in S^{n-p}$ ($i=1,2$).
By the assumptions and (\ref{57a1}), we get 
\begin{equation}\label{57b1}
|d(\tilde{y},\tilde{x})-d(\tilde{y},\tilde{z})|\leq C\eta,
\end{equation}
and so
\begin{align*}
\sin d(\tilde{y},\tilde{x}) |v_1- v_2|
\leq &C d_S(\gamma_{v_1}(d(\tilde{y},\tilde{x})),\gamma_{v_2}(d(\tilde{y},\tilde{x})))\\
\leq &Cd_S(\Psi(\tilde{x}),\Psi(\tilde{z}))+C\eta^{3/13}
\leq C\eta^{3/13}
\end{align*}
by (\ref{57a}) and (\ref{57b}).
By $\eta^{1/26}\leq d(x,A_f)\leq \pi-\frac{1}{C}\eta^{1/78}$,
we have
$\sin d(\tilde{y},\tilde{x})\geq \frac{1}{C}\eta^{1/26}$.
Thus, we get
$
|v_1-v_2|\leq C\eta^{1/26}.
$
This gives
\begin{equation}\label{57c}
d_S(\gamma_{v_1}(s),\gamma_{v_2}(s))\leq C\eta^{1/26}.
\end{equation}
for all $s\in \mathbb{R}$.

Put
$
a:=\gamma_{\tilde{y},\tilde{x}}\left(d(\tilde{y},\tilde{x})/2\right)$ and $ b:=\gamma_{\tilde{y},\tilde{z}}\left(d(\tilde{y},\tilde{z})/2\right).$
By (\ref{57a}), (\ref{57b}), (\ref{57b1}) and (\ref{57c}), we have
$
d_S(\Psi(a),\Psi(b))\leq C\eta^{1/26}.
$
Moreover, other assumptions of Claim \ref{p54k} hold for the pair $(a,y,b)$ by Lemma \ref {p54g} (i), and so
$
d(a,b)\leq C\eta^{1/52}.
$
Therefore, we have
$$
d(\tilde{z}, \gamma_{\tilde{y},\tilde{x}}(s))\leq d(\tilde{z},b)+d(a,b)+d(\gamma_{\tilde{y},\tilde{x}}(s),a)\leq \frac{1}{2}d(\tilde{x},\tilde{y})+\frac{1}{2}d(\tilde{z},\tilde{y})+C\eta^{1/52}\leq \pi
$$
for all $s\in[0,d(\tilde{y},\tilde{x})]$,
and so $d(\tilde{x},\tilde{z})\leq C\eta^{1/52}$ similarly to Claim \ref{p54k}.
Thus, we get the lemma.
\end{proof}

Let us show the almost Pythagorean theorem for the special case.
Recall that we defined $\eta_1:=\eta_0^{1/26}$.
\begin{Lem}\label{p54l}
Take \begin{itemize}
\item $f\in\Span_{\mathbb{R}}\{f_1,\ldots, f_{n-p+1}\}$ with $\|f\|_2^2=1/(n-p+1)$,
\item $x,y,z,w\in M$,
\item $\eta>0$ with $\eta_1\leq \eta\leq L^{1/3n}$.
\end{itemize}
Suppose
\begin{itemize}
\item $d(x,z)\leq C\eta$,
\item $d(x,A_f)\leq \pi- \frac{1}{C}\eta^{1/2}$ and $d(z,A_f)\leq \pi- \frac{1}{C}\eta^{1/2}$, 
\item $d(y,A_f)\leq C\eta_0$ and $d(w,A_f)\leq C\eta_0$,
\item $|d(x,A_f)-d(x,y)|\leq C\eta_0$ and $|d(z,A_f)-d(z,w)|\leq C\eta_0$.
\end{itemize}
Then, we have
$$
|d(x,z)^2-d_S(\Psi(x),\Psi(z))^2-d(y,w)^2|\leq C\eta_1.
$$
\end{Lem}
\begin{proof}
By Lemma \ref{p54c0},
we have
\begin{equation}\label{57d0}
d_S(\Psi(y),\Psi(w))
\leq d(y,A_f)+d(w,A_f)+C\delta^{1/2000n^2}
\leq C\eta_0.
\end{equation}

Put
$a_0:=x$ and $b_0:=z$.
In the following,
we define $a_{i},b_{i}\in M$ ($i=1,2,3$) so that
\begin{itemize}
\item[(i)] $d(a_{i},b_{i})\leq C\eta^{1/2}$,
\item[(ii)] $|d(a_{i},A_f)-d(a_{i},y)|\leq C\eta_0$ and $|d(b_{i},A_f)-d(b_{i},w)|\leq C\eta_0$,
\item[(iii)] $d(a_{i},A_f)\leq \frac{3-i}{3}\pi+C\eta_0$ and $d(b_{i},A_f)\leq \frac{3-i}{3}\pi+C\eta_0$,
\item[(iv)] $|d(a_{i+1},b_{i+1})^2-d_S(\Psi(a_{i+1}),\Psi(b_{i+1}))^2-d(a_{i},b_{i})^2+d_S(\Psi(a_{i}),\Psi(b_{i}))^2|\leq C\eta_0^{1/26}$ ($i=0,1,2$),
\item[(v)] $d(y,a_3)\leq C\eta_0$ and $d(w,b_3)\leq C\eta_0$.
\end{itemize}
If we succeed in defining such $a_i$ and $b_i$, we have 
$$
|d(x,z)^2-d_S(\Psi(x),\Psi(z))^2-d(y,w)^2+d_S(\Psi(y),\Psi(w))^2|\leq C\eta_0^{1/26}=C\eta_1
$$
by (iv) and (v), and so we get the lemma by (\ref{57d0}).

Take arbitrary $i\in\{0,1,2\}$ and suppose that we have chosen $a_i,b_i\in M$ such that (i), (ii) and (iii) hold if $i\geq 1$.
Let us define $a_{i+1},b_{i+1}\in M$ that satisfy our properties.
Take a $\Pi$-triple $(\tilde{b}_i,\tilde{a}_i, \tilde{y}_i)$ for $(b_i,a_i,y,f)$. 
Define
$$
a_{i+1}:=\gamma_{\tilde{y}_i,\tilde{a}_i}\left(\frac{2-i}{3-i}d(\tilde{y}_i,\tilde{a}_i)\right).
$$
Since $$d(\tilde{b}_i, \gamma_{\tilde{y}_i,\tilde{a}_i}(s))\leq d(\tilde{a}_i,\tilde{b}_i)
+d(\tilde{a}_i,\gamma_{\tilde{y}_i,\tilde{a}_i}(s))\leq \frac{\pi}{3}+C\eta^{1/2}$$ for all $s\in\left[\frac{2-i}{3-i}d(\tilde{y}_i,\tilde{a}_i),d(\tilde{y}_i,\tilde{a}_i)\right]$ by the assumptions, we get
\begin{equation}\label{57d}
|d(a_{i+1},b_{i})^2-d_S(\Psi(a_{i+1}),\Psi(b_{i}))^2-d(a_{i},b_{i})^2+d_S(\Psi(a_{i}),\Psi(b_{i}))^2|\leq C\eta_0^{1/26}
\end{equation}
by Lemmas \ref{p54c00} and \ref{ptrp}.
Take a $\Pi$-triple $(\overline{a}_{i+1},\overline{b}_i,\overline{w}_i)$ for $(a_{i+1},b_i,w,f)$.
Define
$$
b_{i+1}:=\gamma_{\overline{w}_i,\overline{b}_i}\left(\frac{2-i}{3-i}d(\overline{w}_i,\overline{b}_i)\right).
$$
Since
$$
d(\overline{a}_{i+1},\gamma_{\overline{w}_i,\overline{b}_i}(s))\leq d(\overline{a}_{i+1},a_i)+d(a_i,\overline{b}_i)+d(\overline{b_i},\gamma_{\overline{w}_i,\overline{b}_i}(s))\leq \frac{2}{3}\pi +C\eta^{1/2}
$$
for all $s\in\left[\frac{2-i}{3-i}d(\overline{w}_i,\overline{b}_i),d(\overline{w}_i,\overline{b}_i)\right]$ by the assumptions,
we get
\begin{equation}\label{57e}
|d(a_{i+1},b_{i+1})^2-d_S(\Psi(a_{i+1}),\Psi(b_{i+1}))^2-d(a_{i+1},b_{i})^2+d_S(\Psi(a_{i+1}),\Psi(b_{i}))^2|\leq C\eta_0^{1/26}
\end{equation}
by Lemmas \ref{p54c00} and \ref{ptrp}.
By (\ref{57d}) and (\ref{57e}), we get (iv).

By the assumptions and Lemma \ref{p54g}, we get (ii) for $a_{i+1}$ and $b_{i+1}$.

By the assumptions, we have
\begin{align*}
d(a_{i+1},A_f)
\leq &d(a_{i+1},\tilde{y}_i)+d(y,A_f)+C\eta_0\\
=&\frac{2-i}{3-i}d(\tilde{a}_{i},\tilde{y}_i)+C\eta_0
\leq \frac{2-i}{3}\pi+C\eta_0.
\end{align*}
Similarly, we have
$d(b_{i+1},A_f)\leq \frac{2-i}{3}\pi+C\eta_0$.
Thus, we get (iii) for $a_{i+1}$ and $b_{i+1}$.

By definition, we have
$
a_3=\tilde{y}_3$ and $b_3=\overline{w}_3.
$
Thus, we get (v).

In the following, we prove (i) for $a_{i+1}$ and $b_{i+1}$.
If $d(a_i,y)\leq \eta_0^{1/26}$,
then we have
\begin{align*}
d(b_i,w)\leq d(b_i,A_f)+C\eta_0
\leq d(a_i,A_f)+C\eta^{1/2}
\leq C\eta^{1/2},
\end{align*}
and so
$
d(y,w)\leq C\eta^{1/2}$, $d(a_{i+1},y)\leq C\eta^{1/2}$ and $d(b_{i+1},w)\leq C\eta^{1/2}.
$
Then, we have $d(a_{i+1},b_{i+1})\leq C\eta^{1/2}$.
Similarly, if $d(b_i,w)\leq \eta_0^{1/26}$, then $d(a_{i+1},b_{i+1})\leq C\eta^{1/2}$.
Thus, in the following, we assume that $d(a_i,y)\geq \eta_0^{1/26}$ and $d(b_i,w)\geq \eta_0^{1/26}$.
By Lemma \ref{p54g}, we can take $u,v_1,v_2\in S^{n-p}$ such that
$f=\sum_{j=1}^{n-p+1}u_j f_j$, $ u\cdot v_k=0$ ($k=1,2$),
\begin{equation}\label{57f}
d_S(\Psi(\gamma_{\tilde{y}_i,\tilde{a}_i}(s)),\gamma_{v_1}(s))\leq C\eta_0^{3/13}
\end{equation}
for all $s\in [0,d(\tilde{a}_i,\tilde{y}_i)]$
and
\begin{equation}\label{57g}
d_S(\Psi(\gamma_{\overline{w}_i,\overline{b}_i}(s)),\gamma_{v_2}(s))\leq C\eta_0^{3/13}
\end{equation}
for all $s\in [0,d(\overline{b}_i,\overline{w}_i)]$,
where $\gamma_{v_k}(s):=(\cos s) u+(\sin s) v_k\in S^{n-p}$ ($k=1,2$).
Since
$$|d(\tilde{a}_i,\tilde{y}_i)-d(\overline{b}_i,\overline{w}_i)|\leq
|d(a_i,A_f)-d(b_i,A_f)|+C\eta_0\leq d(a_i,b_i)+C\eta_0,$$
we have
\begin{equation}\label{57h}
\left|d_S(\Psi(\tilde{a}_i),\Psi(\overline{b}_i))
-d_S\left(\gamma_{v_1}(l_i),\gamma_{v_2} (l_i)\right)
\right|\leq d(a_i,b_i)+C\eta_0^{3/13}
\end{equation}
and
\begin{equation}\label{57i}
\left|d_S(\Psi(a_{i+1}),\Psi(b_{i+1}))
-d_S\left(\gamma_{v_1}\left(\frac{2-i}{3-i}l_i\right),\gamma_{v_2} \left(\frac{2-i}{3-i}l_i\right)\right)
\right|\leq d(a_i,b_i)+C\eta_0^{3/13}
\end{equation}
by (\ref{57f}) and (\ref{57g}),
where we put $l_i:=d(\tilde{a}_i,\tilde{y}_i)$.
By (\ref{57h}) and Lemma \ref{p54c00}, we get
\begin{equation}\label{57i1}
|v_1-v_2|\sin l_i 
\leq C d_S\left(\gamma_{v_1}(l_i),\gamma_{v_2} (l_i)\right)
\leq Cd(a_i,b_i)+C\eta_0^{3/13}.
\end{equation}

We first suppose that $d(a_i,y)\leq \pi/6$.
Since $l_i\leq \pi/2$, we have $$\sin \left(\frac{2-i}{3-i}l_i\right)
\leq \sin l_i,$$
and so
\begin{align*}
d_S(\Psi(a_{i+1}),\Psi(b_{i+1}))
\leq& d_S\left(\gamma_{v_1}\left(\frac{2-i}{3-i}l_i\right),\gamma_{v_2} \left(\frac{2-i}{3-i}l_i\right)\right)+C\eta^{1/2}\\
\leq& C|v_1-v_2|\sin \left(\frac{2-i}{3-i}l_i\right)+C\eta^{1/2}\\
\leq& C|v_1-v_2|\sin l_i+C\eta^{1/2}\\
\leq& C d_S(\Psi(\tilde{a}_i),\Psi(\overline{b}_i))+C\eta^{1/2}
\leq C\eta^{1/2}
\end{align*}
by (\ref{57h}), (\ref{57i}) and $d(a_i,b_i)\leq C\eta^{1/2}$.
Thus, we get 
$d(a_{i+1},b_{i+1})\leq C\eta^{1/2}$
by (iv).

We next suppose that $\pi/6\leq d(a_i,y)\leq 5\pi/6$.
By (\ref{57i1}) and $d(a_i,b_i)\leq C\eta^{1/2}$, we have
$|v_1-v_2|\leq C\eta^{1/2}$.
Thus, we get
$
d_S(\Psi(a_{i+1}),\Psi(b_{i+1}))
\leq C\eta^{1/2}
$
by (\ref{57i}).
Thus, we get 
$d(a_{i+1},b_{i+1})\leq C\eta^{1/2}$
by (iv).

If $i\geq 1$, we have $d(a_i,y)\leq 5\pi/6$, and so we get $d(a_{i+1},b_{i+1})\leq C\eta^{1/2}$ by the above two cases.

Finally, we suppose that $i=0$ and $d(x,y)\geq 5\pi/6$.
By (\ref{57i1}) and $d(a_0,b_0)\leq C\eta$, we have
$|v_1-v_2|\sin l_0\leq C\eta$.
By the definition of $l_0$, we have
$|l_0-d(x,y)|\leq C\eta_0.$
Thus, we have $\sin l_0\geq \frac{1}{C}(\pi- l_0)\geq \frac{1}{C}\eta^{1/2}$, and so
we get
$|v_1-v_2|\leq C\eta^{1/2}$. This gives
$
d_S(\Psi(a_{i+1}),\Psi(b_{i+1}))
\leq C\eta^{1/2}
$
by (\ref{57i}).
Thus, 
$d(a_{i+1},b_{i+1})\leq C\eta^{1/2}$
by (iv).

Therefore, we have (i) for all cases, and we get the lemma.
\end{proof}

Let us show that the map $\Phi_f\colon M\to S^{n-p}\times A_f,\,x\mapsto (\Psi(x), a_f(x))$ is almost surjective.
\begin{Prop}\label{p54m}
Take $f\in \Span_{\mathbb{R}}\{f_1,\ldots, f_{n-p+1}\}$ with $\|f\|_2^2=1/(n-p+1)$.
For any $(v,a)\in S^{n-p}\times A_f$,
there exists $x\in M$ such that
$d(\Phi_f(x),(v,a))\leq C\eta_1^{1/2}$ holds.
\end{Prop}
\begin{proof}
Take arbitrary $(v,a)\in S^{n-p}\times A_f$.
Take $u\in S^{n-p}$ with $f=\sum_{i=1}^{n-p+1} u_i f_i$.
Since there exists $\tilde{v}\in S^{n-p}$ such that
$d_S(u,\tilde{v})\leq \pi-\eta_1^{1/2}$ and $d_S(v,\tilde{v})\leq \eta_1^{1/2}$,
it is enough to prove the proposition assuming $d_S(u,v)\leq \pi-\eta_1^{1/2}$.

Put $F_v:=\sum_{i=1}^{n-p+1}v_i f_i$.
Then,
$|F_v(p_{F_v})-1|\leq C\delta^{1/800n}$ and
$A_{F_v}=\{x\in M:|F_v(x)-1|\leq \delta^{1/900n}\}$ by Proposition \ref{p53a}.
In the following, we show that $a_v:=a_{F_v}(a)\in A_{F_v}$ has the desired property.
By Lemma \ref{p54c0}, we get
\begin{align}
\notag d_S(\Psi(a),u)\leq &C\delta^{1/2000n^2},\\
\label{57n}
d_S(\Psi(a_v),v)\leq &C\delta^{1/2000n^2}.
\end{align}
Thus, by Lemma \ref{p54c0}, we get
\begin{align*}
|d(a,a_v)-d(a_f(a_v),a_v)|=&|d(a,A_{F_v})-d(a_v,A_f)|\\
\leq& |d_S(\Psi(a),v)-d_S(\Psi(a_v),u)|+C\delta^{1/2000n^2}\\
\leq& C\delta^{1/2000n^2}\leq \eta_0
\end{align*}
and
$$
d(a_v,A_f)\leq d_S(\Psi(a_v),u)+C\delta^{1/2000n^2}
\leq d_S(u,v)+C\delta^{1/2000n^2}\leq \pi-\frac{1}{2}\eta_1^{1/2}.
$$
Since we have $d(a_v,A_f)=d(a_v,a_f(a_v))$, we get
\begin{align*}
|d(a_v,A_f)-d(a_v,a)|\leq |d(a_v,A_f)-d(a_v,a_f(a_v))|+\eta_0=\eta_0,
\end{align*}
and so
we get
\begin{equation}\label{57o}
d(a,a_f(a_v))\leq 
C\eta_1^{1/2}
\end{equation}
by Lemma \ref{p54l} putting $x=z=a_v$, $y=a$ and $w=a_f(a_v)$.

By (\ref{57n}) and (\ref{57o}), putting $x=a_v$, we get the proposition.
\end{proof}

Now, we are in position to show $|\dot{\gamma}_{\tilde{y}_1,y_2}^E|d(\tilde{y}_1,y_2)\leq \pi+L$ under the assumption of Lemma \ref{p54d}.
Note that we defined $\eta_2=\eta_1^{1/78}$ and $L=\eta_2^{1/150}$.
\begin{Lem}\label{p54n}
Take $y_1\in M$, $\tilde{y}_1\in D_{f_{y_1}}(p_{y_1})\cap R_{f_{y_1}}\cap Q_{f_{y_1}}$ with $d(y_1,\tilde{y}_1)\leq C\delta^{1/100n}$ and $y_2\in D_{f_{y_1}}(\tilde{y}_1)$.
Let $\{E_1,\ldots,E_n\}$ be a parallel orthonormal basis of $TM$ along $\gamma_{\tilde{y}_1,y_2}$ in Lemma \ref{p5e} for $f_{y_1}$.
Then,
$|\dot{\gamma}_{\tilde{y}_1,y_2}^E|d(\tilde{y}_1,y_2)\leq \pi+ L$
and
$$
||\dot{\gamma}_{\tilde{y}_1,y_2}^E|d(\tilde{y}_1,y_2)-d_S(\Psi(y_1),\Psi(y_2))|\leq CL.
$$
\end{Lem}
\begin{proof}
We immediately get the second assertion by the first assertion and Lemma \ref{p54d}.

Let us show the first assertion by contradiction. Suppose that $|\dot{\gamma}_{\tilde{y}_1,y_2}^E|d(\tilde{y}_1,y_2)>\pi+ L.$
Put
\begin{align*}
f:=-f_{y_1},\,\gamma:=\gamma_{\tilde{y}_1,y_2},\,s_0:=\frac{1}{|\dot{\gamma}^E|}\eta_2^{1/104}\text{ and }s_1:=\frac{1}{|\dot{\gamma}^E|}(\pi+L).
\end{align*}
Take $k\in \mathbb{N}$ to be
$(s_1-s_0)/\eta_2^{-1}<k\leq (s_1-s_0)/\eta_2^{-1}+1,$
and put
$
t_j:= s_0+ (s_1-s_0)j/k
$
for each $j\in\{0,\ldots,k\}$.
Note that we have $t_0=s_0$, $t_k=s_1$ and
\begin{align}\label{57p0}
\frac{1}{C}\eta_2^{-1}\leq k\leq C\eta_2^{-1}.
\end{align}

For all $s\in[s_0,s_1]$, we have
\begin{align*}
\cos d_S(\Psi(y_1),\Psi(\gamma(s)))
\leq \cos (|\dot{\gamma}^E| s)+C\delta^{1/2000n^2}
\leq 1-\frac{1}{C}\eta_2^{1/52}
\end{align*}
for all $s\in[s_0,s_1]$ by Lemma \ref{p54d}.
Since
$f(\gamma(s))=-|\widetilde{\Psi}|(\gamma(s))\cos d_S(\Psi(y_1),\Psi(\gamma(s)))$
by the definitions of $f_{y_1}$ and $f$,
we get
$
f(\gamma(s))\geq -1+\frac{1}{C}\eta_2^{1/52}
$
for all $s\in[s_0,s_1]$ by Lemma \ref{p54a}.
This gives
\begin{equation}\label{57p}
d(\gamma(s),A_f)\leq \pi-\frac{1}{C}\eta_2^{1/104}
\end{equation}
$s\in[s_0,s_1]$
by Proposition \ref{p53a}.
By the definition of $t_j$ and (\ref{57p}), we have
\begin{align}
\label{57p11} d(\gamma(t_j),\gamma(t_{j+1}))\leq& \eta_2,\\
\notag d(\gamma(t_{j+\sigma}),A_f)\leq & \pi-\frac{1}{C}\eta_2^{1/104}\leq \pi-\eta_2^{1/2}
\end{align}
for all $j\in\{0,\ldots,k-1\}$ and $\sigma\in\{0,1\}$,
and so we get
\begin{equation}\label{57q}
|d(\gamma(t_j),\gamma(t_{j+1}))^2-d_S(\Psi(\gamma(t_j)),\Psi(\gamma(t_{j+1})))^2-d(a_f(\gamma(t_j)),a_f(\gamma(t_{j+1})))^2|\leq C\eta_1
\end{equation}
by Lemma \ref{p54l}.
In particular, we get
\begin{equation}\label{57r}
d(a_f(\gamma(t_j)),a_f(\gamma(t_{j+1})))\leq C\eta_2
\end{equation}
by (\ref{57p11}).

Take $j_0\in\{1,\ldots, k-1\}$ to be
$
|\dot{\gamma}^E|t_{j_0}< \pi \leq |\dot{\gamma}^E|t_{j_0+1}.
$
Since
$$
||\dot{\gamma}^E|s-d_S(\Psi(y_1),\Psi(\gamma(s)))|\leq C\delta^{1/4000n^2}
$$
for all $s\in\left[0,\frac{1}{|\dot{\gamma}^E|}\pi\right]$ by Lemma \ref{p54d},
we get
\begin{equation}\label{57s}
\begin{split}
d_S(\Psi(\gamma(t_j)),\Psi(\gamma(t_{j+1})))
\geq &d_S(\Psi(y_1),\Psi(\gamma(t_{j+1})))- d_S(\Psi(y_1),\Psi(\gamma(t_{j})))\\
\geq &|\dot{\gamma}^E|(t_{j+1}-t_j)-C\delta^{1/4000n^2}
\end{split}
\end{equation}
for all $j\in \{0,\ldots,j_0-1\}$.
Since
$$
|2\pi-|\dot{\gamma}^E|s-d_S(\Psi(y_1),\Psi(\gamma(s)))|\leq C\delta^{1/4000n^2}
$$
for all $s\in\left[\frac{1}{|\dot{\gamma}^E|}\pi,s_1\right]$
by Lemma \ref{p54d},
we get
\begin{equation}\label{57t}
d_S(\Psi(\gamma(t_j)),\Psi(\gamma(t_{j+1})))
\geq |\dot{\gamma}^E|(t_{j+1}-t_j)-C\delta^{1/4000n^2}
\end{equation}
for all $j\in \{j_0+1,\ldots,k-1\}$.
By (\ref{57q}), (\ref{57s}) and (\ref{57t}), we get
\begin{equation}\label{57u}
d(a_f(\gamma(t_j)),a_f(\gamma(t_{j+1})))^2
\leq d(\gamma(t_j),\gamma(t_{j+1}))^2-|\dot{\gamma}^E|^2(t_{j+1}-t_j)^2+C\eta_1
\end{equation}
for all $j\in\{0,\ldots,k-1\}\setminus \{j_0\}$.

Since we have
\begin{align*}
d_S(\Psi(\gamma(s_l)),\Psi(p_f))
\leq d(\gamma(s_l),A_f)+C\delta^{1/2000n^2}
\leq \pi-\frac{1}{C}\eta_2^{1/104}
\end{align*}
for each $l=0,1$ by Lemma \ref{p54c0}, Corollary \ref{p54c01} and (\ref{57p}),
we can take a curve $\beta\colon[0,K]\to S^{n-p}$
 in $S^{n-p}$ with unit speed ($K$ is some constant) such that
\begin{align*}
\beta(0)=&\Psi(\gamma(s_0)),\\
\beta(K)=&\Psi(\gamma(s_1)),\\
|d_S(\Psi(\gamma(s_0)),\Psi(\gamma(s_1)))-K|\leq &C\eta_2^{1/104},\\
d_S(\beta(s),\Psi(p_f))\leq &\pi-\frac{1}{C}\eta_2^{1/104}
\end{align*}
for all $s\in[0,K]$.
Note that we can find such $\beta$ by taking an almost shortest pass in 
$\left\{u\in S^{n-p}: d(u,\Psi(p_f))\leq \pi-\frac{1}{C}\eta_2^{1/104}\right\}.$
By Proposition \ref{p54m}, there exists $x_j\in M$ such that
\begin{equation}\label{57v}
d\left(\Phi_f(x_j),\left(\beta\left(\frac{j}{k}K\right),a_f(\gamma(t_j))\right)\right)\leq C\eta_1^{1/2}
\end{equation}
for each $j\in\{0,\ldots,k\}$.
Moreover, we can take $x_0:=\gamma(s_0)$ and $x_k:=\gamma(s_1)$.
By (\ref{57p0}), (\ref{57r}), (\ref{57v}), Lemma \ref{p54c0} and Corollary \ref{p54c01},
we have
\begin{align}
\notag d(a_f(x_j),a_f(x_{j+1}))\leq &C\eta_2,\\
\label{57v1}d_S(\Psi(x_j),\Psi(x_{j+1}))\leq &\frac{1}{k}K+C\eta_1^{1/2}\leq C\eta_2,\\
\label{57v2}d(x_j,A_f)\leq &d_S(\Psi(x_j),\Psi(p_f))+C\delta^{1/2000n^2}\\
\notag \leq &d_S\left(\beta\left(\frac{j}{k}K\right),\Psi(p_f)\right)+C\eta_1^{1/2}
\leq\pi-\frac{1}{C}\eta_2^{1/104}
\end{align}
for all $j$, and so
\begin{equation}\label{57v3}
d(x_j,x_{j+1})\leq C\eta_2^{1/52}
\end{equation}
by Lemma \ref{p54j} putting $x=x_j, y=a_f(x_j), z=x_{j+1}$ and $\eta=\eta_2$.
By (\ref{57v2}), (\ref{57v3}) and
Lemma \ref{p54l}  putting $x=x_j, y=a_f(x_j), z=x_{j+1}, w=a_f(x_{j+1})$ and $\eta=\eta_2^{1/52}$, we get
\begin{equation}\label{57w}
|d(x_j,x_{j+1})^2-d_S(\Psi(x_j),\Psi(x_{j+1}))^2-d(a_f(x_j),a_f(x_{j+1}))^2|\leq C\eta_1
\end{equation}
for all $j\in\{0,\ldots,k-1\}$.
By (\ref{57u}), (\ref{57v1}) and (\ref{57w}), we have
\begin{equation}\label{57x}
\begin{split}
d(x_j,x_{j+1})^2
\leq &\frac{1}{k^2}K^2+d(a_f(x_j),a_f(x_{j+1}))^2+C\eta_1^{1/2}\\
\leq &\frac{1}{k^2}K^2+d(\gamma(t_j),\gamma(t_{j+1}))^2-|\dot{\gamma}^E|^2(t_{j+1}-t_j)^2+C\eta_1^{1/2}
\end{split}
\end{equation}
for all $j\in\{0,\ldots,k-1\}\setminus\{j_0\}$.
Since $K\leq \pi+C\eta_2^{1/104}$,
we have
\begin{equation}\label{57y}
\frac{1}{k^2}K^2\leq \frac{\pi^2}{k^2}+\frac{C}{k^2}\eta_2^{1/104}.
\end{equation}
Since
$$|\dot{\gamma}^E|(t_{j+1}-t_j)=\frac{|\dot{\gamma}^E|}{k}(s_1-s_0)
=\frac{1}{k}(\pi+L-\eta_2^{1/104})
\geq\frac{1}{k}\left(\pi+\frac{1}{2}L\right),$$
we have
\begin{equation}\label{57z}
|\dot{\gamma}^E|^2(t_{j+1}-t_j)^2\geq \frac{\pi^2}{k^2}+\frac{1}{k^2}L
\end{equation}
for all $j\in\{0,\ldots,k-1\}$.
By (\ref{57y}) and (\ref{57z}), we get
$$
|\dot{\gamma}^E|^2(t_{j+1}-t_j)^2-\frac{1}{k^2}K^2\geq \frac{1}{k^2}L-\frac{C}{k^2}\eta_2^{1/104} \geq
\frac{1}{2k^2}L
$$
for all $j\in\{0,\ldots,k-1\}$.
Thus, by (\ref{57x}), we have
\begin{equation*}
d(x_j,x_{j+1})^2
\leq d(\gamma(t_j),\gamma(t_{j+1}))^2-\frac{1}{2k^2}L+C\eta_1^{1/2}
\leq d(\gamma(t_j),\gamma(t_{j+1}))^2-\frac{1}{4k^2}L
\end{equation*}
for all $j\in\{0,\ldots,k-1\}\setminus\{j_0\}$.
Since $d(\gamma(t_j),\gamma(t_{j+1}))+d(x_j,x_{j+1})\leq 1$, we get
\begin{equation}\label{58a}
\frac{1}{4k^2}L\leq d(\gamma(t_j),\gamma(t_{j+1}))^2-d(x_j,x_{j+1})^2
\leq d(\gamma(t_j),\gamma(t_{j+1}))-d(x_j,x_{j+1})
\end{equation}
$j\in\{0,\ldots,k-1\}\setminus\{j_0\}$.
By (\ref{57p0}), (\ref{57v3}) and (\ref{58a}),
we get
\begin{equation*}
\begin{split}
d(x_0,x_k)\leq \sum_{i=0}^{k-1}d(x_j,x_{j+1})
\leq &\sum_{i=0}^{k-1}d(\gamma(t_j),\gamma(t_{j+1}))-\frac{k-1}{4k^2}L+d(x_{j_0},x_{j_0+1})\\
\leq& d(x_0,x_k)-\frac{1}{8k}L.
\end{split}
\end{equation*}
This is a contradiction.
Thus, we get the lemma.
\end{proof}
\begin{notation}
For all $y_1,y_2\in M$, define
\begin{align*}
\overline{C}_f^{y_1}(y_2)=&\Big\{y_3\in M : \gamma_{y_2,y_3}(s)\in I_{y_1}\setminus\{y_1\} \text{ for almost all $s\in[0,d(y_2,y_3)]$, and}\\
&\qquad \qquad\qquad \qquad \int_{0}^{d(y_2,y_3)} |G_f^{y_1}|(\gamma_{y_2,y_3}(s))\,d s\leq L^{1/3}\Big\},\\
\overline{P}_f^{y_1}=&\{y_2\in M: \Vol(M\setminus \overline{C}_f^{y_1}(y_2))\leq L^{1/3}\Vol(M)\}.
\end{align*}
\end{notation}
Let us complete the Gromov-Hausdorff approximation.
\begin{Thm}\label{MT2}
Take $f\in\Span_{\mathbb{R}}\{f_1,\ldots, f_{n-p+1}\}$ with $\|f\|_2^2=1/(n-p+1)$.
Then, the map $\Phi_f\colon M\to S^{n-p}\times A_f$ is a $CL^{1/156n}$-Hausdorff approximation map.
In particular, we have $d_{GH}(M, S^{n-p}\times A_f)\leq CL^{1/156n}$.
\end{Thm}
\begin{proof}
Take arbitrary $y_1\in M$ and $\tilde{y}_1\in D_{f_{y_1}}(p_{y_1})\cap R_{f_{y_1}}\cap Q_{f_{y_1}}\cap Q_f$ with $d(y_1,\tilde{y}_1)\leq C\delta^{1/100n}$.
By Lemmas \ref{p54c00}, \ref{p54n} and Corollary \ref{p54f0}, we have
$
|G_f^{\tilde{y}_1}|(y_2)\leq CL
$
for all $y\in D_f(\tilde{y_1})\cap D_{f_{y_1}}(\tilde{y}_1)$.
Since $\Vol(M\setminus (D_f(\tilde{y_1})\cap D_{f_{y_1}}(\tilde{y}_1)))\leq C\delta^{1/100}\Vol(M)$ and $\|G_f^{\tilde{y}_1}\|_\infty\leq C$,
we get
$\|G_f^{\tilde{y}_1}\|_1\leq CL.$
Thus, by the segment inequality, we get
$
\Vol(M\setminus \overline{P}^{\tilde{y}_1}_f)\leq CL^{1/3}.
$

Take arbitrary $x,z\in M$.
By the Bishop-Gromov inequality, there exist $\tilde{z}\in D_{f_{z}}(p_{z})\cap Q_{f_{z}} \cap R_{f_{z}}\cap Q_f$, $\tilde{x}\in D_f(p_f)\cap Q_f \cap R_f\cap\overline{P}_f^{\tilde{z}}$ and $\tilde{y}\in D_f(\tilde{x})\cap D_f(p_f)\cap Q_f\cap R_f\cap \overline{C}_f^{\tilde{z}}(\tilde{x})$ such that $d(z,\tilde{z})\leq C \delta^{1/100n}$,
$d(x,\tilde{x})\leq CL^{1/3n}$ and $d(a_f(x),\tilde{y})\leq CL^{1/3n}$.
Here, we used the estimate $\Vol(M\setminus \overline{P}^{\tilde{z}}_f)\leq CL^{1/3}$.
Then, we get
$$
\left| d(\tilde{z},\tilde{x})^2-d_S(\Psi(\tilde{z}),\Psi(\tilde{x}))^2- d(\tilde{z},\tilde{y})^2+d_S(\Psi(\tilde{z}),\Psi(\tilde{y}))^2
\right|\leq CL^{1/78n}
$$
by Lemma \ref{p54i}.
Thus, we get
\begin{equation}\label{59a}
\left| d(z,x)^2-d_S(\Psi(z),\Psi(x))^2- d(z,a_f(x))^2+d_S(\Psi(z),\Psi(a_f(x)))^2
\right|\leq CL^{1/78n}
\end{equation}
by Lemma \ref{p54c00}.
Similarly, we have
\begin{equation}\label{59b}
\begin{split}
&\left| d(a_f(x),z)^2-d_S(\Psi(a_f(x)),\Psi(z))^2- d(a_f(x),a_f(z))^2+ d_S(\Psi(a_f(x)),\Psi(a_f(z)))^2\right|\\
\leq &CL^{1/78n}.
\end{split}
\end{equation}
Since we have $d_S(\Psi(a_f(x)),\Psi(a_f(z)))\leq C\delta^{1/2000n^2}$ by Lemma \ref{p54c0},
we get
\begin{equation*}
\left| d(z,x)^2-d_S(\Psi(z),\Psi(x))^2- d(a_f(x),a_f(z))^2\right|\leq CL^{1/78n}.
\end{equation*}
by (\ref{59a}) and (\ref{59b}).
This gives
\begin{equation*}
\begin{split}
&\left| d(z,x)-d(\Phi_f(z),\Phi_f(x))\right|\\
=&\left| d(z,x)-\left(d_S(\Psi(z),\Psi(x))^2+ d(a_f(x),a_f(z))^2\right)^{1/2}\right|\leq CL^{1/156n}.
\end{split}
\end{equation*}
Combining this and Proposition \ref{p54l}, we get the theorem.
\end{proof}
By the above theorem, we get Main Theorem 2 except for the orientability, which is proved in subsection 4.7.
\subsection{Further Inequalities}
In this subsection, we assume that Assumption \ref{aspform} holds, and prepare two lemmas to prove the remaining part of main theorems.
\begin{Lem}\label{pfua}
For any $f\in \Span_{\mathbb{R}}\{f_1,\ldots,f_{k}\}$, we have
$$
\left\|\sum_{i=1}^n e^i\otimes (\nabla_{e_i}d f+f e^i)\wedge \omega \right\|_2\leq C\delta^{1/8}\|f\|_2.
$$
\end{Lem}
\begin{proof}
We have
\begin{equation}\label{fua}
\begin{split}
&\left|\sum_{i=1}^n e^i\otimes (\nabla_{e_i}d f+f e^i)\wedge \omega\right|^2\\
=& |\nabla^2 f|^2|\omega|^2-\frac{1}{n-p}(\Delta f)^2|\omega|^2
+2\Delta f \left(\frac{1}{n-p}\Delta f-f\right)|\omega|^2\\
-&(n-p)\left(\left(\frac{\Delta f}{n-p}\right)^2-f^2\right)|\omega|^2
-\left|\sum_{i=1}^n e^i \otimes\iota(\nabla_{e_i}\nabla f)\omega\right|^2-2\sum_{i=1}^n f \langle\omega , e^i\wedge \iota(\nabla_{e_i}\nabla f)\omega\rangle.
\end{split}
\end{equation}
By the assumption, we have
\begin{align}
\label{fub}\left\|\Delta f \left(\frac{1}{n-p}\Delta f-f\right)|\omega|^2\right\|_1\leq &C\delta^{1/2}\|f\|_2^2,\\
\label{fuc}\left\|\left(\left(\frac{\Delta f}{n-p}\right)^2-f^2\right)|\omega|^2\right\|_1\leq &C\delta^{1/2}\|f\|_2^2.
\end{align}
By Lemma \ref{p4d} (iv) and Lemma \ref{p5c} (ii),
we have
\begin{equation}\label{fud}
\left\|\sum_{i=1}^n e^i \otimes\iota(\nabla_{e_i}\nabla f)\omega\right\|_2
\leq \|\nabla (\iota(\nabla f)\omega)\|_2+ C\delta^{1/2}\|f\|_2\leq C\delta^{1/4}\|f\|_2,
\end{equation}
and so
\begin{equation}\label{fue}
\left\|\sum_{i=1}^n f \langle\omega , e^i\wedge \iota(\nabla_{e_i}\nabla f)\omega\rangle\right\|_1\leq C\|f\|_2\left\|\sum_{i=1}^n e^i \otimes\iota(\nabla_{e_i}\nabla f)\omega\right\|_2
\leq C\delta^{1/4}\|f\|_2^2.
\end{equation}
By Lemma \ref{p5c}, (\ref{fua}), (\ref{fub}), (\ref{fuc}), (\ref{fud}) and (\ref{fue}), we get the lemma.
\end{proof}

\begin{Lem}\label{pfub}
Define $G=G(f_1,\ldots,f_k)$ by
\begin{equation*}
\begin{split}
G:=\Big\{x\in M: & |f_i^2+|\nabla f_i|^2-1|(x)\leq\delta^{1/1600n}\text{ for all $i=1,\ldots,k$, and}\\
&\left|\frac{1}{2}(f_i\pm f_j)^2+\frac{1}{2}|\nabla f_i\pm \nabla f_j|^2-1\right|(x)\leq \delta^{1/1600n}\text{ for all $i\neq j$}
\Big\}.
\end{split}
\end{equation*}
Then, we have the following properties.
\begin{itemize}
\item[(i)] We have $\Vol(M\setminus G)\leq C\delta^{1/1600n}\Vol(M)$.
\item[(ii)] For all $x\in G$ and $i,j$ with $i\neq j$, we have $\left|f_i f_j+\langle\nabla f_i,\nabla f_j\rangle\right|(x)\leq\delta^{1/1600n}$.
\end{itemize}
\end{Lem}
\begin{proof}
By Proposition \ref{p53a} (iii), we have
\begin{equation*}
\begin{split}
 \|f_i^2+|\nabla f_i|^2-1\|_1\leq&C\delta^{1/800n},\\
\left\|\frac{1}{2}(f_i\pm f_j)^2+\frac{1}{2}|\nabla f_i\pm \nabla f_j|^2-1\right\|_1\leq &C\delta^{1/800n}
\end{split}
\end{equation*}
for all $i\neq j$.
Therefore, we get
\begin{align*}
&\Vol\left(\left\{x\in M: \left|f_i^2+|\nabla f_i|^2-1\right|(x)>\delta^{1/1600n}\right\}\right)\\
\leq& \delta^{-1/1600n}\int_M \left|f_i^2+|\nabla f_i|^2-1\right|\,d\mu_g\leq C\delta^{1/1600n}\Vol(M)
\end{align*}
for all $i$.
Similarly, we have
\begin{align*}
\Vol\left(\left\{x\in M: \left|\frac{1}{2}(f_i\pm f_j)^2+\frac{1}{2}|\nabla f_i\pm\nabla f_j|^2-1\right|(x)>\delta^{1/1600n}\right\}\right)\leq C\delta^{1/1600n}\Vol(M)
\end{align*}
for all $i\neq j$.
Thus, we get (i).

For all $x\in G$ and $i,j$ with $i\neq j$, we have 
\begin{align*}
&\left|f_i f_j+\langle\nabla f_i,\nabla f_j\rangle\right|(x)\\
=&\frac{1}{2}\left|
\frac{1}{2}(f_i+f_j)^2+\frac{1}{2}|\nabla f_i+\nabla f_j|^2
-\frac{1}{2}(f_i-f_j)^2-\frac{1}{2}|\nabla f_i-\nabla f_j|^2\right|(x)
\leq\delta^{1/1600n}.
\end{align*}
Thus, we get (ii).
\end{proof}
\subsection{Orientability}
The goal of this subsection is to show the orientability  of the manifold under the assumption of Main Theorem 2.
\begin{Thm}\label{pora}
If Assumption \ref{asu1} for $k=n-p+1$ and Assumption \ref{aspform} hold, then $M$ is orientable.
\end{Thm}
\begin{proof}
To prove the theorem, we use the following claim:
\begin{Clm}\label{porb}
Define
$$
\lambda_1(\Delta_{C,n}):=\inf \left\{\frac{\|\nabla \eta\|_2^2}{\|\eta\|_2^2}: \eta\in \Gamma(\bigwedge^n T^\ast M)\text{ with } \eta\neq 0\right\}.
$$
If
$
\lambda_1(\Delta_{C,n})< n(n-p-1)/(n-1)
$
holds, then $M$ is orientable.
\end{Clm}
\begin{proof}[Proof of Claim \ref{porb}]
Suppose that $M$ is not orientable.
Take the two-sheeted oriented Riemannian covering $\pi\colon (\widetilde{M},\tilde{g})\to (M,g)$.
Since we have $\Ric_{\tilde{g}}\geq (n-p-1)\tilde{g}$, we get
$$
\lambda_1(\Delta_{C,n},g)\geq \lambda_2(\Delta_{C,n},\tilde{g})=\lambda_1(\tilde{g})\geq \frac{n}{n-1}(n-p-1)
$$
by the Lichnerowicz estimate (note that $\lambda_1(\Delta_{C,n},\tilde{g})=\lambda_0(\tilde{g})=0$).
This gives the claim.
\end{proof}


Put
$$
V:=\sum_{i=1}^{n-p+1} (-1)^{i-1} f_i d f_1\wedge\cdots \wedge \widehat{d f_i}\wedge\cdots \wedge d f_{n-p+1}\wedge \omega\in \Gamma(\bigwedge^n T^\ast M).
$$
In the following, we show that $\|\nabla V\|_2^2/\|V\|_2^2< n(n-p+1)/(n-1)$.

Define a vector bundle $E:=T^\ast M\oplus \mathbb{R}e$, where $\mathbb{R}e$ denotes the trivial bundle of rank $1$ with a nowhere vanishing section $e$.
We consider an inner product $\langle\cdot,\cdot\rangle$ on $E$ defined by
$
\langle \alpha+ f e,\beta +h e\rangle:=\langle\alpha,\beta\rangle+ fh
$
for all $\alpha,\beta\in\Gamma(T^\ast M)$ and $f,h\in C^\infty(M)$.
Put
$$
S_i:=d f_i +f_i e\in \Gamma(E)
$$
for each $i$, and
$$
\alpha:=S_1\wedge\cdots \wedge S_{n-p+1}\in \Gamma(\bigwedge^{n-p+1} E).
$$
Then, we have
$\alpha\wedge \omega=e\wedge V$, and so
\begin{equation}\label{ora}
|\alpha\wedge \omega|=|V|.
\end{equation}

For each $k=1,\ldots,n-p+1$,
we have
\begin{equation*}
\begin{split}
&\Big\|
\big\langle
S_k\wedge\cdots \wedge S_{n-p+1}\wedge \omega,
\left(\iota(S_{k-1})\cdots\iota(S_1)\alpha\right)\wedge \omega
\big\rangle\\
&\qquad-\big\langle S_{k+1} \wedge\cdots \wedge S_{n-p+1}\wedge \omega,
\left(\iota(S_k)\cdots\iota(S_1)\alpha\right)\wedge \omega
\big\rangle
\Big\|_1\\
=&\left\|
\big\langle S_{k+1} \wedge\cdots \wedge S_{n-p+1}\wedge \omega,
\left(\iota(S_{k-1})\cdots\iota(S_1)\alpha\right)\wedge \iota(d f_k)\omega
\big\rangle
\right\|_1\\
\leq& C\|\iota(d f_k)\omega\|_2\leq C\delta^{1/4}
\end{split}
\end{equation*}
by Lemma \ref{p5c} (i).
By induction, we get
\begin{equation}\label{orb}
\||\alpha\wedge \omega|^2-|\alpha|^2|\omega|^2\|_1\leq C\delta^{1/4}.
\end{equation}
In particular, we have
\begin{equation}\label{orc}
\left|\|\alpha\wedge \omega\|_2^2-\||\alpha|^2|\omega|^2\|_1\right|
\leq C\delta^{1/4}.
\end{equation}

Since we have
$
\left|\langle S_i(x), S_j(x)\rangle -\delta_{i j}\right|\leq \delta^{1/1600n}
$
for all $x\in G=G(f_1,\ldots,f_{n-p+1})$ and $i,j$ by Lemma \ref{pfub} (ii),
we get
$
||\alpha|^2(x)-1|\leq C\delta^{1/1600n}
$
for all $x\in G$.
Thus, we get
\begin{equation}\label{ord}
\begin{split}
&\left|
\frac{1}{\Vol(M)}\int_M(|\alpha|^2|\omega|^2-1) \,d\mu_g
\right|\\
=&\Bigg|
\frac{1}{\Vol(M)}\int_G(|\alpha|^2-1)|\omega|^2 \,d\mu_g\\
&\qquad+\frac{1}{\Vol(M)}\int_{M\setminus G}(|\alpha|^2-1)|\omega|^2 \,d\mu_g+\frac{1}{\Vol(M)}\int_M(|\omega|^2-1) \,d\mu_g
\Bigg|\\
\leq &C\delta^{1/1600n}
\end{split}
\end{equation}
by Lemmas \ref{p4c} and \ref{pfub} (i).
By (\ref{ora}), (\ref{orc}) and (\ref{ord}), we get
\begin{equation}\label{ore}
|\|V\|_2^2-1|\leq C\delta^{1/1600n}.
\end{equation}

We next estimate $\|\nabla V\|_2^2$.
We have
\begin{equation*}
\begin{split}
&\nabla V\\
=& \sum_{i=1}^{n-p+1} (-1)^{i-1} d f_i\otimes d f_1\wedge\cdots \wedge \widehat{d f_i}\wedge\cdots \wedge d f_{n-p+1}\wedge \omega\\
+&\sum_{j<i}\sum_{k=1}^n (-1)^{i-1}(-1)^{j-1} f_i e^k\otimes (\nabla_{e_k} d f_j)\wedge d f_1\wedge\cdots\wedge\widehat{d f_j} \wedge\cdots\wedge \widehat{d f_i}\wedge\cdots \wedge d f_{n-p+1}\wedge \omega\\
+&\sum_{i<j}\sum_{k=1}^n (-1)^{i-1}(-1)^{j} f_i e^k\otimes (\nabla_{e_k} d f_j)\wedge d f_1\wedge\cdots\wedge\widehat{d f_i} \wedge\cdots\wedge\widehat{d f_j}\wedge\cdots \wedge d f_{n-p+1}\wedge \omega\\
+&\sum_{i=1}^{n-p+1} \sum_{k=1}^n (-1)^{i-1}  f_i e^k \otimes d f_1\wedge\cdots \wedge \widehat{d f_i}\wedge\cdots \wedge d f_{n-p+1}\wedge \nabla_{e_k}\omega.
\end{split}
\end{equation*}
Thus, we get
\begin{equation}\label{orf}
\begin{split}
&\left\|
\nabla V
- \sum_{i=1}^{n-p+1} (-1)^{i-1} d f_i\otimes d f_1\wedge\cdots \wedge \widehat{d f_i}\wedge\cdots \wedge d f_{n-p+1}\wedge \omega
\right\|_2\\
\leq& 
\Bigg\|\sum_{j<i}\sum_{k=1}^n (-1)^{i-1}(-1)^{j-1} f_i f_j e^k\otimes e^k\wedge d f_1\wedge\cdots\wedge\widehat{d f_j} \wedge\cdots\wedge \widehat{d f_i}\wedge\cdots \wedge d f_{n-p+1}\wedge \omega\\
&+\sum_{i<j}\sum_{k=1}^n (-1)^{i-1}(-1)^{j} f_i f_j e^k\otimes e^k\wedge d f_1\wedge\cdots\wedge\widehat{d f_i} \wedge\cdots\wedge\widehat{d f_j}\wedge\cdots \wedge d f_{n-p+1}\wedge \omega\Bigg\|_2\\
&+C\sum_{i=1}^{n-p+1}\left\|\sum_{k=1}^n e^k\otimes (\nabla_{e_k}d f_i+f_i e^k)\wedge\omega\right\|_2
+ C\|\nabla\omega\|_2\\
\leq &C\delta^{1/8}
\end{split}
\end{equation}
by Lemma \ref{pfua}.

Similarly to (\ref{orb}),
we have
\begin{equation}\label{org}
\begin{split}
&\Bigg\|\left|\sum_{i=1}^{n-p+1} (-1)^{i-1} d f_i\otimes d f_1\wedge\cdots \wedge \widehat{d f_i}\wedge\cdots \wedge d f_{n-p+1}\wedge \omega\right|^2\\
&\qquad-\left|\sum_{i=1}^{n-p+1} (-1)^{i-1} d f_i\otimes d f_1\wedge\cdots \wedge \widehat{d f_i}\wedge\cdots \wedge d f_{n-p+1}\right|^2|\omega|^2\Bigg\|_1\\
&\leq C\delta^{1/4}.
\end{split}
\end{equation}

Since we have
$
d f_1\wedge\cdots\wedge d f_{n-p+1}\wedge\omega=0,
$
we get
\begin{equation}\label{orh}
\begin{split}
&\|
|d f_1\wedge\cdots\wedge d f_{n-p+1}|^2|\omega|^2
\|_1\\
=&
\||d f_1\wedge\cdots\wedge d f_{n-p+1}|^2|\omega|^2-
|d f_1\wedge\cdots\wedge d f_{n-p+1}\wedge\omega|^2
\|_1
\leq C\delta^{1/4}
\end{split}
\end{equation}
similarly to (\ref{orb}).
By (\ref{q1k}), we get
\begin{equation}\label{ori}
\begin{split}
&\left|
\sum_{i=1}^{n-p+1}(-1)^{i-1}d f_i\otimes d f_1\wedge\cdots\wedge \widehat{d f_i}\wedge \cdots\wedge d f_{n-p+1}
\right|^2\\
=&(n-p+1)  |d f_1\wedge \cdots\wedge d f_{n-p+1}|^2.
\end{split}
\end{equation}
By (\ref{orh}) and (\ref{ori}),
we get
\begin{equation}\label{orj}
\left\|
\left|\sum_{i=1}^{n-p+1}(-1)^{i-1}d f_i\otimes d f_1\wedge\cdots\wedge \widehat{d f_i}\wedge \cdots\wedge d f_{n-p+1}\right|^2|\omega|^2
\right\|_1\leq C\delta^{1/4}.
\end{equation}
By  (\ref{org}) and (\ref{orj}),
we have
\begin{equation}\label{ork}
\left\|
\sum_{i=1}^{n-p+1}(-1)^{i-1}d f_i\otimes d f_1\wedge\cdots\wedge \widehat{d f_i}\wedge \cdots\wedge d f_{n-p+1}\wedge\omega
\right\|_2^2
\leq C\delta^{1/4}.
\end{equation}
By (\ref{orf}) and (\ref{ork}), we get
\begin{equation}\label{orl}
\|\nabla V\|_2\leq C\delta^{1/8}.
\end{equation}
By (\ref{ore}) and (\ref{orl}), we get
$
\lambda_1(\Delta_{C,n})\leq C\delta^{1/4},
$
and so we get the theorem by Claim \ref{porb}.
\end{proof} 
Combining Theorems \ref{MT2} and \ref{pora}, we get Main Theorem 2.
\subsection{Almost Parallel $(n-p)$-form II}
In this subsection, we show that the assumption ``$\lambda_{n-p}(g)$ is close to $n-p$'' implies the condition ``$\lambda_{n-p+1}(g)$ is close to $n-p$'' under the assumption $\lambda_1(\Delta_{C,n-p})\leq \delta$.

\begin{Lem}\label{pala}
Suppose that Assumption \ref{asu1} for $k=n-p$ and Assumption \ref{asn-pform} hold.
Put
$
F:= \langle d f_1\wedge\ldots\wedge d f_{n-p}, \xi \rangle\in C^\infty(M).
$
Then, we have
$$
\left|\|F\|_2^2-\frac{1}{n-p+1}\right|\leq C\delta^{1/1600n},\quad \left|\|\nabla F\|_2^2-\frac{n-p}{n-p+1}\right|\leq C\delta^{1/1600n}
$$
and
$$
\left|\frac{1}{\Vol(M)}\int_M f_i F\,d\mu_g\right|\leq C\delta^{1/2}
$$
for all $i=1,\ldots, n-p$.
\end{Lem}
\begin{proof}
If $M$ is not orientable, we take the two-sheeted oriented Riemannian covering $\pi\colon (\widetilde{M},\tilde{g})\to (M,g)$, and put 
$
\widetilde{F}:=F\circ \pi$ and $\tilde{f}_i:=f_i\circ \pi.
$
Then, we have
$
\|F\|_2=\|\widetilde{F}\|_2$, $\|\nabla F\|_2=\|\nabla \widetilde{F}\|_2,$
$$
\frac{1}{\Vol(\widetilde{M})}\int_{\widetilde{M}} \tilde{f}_i \widetilde{F} \,d\mu_{\tilde{g}}=
\frac{1}{\Vol(M)}\int_M f_i F \,d\mu_g
$$
and
$
\widetilde{F}=\langle d \tilde{f}_1\wedge\ldots\wedge d \tilde{f}_{n-p}, \pi^\ast \xi\rangle.
$
Thus, it is enough to consider the case when $M$ is orientable.
In the following, we assume that $M$ is orientable, and we fix an orientation of $M$.

Put
$
\omega:=\ast \xi\in \Gamma(\bigwedge^p T^\ast M).
$
Let $V_g\in \Gamma(\bigwedge^n T^\ast M)$ be the volume form of $(M,g)$.
Then, we have
\begin{equation}\label{ala}
F V_g= d f_1\wedge\cdots \wedge d f_{n-p}\wedge \omega.
\end{equation}
Define a vector bundle $E:=T^\ast M\oplus \mathbb{R}e$ and an inner product $\langle,\rangle$ on it as in the proof of Theorem \ref{pora}.
Put
$$
S_i:=d f_i +f_i e\in \Gamma(E)
$$
for each $i$, and
$$
\beta:=S_1\wedge\cdots \wedge S_{n-p}\in \Gamma(\bigwedge^{n-p} E).
$$

Since we have $|F|=|F V_g|$, we get
$
\||F|^2-|d f_1\wedge\cdots \wedge d f_{n-p} |^2|\omega|^2\|_1\leq C\delta^{1/4}
$
similarly to (\ref{orb}) by (\ref{ala}), and so
\begin{equation}\label{alb}
\left|\|F\|_2^2-\left\||d f_1\wedge\cdots \wedge d f_{n-p} |^2|\omega|^2\right\|_1\right|\leq C\delta^{1/4}
\end{equation}

By Lemma \ref{pfua} and (\ref{ala}), we have
\begin{equation*}
\left\|
\nabla (F V_g)+\sum_{i=1}^{n-p} \sum_{k=1}^n(-1)^{i-1} f_i e^k\otimes e^k\wedge d f_1\wedge \cdots\wedge \widehat{d f_i}\wedge \cdots\wedge d f_{n-p}\wedge \omega
\right\|_2\leq C\delta^{1/8}.
\end{equation*}
Since $|\nabla(F V_g)|=|\nabla F|$, we get
\begin{equation}\label{alc}
\left|\|\nabla F\|_2^2-\left\|\left|\sum_{i=1}^{n-p} \sum_{k=1}^n(-1)^{i-1} f_i e^k\otimes e^k\wedge d f_1\wedge \cdots\wedge \widehat{d f_i}\wedge \cdots\wedge d f_{n-p}\wedge \omega\right|^2\right\|_1\right|\leq C\delta^{1/8}.
\end{equation}
We have
\begin{equation}\label{ald}
\begin{split}
&\left|\sum_{i=1}^{n-p} \sum_{k=1}^n(-1)^{i-1} f_i e^k\otimes e^k\wedge d f_1\wedge \cdots\wedge \widehat{d f_i}\wedge \cdots\wedge d f_{n-p}\wedge \omega
\right|^2\\
=&\left|\sum_{i=1}^{n-p} (-1)^{i-1} f_i d f_1\wedge \cdots \wedge \widehat{d f_i}\wedge \cdots\wedge d f_{n-p}\wedge \omega\right|^2.
\end{split}
\end{equation}
Similarly to (\ref{orb}), we have
\begin{equation*}
\begin{split}
&\Bigg\|\left|\sum_{i=1}^{n-p} (-1)^{i-1} f_i d f_1\wedge \cdots \wedge\widehat{d f_i}\wedge \cdots\wedge d f_{n-p}\wedge \omega\right|^2\\
&\qquad -\left|\sum_{i=1}^{n-p} (-1)^{i-1} f_i d f_1\wedge \cdots\wedge \widehat{d f_i}\wedge \cdots\wedge d f_{n-p}\right|^2|\omega|^2\Bigg\|_1\leq C\delta^{1/4}.
\end{split}
\end{equation*}
Since we have
\begin{equation*}
\iota(e)\beta=\sum_{i=1}^{n-p} (-1)^{i-1} f_i d f_1\wedge \cdots\wedge \widehat{d f_i}\wedge \cdots\wedge d f_{n-p},
\end{equation*}
we get
\begin{equation}\label{alf}
\left\|\left|\sum_{i=1}^{n-p} (-1)^{i-1} f_i d f_1\wedge \cdots\wedge \widehat{d f_i}\wedge \cdots\wedge d f_{n-p}\wedge \omega\right|^2 -|\iota(e)\beta|^2|\omega|^2\right\|_1\leq C\delta^{1/4}.
\end{equation}
By (\ref{alc}), (\ref{ald}) and (\ref{alf}), we get
\begin{equation}\label{alf1}
\left|\|\nabla F\|_2^2-\left\||\iota(e)\beta|^2|\omega|^2\right\|_1\right|\leq C\delta^{1/8}.
\end{equation}

We have
\begin{equation}\label{alg}
|\beta|^2=|d f_1\wedge\cdots\wedge d f_{n-p}|^2+|\iota(e)\beta|^2.
\end{equation}
We calculate $\sum_{k=1}^n\left|e^k\wedge \beta\right|^2$ in two ways.
We have
\begin{equation}\label{alh}
\begin{split}
\sum_{k=1}^n|e^k\wedge \beta|^2=&(p+1)|\beta|^2-|e\wedge\beta|^2\\
=&(p+1)|\beta|^2-|d f_1\wedge\cdots\wedge d f_{n-p}|^2= p|\beta|^2+|\iota(e)\beta|^2
\end{split}
\end{equation}
by (\ref{alg}).
For all $\eta\in \Gamma(T^\ast M)$, we have
\begin{align*}
&|\eta\wedge\beta|^2\\
=&|\eta|^2|\beta|^2-\langle\iota(\eta)\beta,\iota(\eta)\beta\rangle\\
=&|\eta|^2|\beta|^2-\sum_{i,j=1}^{n-p}(-1)^{i+j}\langle \eta, d f_i\rangle\langle\eta, d f_j\rangle
\langle S_1\wedge\cdots\wedge \widehat{S_i}\wedge\cdots \wedge S_{n-p},S_1\wedge\cdots\wedge \widehat{S_j}\wedge\cdots \wedge S_{n-p}
\rangle,
\end{align*}
and so we get
\begin{equation}\label{ali}
\begin{split}
&\sum_{k=1}^n|e^k\wedge \beta|^2\\
=&n|\beta|^2-\sum_{i,j=1}^{n-p}(-1)^{i+j}\langle d f_i,d f_j\rangle
\langle S_1\wedge\cdots\wedge \widehat{S_i}\wedge\cdots \wedge S_{n-p},S_1\wedge\cdots\wedge \widehat{S_j}\wedge\cdots \wedge S_{n-p}\rangle.
\end{split}
\end{equation}
By (\ref{alh}) and (\ref{ali}), we get
\begin{equation}\label{alj}
\begin{split}
&|\iota(e)\beta|^2\\
=&(n-p)|\beta|^2-\sum_{i,j=1}^{n-p}(-1)^{i+j}\langle d f_i,d f_j\rangle
\langle S_1\wedge\cdots\wedge \widehat{S_i}\wedge\cdots \wedge S_{n-p},S_1\wedge\cdots\wedge \widehat{S_j}\wedge\cdots \wedge S_{n-p}\rangle
\end{split}
\end{equation}

Since we have $|\langle S_i,S_j\rangle(x)-\delta_{i j}|\leq C\delta^{1/1600n}$ for all $x\in G=G(f_1,\ldots, f_{n-p})$ by Lemma \ref{pfub} (ii),
we have
\begin{equation}\label{alk}
\begin{split}
&\Bigg\|\sum_{i=1}^{n-p}|d f_i|^2\\
&-\sum_{i,j=1}^{n-p}(-1)^{i+j}\langle d f_i,d f_j\rangle
\langle S_1\wedge\cdots\wedge \widehat{S_i}\wedge\cdots \wedge S_{n-p},S_1\wedge\cdots\wedge \widehat{S_j}\wedge\cdots \wedge S_{n-p}\rangle|\omega|^2
\Bigg\|_1
\leq C\delta^{1/1600n}
\end{split}
\end{equation}
and
\begin{equation}\label{all}
\left|\left\||\beta|^2|\omega|^2\right\|_1-1\right|\leq C\delta^{1/1600n}
\end{equation}
by Lemmas \ref{p4c} and \ref{pfub} (i).
By the assumption, we have
\begin{equation}\label{alm}
\left|\sum_{i=1}^{n-p}\|d f_i\|_2^2-\frac{(n-p)^2}{n-p+1}\right|\leq C\delta^{1/2}.
\end{equation}
By (\ref{alj}), (\ref{alk}), (\ref{all}) and (\ref{alm}), we get
\begin{equation}\label{aln}
\left|\left\||\iota(e)\beta|^2|\omega|^2\right\|_1-\frac{n-p}{n-p+1}\right|\leq C\delta^{1/1600n},
\end{equation}
and so
\begin{equation}\label{alo}
\left|\left\||d f_1\wedge\cdots\wedge d f_{n-p}|^2|\omega|^2\right\|_1-\frac{1}{n-p+1}\right|\leq C\delta^{1/1600n}
\end{equation}
by (\ref{alg}) and (\ref{all}).
By (\ref{alb}) and (\ref{alo}),
we get
$$
\left|\|F\|_2^2-\frac{1}{n-p+1}\right|\leq C\delta^{1/1600n}.
$$
By (\ref{alf1}) and (\ref{aln}),
we get
$$
\left|\|\nabla F\|_2^2- \frac{n-p}{n-p+1}\right| \leq C\delta^{1/1600n}.
$$

Let us show the remaining assertion.
Since we have
\begin{align*}
f_i F V_g=&\frac{1}{2}(-1)^{i-1} d \left(f_i^2 d f_1\wedge\cdots\wedge\widehat{d f_i}\wedge \cdots \wedge d f_{n-p}\wedge\omega\right)\\
-&\frac{1}{2}(-1)^{i-1} (-1)^{n-p-1}f_i^2 d f_1\wedge\cdots\wedge\widehat{d f_i}\wedge\cdots \wedge d f_{n-p}\wedge d \omega,
\end{align*}
we get
$$
\left|\frac{1}{\Vol(M)}\int_M f_i F\,d\mu_g\right|\leq
 C\|\nabla \omega\|_2 \leq C\delta^{1/2}
$$
by the Stokes theorem.
\end{proof}

By applying the min-max principle
\begin{align*}
&\lambda_{n-p+1}(g)\\
=&\inf\left\{\sup_{f\in V\setminus\{0\}}\frac{\|\nabla f\|_2^2}{\|f\|_2^2}: V\text{ is an $(n-p+1)$-dimensional subspace of } C^\infty (M)
\right\}
\end{align*}
to the subspace $\Span_{\mathbb{R}}\{f_1,\ldots, f_{n-p}, F\}$,
we immediately get the following corollary:
\begin{Cor}\label{palb}
If Assumption \ref{asu1} for $k=n-p$ and Assumption \ref{asn-pform} hold, then we have
$
\lambda_{n-p+1}(g)\leq n-p+C\delta^{1/1600n}.
$
\end{Cor}
Combining Theorem \ref{MT2} and Corollary \ref{palb}, we get Main Theorem 4.

Finally, we investigate the Gromov-Hausdorff limit of the sequence of the Riemannian manifolds that satisfy our pinching condition.
\begin{Thm}
Take $n\geq 5$ and $2\leq p < n/2$.
Let $\{(M_i,g_i)\}_{i\in\mathbb{N}}$ be a sequence of $n$-dimensional closed Riemannian manifolds with $\Ric_{g_i}\geq (n-p-1)g_i$ that satisfies one of the following:
\begin{itemize}
\item[(i)] $\lim_{i\to\infty}\lambda_{n-p+1}(g_i)=n-p$ and $\lim_{i\to \infty}\lambda_1(\Delta_{C,p},g_i)=0$,
\item[(ii)] $\lim_{i\to\infty}\lambda_{n-p}(g_i)=n-p$ and $\lim_{i\to \infty}\lambda_1(\Delta_{C,n-p},g_i)=0$.
\end{itemize}
If $\{(M_i,g_i)\}_{i\in\mathbb{N}}$ converges to a geodesic space $X$, then there exists a geodesic space $Y$ such that $X$ is isometric to $S^{n-p}\times Y$.
\end{Thm}
\begin{proof}
By Main Theorems 2 and 4, we get that there exist a sequence of positive real numbers $\{\epsilon_i\}$ and compact metric spaces $\{Y_i\}$ such that $\lim_{i\to \infty}\epsilon_i=0$ and $d_{GH}(M_i,S^{n-p}\times Y_i)\leq \epsilon_i$.
Then, $\{S^{n-p}\times Y_i\}$ converges to $X$ in the Gromov-Hausdorff topology, and so $\{Y_i\}$ is pre-compact in the Gromov-Hausdorff topology by \cite[Theorem 11.1.10]{Pe3}.
Thus, there exists a subsequence that converges to some compact metric space $Y$.
Therefore, we get that $X$ is isometric to $S^{n-p}\times Y$.
Since $X$ is a geodesic space, $Y$ is also a geodesic space.
\end{proof}


\appendix
\section{Limit Spaces and Unorientability}

In this appendix, we show the stability of unorientability under the non-collapsing Gromov-Hausdorff convergence assuming the two-sided bound on the Ricci curvature.
Similarly to Claim \ref{porb}, we have the following.
\begin{Lem}\label{papea}
For any $n$-dimensional unorientable closed Riemannian manifold $(M,g)$ with $\Ric\geq -K g$ and $\diam(M)\leq D$ $(K,D>0)$ we have
$
\lambda_1(\Delta_{C,n},g)\geq C_1(n,K,2D),
$
where $C_1(n,K,D)$ is defined by
$$
C_1(n,K,D):=\frac{1}{(n-1)D^2\exp\left(1+\sqrt{1+4(n-1)KD^2}\right)}.
$$
\end{Lem}
%
Note that we have $\lambda_1(g_1)\geq C_1(n,K,D)$
for any $n$-dimensional closed Riemannian manifold $(N_1,g_1)$ with $\Ric_{g_1}\geq -K g_1$ and $\diam(N_1)\leq D$ by the Li-Yau estimate \cite[p.116]{SY}.

We immediately get the following corollary.
\begin{Cor}\label{papeb}
Let $(M,g)$ be an $n$-dimensional closed Riemannian manifold with $\Ric\geq -K g$ and $\diam(M)\leq D$ $(K,D>0)$.
If 
$
\lambda_1(\Delta_{C,n},g)< C_1(n,K,2D),
$
then $M$ is orientable.
\end{Cor}

The following theorem is the main result of this section.
\begin{Thm}\label{paped}
Take real numbers $K_1,K_2\in\mathbb{R}$ and positive real numbers $D>0$ and $v>0$.
Let $\{(M_i,g_i)\}$ be a sequence of $n$-dimensional unorientable closed Riemannian manifolds with
$K_1 g_i\leq \Ric_{g_i}\leq K_2 g_i$, $\diam(M)\leq D$ and $\Vol(M)\geq v$.
Suppose that $\{(M_i,g_i)\}$ converges to a limit space $X$ in the Gromov-Hausdorff sense.
Then, $X$ is not orientable in the sense of Honda \cite{Hoor} $($see also the definition below$)$.
\end{Thm}
Note that Honda \cite[Theorem 1.3]{Hoor} showed the stability of orientability without assuming the upper bound on the Ricci curvature.

Before proving Theorem \ref{paped}, we fix our notation and recall definitions about limit spaces.
\begin{notation}
Take real numbers $K_1,K_2\in\mathbb{R}$ and positive real numbers $D>0$ and $v>0$.
 Let $\mathcal{M}=\mathcal{M}(n,K_1,K_2,D,v)$ be the set of isometry classes of $n$-dimensional closed Riemannian manifolds $(M,g)$ with $K_1g\leq \Ric_g \leq K_2 g$, $\diam(M)\leq D$ and $\Vol(M)\geq v$.
Let $\overline{\mathcal{M}}=\overline{\mathcal{M}}(n,K_1,K_2,D,v)$ be the closure of $\mathcal{M}$ in the Gromov-Hausdorff topology.
\end{notation}

If $X_i\in\overline{\mathcal{M}}$ ($i\in \mathbb{N}$) converges to $X\in\overline{\mathcal{M}}$ in the Gromov-Hausdorff topology, then there exist a sequence of positive real numbers $\{\epsilon_i\}_{i\in \mathbb{N}}$ with $\lim_{i\to \infty}\epsilon_i=0$, and a sequence of $\epsilon_i$-Hausdorff approximation maps $\phi_i \colon X_i\to X$. Fix such a sequence. We say a sequence $x_i\in X_i$ converges to $x\in X$ if $\lim_{i\to \infty}\phi_i(x_i)=x$ (denote it by $x_i\stackrel{GH}{\to} x$).
By the volume convergence theorem \cite[Theorem 5.9]{CC1}, $(X_i,H^n)$ converges to $(X,H^n)$ in the measured Gromov-Hausdorff sense, i.e., for all $r>0$ and all sequence $x_i\in X_i$ that converges to $x\in X$, we have $\lim_{i\to \infty}H^n(B_r (x_i))=H^n(B_r(x))$, where $H^n$ denotes the $n$-dimensional Hausdorff measure.

For all $X\in \overline{\mathcal{M}}$, we can consider the cotangent bundle $\pi \colon T^\ast X \to X$ with a canonical inner product by \cite{Ch0} and \cite{CC3} (see also \cite[Section 2]{Ho1} for a short review).
We have $H^n(X\setminus \pi(T^\ast X))=0$ and $T^\ast_x X:=\pi^{-1}(x)$ is an $n$-dimensional vector space for all $x\in \pi(T^\ast X)$.
For all Lipschitz function $f$ on $X$, we can define $d f(x)\in T_x^\ast X$ for almost all $x\in X$, and we have $d f\in L^\infty(T^\ast X)$.

Let us recall definitions of functional spaces on limit spaces briefly.
Note that we can define such functional spaces on more general spaces than our assumption.
Some of the following functional spaces are first introduced by Gigli \cite{Gig}.
\begin{Def}
Let $X\in \overline{\mathcal{M}}$.
\begin{itemize}
\item[(i)] Let $\LIP(X)$ be the set of the Lipschitz functions on $X$. For all $f\in \LIP(X)$, we define $\|f\|_{H^{1,2}}^2=\|f\|_2^2+\|d f\|_2^2$.
Let $H^{1,2}(X)$ be the completion of $\LIP(X)$ with respect to this norm.
\item[(ii)] Define
\begin{equation*}
\begin{split}
\mathcal{D}^2(\Delta,X):=\Big\{f\in H^{1,2}(X)& : \text{there exists $F\in L^2(X)$ such that}\\
&\int_X \langle df, dh \rangle\,d H^n=\int_X F h\,d H^n \text{ for all $h\in H^{1,2}(X)$} \Big\}.
\end{split}
\end{equation*}
For any $f\in\mathcal{D}^2(\Delta,X)$, the function $F\in L^2(X)$ is uniquely determined.
Thus, we define $\Delta f:=F$.
\item[(iii)] Define
\begin{equation*}
\begin{split}
\Test F(X):=&\left\{f\in\mathcal{D}^2(\Delta,X)\cap \LIP(X):\Delta f\in H^{1,2}(X)\right\},\\
\TestForm_p(X):=&\left\{\sum_{i=1}^N f_{0,i} d f_{1,i}\wedge\ldots\wedge d f_{p,i}: N\in \mathbb{N},\, f_{j,i}\in \Test F(X)\right\}
\end{split}
\end{equation*}
for all $p\in\{1,\ldots,n\}$.
\item[(vi)] The operator
$
\nabla\colon \TestForm_p(X)\to L^2(T^\ast X \otimes \bigwedge^p T^\ast X)
$
is defined by
\begin{align*}
&\nabla\sum_{i=1}^N f_{0,i} d f_{1,i}\wedge\ldots\wedge d f_{p,i}\\
:=&\sum_{i=1}^N \left(d f_{0,i}\otimes d f_{1,i}\wedge\ldots\wedge d f_{p,i}+ \sum_{j=1}^p f_{0,i} d f_{1,i}\wedge\ldots\wedge\nabla^2 f_{j,i}\wedge\ldots\wedge d f_{p,i}\right),
\end{align*}
where $\nabla^2$ denotes the Hessian $\Hess$ defined in \cite[Definition 3.3.1]{Gig} or
\cite{Ho0}.
\item[(v)] For any $\omega\in \TestForm_p(X)$, we define $\|\omega\|_{H_C^{1,2}}^2:=\|\omega\|_2^2+\|\nabla \omega\|_2^2$.
Let $H^{1,2}_C(\bigwedge^p T^\ast X)$ be the completion of $\TestForm_p (X)$ with respect to this norm.
\item[(vi)] Define
\begin{equation*}
\begin{split}
\mathcal{D}^2(\Delta_{C,p},X):=\Big\{\omega \in& H^{1,2}_C(\bigwedge^p T^\ast X) : \text{there exists $\hat{\omega}\in L^2(\bigwedge^p T^\ast X)$ such that}\\
&\int_X \langle \nabla \omega, \nabla \eta \rangle\,d H^n=\int_X \langle\hat{\omega}, \eta\rangle \,d H^n \text{ for all $\eta\in H_C^{1,2}(\bigwedge^p T^\ast X)$} \Big\}.
\end{split}
\end{equation*}
For any $\omega\in\mathcal{D}^2(\Delta_{C,p},X)$, the form $\hat{\omega}\in L^2(\bigwedge^p T^\ast X)$ is uniquely determined.
Thus, we put $\Delta_{C,p} \omega:=\hat{\omega}$.
\item[(viii)] For all $k\in \mathbb{Z}_{>0}$, we define
\begin{equation*}
\begin{split}
\lambda_k(\Delta_{C,p},X):=\inf\left\{\sup_{\omega\in \mathcal{E}_k\setminus \{0\}}\frac{\|\nabla \omega\|^2_2}{\|\omega\|^2_2}: \mathcal{E}_k\subset H^{1,2}_C(\bigwedge^p T^\ast X)\text{ is a $k$-dimensional subspace}\right\}.
\end{split}
\end{equation*}
\end{itemize}
Similarly to the smooth case, there exists a complete orthonormal system of eigenforms of the connection Laplacian $\Delta_{C,p}$ in $L^2(\bigwedge^p T^\ast M)$, and each eigenform is an element of $\mathcal{D}^2(\Delta_{C,p},X)$ (see \cite[Theorem 4.17]{Ho2}).

\end{Def}
Honda \cite{Ho2} showed the following theorem:
\begin{Thm}[\cite{Ho2}]\label{papee}
Let $\{X_i\}_{i\in \mathbb{N}}$ be a sequence in $\overline{\mathcal{M}}$ and let $X\in\overline{\mathcal{M}}$ be its Gromov-Hausdorff limit.
Then, we have 
$\lim_{i\to \infty}\lambda_k(\Delta_{C,p},X_i)=\lambda_k(\Delta_{C,p},X)$
for all $p\in\{0,\ldots,n\}$ and $k\in\mathbb{Z}_{>0}$.
\end{Thm}
\begin{Def}[Orientation \cite{Hoor}]
Let $X\in \overline{\mathcal{M}}$.
We say that $X$ is orientable if there exists $\omega\in L^\infty(\bigwedge^n T^\ast X)$ such that $|\omega|(z)=1$ for almost all $z\in X$ and that
$
\langle\omega,\eta\rangle\in H^{1,2}(X)
$
for any $\eta \in \TestForm_n(X)$.
We call $\omega$ an orientation of $X$.
\end{Def}
\begin{Lem}\label{papef}
Let $X\in \overline{\mathcal{M}}$.
Then, $X$ is orientable if and only if $\lambda_1(\Delta_{C,n},X)=0$.
\end{Lem}
\begin{proof}
We first suppose that $X$ is orientable and show $\lambda_1(\Delta_{C,n},X)=0$.
Let $\omega\in L^\infty(\bigwedge^n T^\ast X)$ be the orientation of $X$.
By \cite[Proposition 6.5]{Hoor}, for almost all $z\in X$, $\omega$ is differentiable at $z$ and $\nabla^{g_X}\omega(z)=0$, where $\nabla^{g_X}$ denotes the Levi-Civita connection defined in \cite{Ho0}. 
By Proposition 4.5 and Remark 4.7 in \cite{Ho2}, we have $\omega\in H^{1,2}_C(\bigwedge^p T^\ast X)$.
By \cite[Corollary 7.10]{Ho3},
we have $\nabla \omega(z)=\nabla^{g_X}\omega(z)=0$ for almost all $z\in X$.
Thus, we get
$
\lambda_1(\Delta_{C,n},X)=0
$
by the definition of $\lambda_1(\Delta_{C,n},X)$.

We next suppose $\lambda_1(\Delta_{C,n},X)=0$ and show that $X$ is orientable.
Let $\{(M_i,g_i)\}_{i\in \mathbb{N}}$ be a sequence in $\mathcal{M}$ that converges to $X$ in the Gromov-Hausdorff topology.
Then, we have $\lim_{i\to \infty}\lambda_1(\Delta_{C,n},g_i)=0$ by Theorem \ref{papee}.
Thus, by Corollary \ref{papeb}, we get that $M_i$ is orientable for sufficiently large $i$,
and so $X$ is orientable by the stability of orientability \cite[Theorem 1.3]{Hoor}.
\end{proof}
\begin{proof}[Proof of Theorem \ref{paped}]
Let $\{(M_i,g_i)\}_{i\in \mathbb{N}}$ be a sequence in $\mathcal{M}$ and let $X$ be its Gromov-Hausdorff limit.
Suppose that each $M_i$ is not orientable.
Then, we have
$
\lambda_1(\Delta_{C,n},g_i)\geq C_1(n,K_1,2D)
$
by Lemma \ref{papea}.
By Theorem \ref{papee}, we get
$
\lambda_1(\Delta_{C,n},X)\geq C_1(n,K_1,2D).
$
Thus, by Lemma \ref{papef}, we get the theorem.
\end{proof}
\begin{Thm}\label{papeg}
Let $X\in \overline{\mathcal{M}}$.
If $X$ is not orientable, then we have
$
\lambda_1(\Delta_{C,n},X)\geq C_1(n,K_1,2D).
$
\end{Thm}
\begin{proof}
Let $\{(M_i,g_i)\}_{i\in \mathbb{N}}$ be a sequence in $\mathcal{M}$ that converges to $X$ in the Gromov-Hausdorff topology.
By Lemma \ref{papef}, we have $\lambda_1(\Delta_{C,n},X)>0$, and so we get
$
\lambda_1(\Delta_{C,n},g_i)>0
$
for sufficiently large $i$ by Theorem \ref{papee}.
Thus, $M_i$ is not orientable and
$
\lambda_1(\Delta_{C,n},g_i)\geq C_1(n,K_1,2D)
$
for sufficiently large $i$ by Lemma \ref{papea}.
By Theorem \ref{papee}, we get the theorem.
\end{proof}
We immediately get the following corollaries:
\begin{Cor}\label{papeh}
Let $\{X_i\}_{i\in \mathbb{N}}$ be a sequence in $\overline{\mathcal{M}}$ and let $X\in\overline{\mathcal{M}}$ be its Gromov-Hausdorff limit.
If $X_i$ is not orientable for each $i$, then $X$ is not orientable.
\end{Cor}
\begin{Cor}\label{papei}
Let $\{X_i\}_{i\in \mathbb{N}}$ be a sequence in $\overline{\mathcal{M}}$ and let $X\in\overline{\mathcal{M}}$ be its Gromov-Hausdorff limit.
Then, the following two conditions are mutually equivalent.
\begin{itemize}
\item[(i)] $X_i$ is orientable for sufficiently large $i$.
\item[(ii)] $X$ is orientable.
\end{itemize}
\end{Cor}
By Corollary \ref{papei}, we have that if $X_1\in\overline{\mathcal{M}}$ is orientable and  $X_2\in\overline{\mathcal{M}}$ is unorientable, then $X_1$ and $X_2$ belong to different connected components in $\overline{\mathcal{M}}$ with respect to the Gromov-Hausdorff topology.

\section{Eigenvalue Estimate for $L^2$ Almost K\"{a}hler Manifolds}

In this section, we consider $L^2$ almost K\"{a}hler manifolds, i.e., we assume that there exists a $2$-form $\omega$ which satisfies that $\|\nabla \omega\|_2$ and $\|J_\omega^2+\Id\|_1$ are small, where $J_\omega\in\Gamma(T^\ast M\otimes T M)$ is defined so that $\omega=g(J_\omega\cdot,\cdot)$.
The main goal is to give the almost version of (\ref{kae}).
\begin{notation}
Let $(M,g)$ be a Riemannian manifold.
For each $2$-form $\omega\in \Gamma(\bigwedge^2 T^\ast M)$, let $J_\omega\in\Gamma(T^\ast M\otimes T M)$ denotes the anti-symmetric tensor that satisfies $\omega=g(J_\omega\cdot,\cdot)$.
\end{notation}
We first show the following easy lemmas.
\begin{Lem}\label{pB2}
Let $(M,g)$ be an $n$-dimensional closed Riemannian manifold.
If there exists a $2$-form $\omega$ such that $\|J_\omega^2+\Id\|_1<1$ holds,
then $n$ is an even integer.
\end{Lem}
\begin{proof}
There exists a point $x\in M$ such that $|J_\omega^2(x)+\Id_{T_x M}|<1$.
For any $v\in T_x M$ with $|v|=1$, we have $|J_\omega^2(x)(v)+v|<1$, and so
$|J_\omega^2(x)(v)|>0$.
Thus, $J_\omega(x)$ is non-degenerate.
Therefore, $(T_x M,\omega_x)$ is a symplectic vector space.
This implies the lemma.
\end{proof}
\begin{Lem}\label{pB3}
Given integers $n\geq 2$, $1\leq p\leq n-1$, and positive real numbers $K>0$, $D>0$, there exists $\delta_0(n,p,K,D)>0$ such that if $(M,g)$ is an $n$-dimensional closed Riemannian manifold with  $\Ric \geq-K g$ and $\diam(M)\leq D$, then we have
$
\lambda_{\alpha(n,p)+1}(\Delta_{C,p})\geq \delta_0(n,p,K,D),
$
where we defined
$
\alpha(n,p):=n!/(p!(n-p)!).
$
\end{Lem}
\begin{proof}
Put $\delta:=\lambda_{\alpha(n,p)+1}(\Delta_{C,p})$.
If $\delta\geq 1$, we get the lemma. Thus, we assume that $\delta<1$.
Let $\omega_i\in\Gamma(\bigwedge^p T^\ast M)$ denotes the $i$-th eigenform of the connection Laplacian $\Delta_{C,p}$ acting on $p$-forms with $\|\omega_i\|_2=1$.

We have
\begin{equation}\label{l2es}
\|\langle \omega_i,\omega_j\rangle\|_2^2\leq \frac{1}{\lambda_1(g)}\|\nabla\langle \omega_i,\omega_j\rangle\|_2^2\leq C(n,p,K,D)\delta
\end{equation}
for each $i,j=1,\ldots, \alpha(n,p)+1$ with $i\neq j$ by the Li-Yau estimate \cite[p.116]{SY} and Lemma \ref{Linfes}.
By Lemma \ref{p4c} and (\ref{l2es}), we have 
\begin{align*}
\|\langle \omega_i,\omega_j\rangle\|_1&\leq C(n,p,K,D)\delta^{1/2} \quad (i,j=1,\ldots, \alpha(n,p)+1 \text{ with } i\neq j),\\
\||\omega_i|^2-1\|_1&\leq C(n,p,K,D)\delta^{1/2} \quad (i=1,\ldots, \alpha(n,p)+1).
\end{align*}
Put
\begin{align*}
G:=\Big\{x\in M: & ||\omega_i|^2-1|(x)\leq\delta^{1/4}\text{ for all $i=1,\ldots, \alpha(n,p)+1$, and}\\
&\left|\langle \omega_i,\omega_j\rangle\right|(x)\leq \delta^{1/4}\text{ for all $i,j=1,\ldots, \alpha(n,p)+1$ with $i\neq j$}
\Big\}.
\end{align*}
Then, we have $\Vol(M\backslash G)\leq C_1(n,p,K,D)\delta^{1/4}\Vol(M)$ for some positive constant $C_1(n,p,K,D)$ depending only on $n,p,K$ and $D$ similarly to Lemma \ref{pfub}.

Let us show $\delta\geq \min\left\{1/C_1(n,p,K,D)^4,1/(\alpha(n,p)+1)^4\right\}$ by contradiction.
Suppose that that $\delta< \min\left\{1/C_1(n,p,K,D)^4,1/(\alpha(n,p)+1)^4\right\}$.
Then, we have $G\neq \emptyset$, and so we can take a point $x_0\in G$.
We show that $\omega_{1}(x_0),\ldots, \omega_{\alpha(n,p)+1}(x_0)\in \bigwedge^p T_{x_0}^\ast M$ are linearly independent.
Take arbitrary $a_1,\ldots, a_{\alpha(n,p)+1}\in \mathbb{R}$ with
$a_1 \omega_1(x_0)+\cdots+a_{\alpha(n,p)+1} \omega_{\alpha(n,p)+1}(x_0)=0$.
Take $i$ with $|a_i|=\max\{|a_1|,\ldots,|a_{\alpha(n,p)+1}|\}$.
Since we have $\langle a_1 \omega_1(x_0)+\cdots+a_k \omega_{\alpha(n,p)+1}(x_0),\omega_i(x_0)\rangle=0$, we get
\begin{equation*}
0\geq |a_i||\omega_i(x_0)|^2-\sum_{i\neq j}\left|a_j\langle \omega_i(x_0), \omega_j(x_0)\rangle\right|
\geq |a_i|\left(1-(\alpha(n,p)+1)\delta^{1/4}\right).
\end{equation*}
Thus, $|a_i|=0$, and so $a_1=\cdots=a_k=0$.
This implies the linearly independence of $\omega_{1}(x_0),\ldots, \omega_{\alpha(n,p)+1}(x_0)$.
This contradicts to $\dim\left(\bigwedge^p T^\ast_{x_0} M\right)=\alpha(n,p)$.
Thus, we get $\lambda_{\alpha(n,p)+1}(\Delta_{C,p})=\delta\geq \min\left\{1/C_1(n,p,K,D)^4,1/(\alpha(n,p)+1)^4\right\}$.
\end{proof}
\begin{Lem}\label{pB4}
Let $(M,g)$ be an $n$-dimensional closed Riemannian manifold.
Suppose that a $2$-form $\omega$ satisfies
\begin{itemize}
\item[(i)] $\|\nabla \omega\|_2^2\leq \delta\|\omega\|_2^2$,
\item[(ii)] $\|J_\omega^2+\Id\|_1\leq \delta^{1/4}\|\omega\|_2^2$
\end{itemize}
for some $0<\delta \leq 1/4$.
Let $\omega_\alpha$ be its image of the orthogonal projection
$$
P_{\delta}:L^2\left(\bigwedge^2 T^\ast M \right)\to \bigoplus_{\lambda_i(\Delta_{C,2})\leq \delta^{1/2}} \mathbb{R}\omega_i,
$$
where 
$\omega_i$ denotes the $i$-th eigenform of the connection Laplacian $\Delta_{C,2}$ with $\|\omega_i\|_2=1$ $(\omega_\alpha :=P_{\delta} (\omega))$.
Then, we have
\begin{itemize}
\item $\|\nabla \omega_\alpha\|_2^2\leq 2\delta\|\omega_\alpha\|_2^2$,
\item $\|J_{\omega_\alpha}^2+\Id\|_1\leq 10\delta^{1/4}\|\omega_\alpha\|_2^2$.
\end{itemize}
\end{Lem}
\begin{proof}
Put $\omega_\beta:=\omega-\omega_\alpha$.
Then, we have $\|\omega\|^2_2=\|\omega_\alpha\|^2_2+\|\omega_\beta\|^2_2$.
By the assumption (i), we have
\begin{align*}
\delta \|\omega\|_2^2
\geq \|\nabla\omega\|_2^2
=\|\nabla \omega_\alpha\|_2^2+\|\nabla \omega_\beta\|_2^2
\geq \|\nabla \omega_\alpha\|_2^2+\delta^{1/2}\|\omega_\beta\|_2^2.
\end{align*}
Thus, we get
\begin{align}
\label{AB1}\|\nabla \omega_\alpha\|_2^2&\leq \delta \|\omega\|_2^2,\\
\label{AB2}\|\omega_\beta\|_2^2&\leq \delta^{1/2}\|\omega\|_2^2,
\end{align}
and so
\begin{equation}\label{AB3}
\|\omega_\alpha\|_2^2=\|\omega\|^2_2 -\|\omega_\beta\|^2_2\geq (1-\delta^{1/2})\|\omega\|_2^2\geq \frac{1}{2}\|\omega\|_2^2.
\end{equation}
By the definitions of the norms, we have
$|J_\omega|^2=2|\omega|^2$ and $|J_{\omega_\alpha}|^2=2|\omega_\alpha|^2$.
Since we have
$
J_\omega^2-J_{\omega_\alpha}^2
=J_\omega(J_\omega-J_{\omega_\alpha})+(J_\omega-J_{\omega_\alpha})J_{\omega_\alpha},
$
we get
$
|J_\omega^2-J_{\omega_\alpha}^2|
\leq 2(|\omega|+|\omega_\alpha|)|\omega_\beta|.
$
Therefore, we have
$$
\|J_\omega^2-J_{\omega_\alpha}^2\|_1
\leq 4\|\omega\|_2\|\omega_\beta\|_2\leq 4\delta^{1/4}\|\omega\|_2^2
$$
by (\ref{AB2}),
and so
\begin{equation}\label{AB4}
\|J_{\omega_\alpha}^2+\Id\|_1
\leq \|J_{\omega}^2+\Id\|_1+\|J_\omega^2-J_{\omega_\alpha}^2\|_1\leq 5 \delta^{1/4}\|\omega\|_2^2\leq 10 \delta^{1/4}\|\omega_\alpha\|_2^2
\end{equation}
by (\ref{AB3}).
By (\ref{AB1}) and (\ref{AB4}), we get the lemma.
\end{proof}

Let us show the orientability for $L^2$ almost K\"{a}hler manifolds.

\begin{Prop}\label{orika}
For any integer $n\geq 2$ and positive real numbers $K>0$, $D>0$, there exists a constant $\delta_1(n,K,D)>0$ such that the following property holds.
Let $(M,g)$ be an $n$-dimensional closed Riemannian manifold with  $\Ric \geq-K g$ and $\diam(M)\leq D$.
If there exists a $2$-form $\omega$ such that
\begin{itemize}
\item[(i)] $\|\nabla \omega\|_2^2\leq \delta_1\|\omega\|_2^2$,
\item[(ii)] $\|J_\omega^2+\Id\|_1\leq \delta_1^{1/4}\|\omega\|_2^2$,
\end{itemize}
then $M$ is orientable.
\end{Prop}
\begin{proof}
By Lemma \ref{pB2}, we have that $n=2m$ is an even integer. We first assume that $\delta_1< \min\{1/4m^2,\delta_0(n,2,K,D)^2\}$.
Since $J_{\omega}$ is anti-symmetric, we have $|J_{\omega}|^2\leq \sqrt{2m}|J_{\omega}^2|$.
Thus, we get 
$$
\sqrt{\frac{2}{m}}\|\omega\|_2^2=
\frac{1}{\sqrt{2m}}\|J_{\omega}\|_2^2\leq \sqrt{2m}+\delta_1^{1/4} \|\omega\|_2^2
$$
by $|\Id|=\sqrt{2m}$.
This and $\delta_1^{1/4}\leq \frac{1}{2}\sqrt{\frac{2}{m}}$ imply that $\|\omega\|_2\leq \sqrt{2m}$.
Put $\omega_\alpha:=P_{\delta_1}(\omega)$.
Note that we have that $\|\omega_\alpha\|_2\leq\|\omega\|_2\leq \sqrt{2m}$ and that
 $\|\omega_\alpha\|_\infty\leq C(n,K,D)$  by Lemmas \ref{Linfes} and \ref{pB3}.

We first fix $x\in M$, and consider
the $\mathbb{C}$-linear map
$$
J_{\omega_\alpha}(x)\colon T_x M\otimes_\mathbb{R} \mathbb{C}\to T_x M\otimes_\mathbb{R} \mathbb{C}.
$$
Let us extend the Riemannian metric $\langle\cdot,\cdot\rangle
$ to $T_x M\otimes_\mathbb{R} \mathbb{C}$ so that
$$
\langle u_1 + i v_1, u_2+iv_2 \rangle=(\langle u_1,u_2\rangle+\langle v_1,v_2\rangle)+i(\langle v_1, u_2\rangle-\langle u_1, v_2\rangle)
$$
for all $u_1,u_2,v_1,v_2\in T_x M$.
Since $J_{\omega_\alpha}(x)$ is anti-symmetric,
there exist eigenvalues $\{\lambda_1,\overline{\lambda_1},\ldots,\lambda_m,\overline{\lambda_m}\}$ of $J_{\omega_\alpha}(x)$ and
an orthogonal basis $\{E_1,\overline{E_1},\ldots,E_m,\overline{E_m}\}$ of $T_{x} M\otimes_\mathbb{R} \mathbb{C}$
such that $J_{\omega_\alpha}(x) E_i=\lambda_i E_i$, where the overline denotes the complex conjugate.
Note that each $\lambda_i$ is a pure imaginary number.
Let $\{E^1,\overline{E^1},\ldots,E^m,\overline{E^m}\}\subset T_x^\ast M\otimes_{\mathbb{R}} \mathbb{C}\cong(T_x M\otimes_{\mathbb{R}} \mathbb{C} )^\ast$
be the dual basis of $\{E_1,\overline{E_1},\ldots,E_m,\overline{E_m}\}$.
If we extend $\omega_\alpha(x)$ to a complex bilinear form, then we have
$
\omega_\alpha(x)=\sum_{i=1}^m \lambda_i E^i\wedge \overline{E^i}.
$
Thus, we get
$
\omega_{\alpha}^m(x)=m! \lambda_1\cdots \lambda_m E^1\wedge\overline{E^1}\wedge E^m\wedge\overline{E^m},
$
and so
$
|\omega_{\alpha}^m(x)|=m!|\lambda_1|\cdots |\lambda_m|.
$
Since we have
$
|\lambda_i|^2
=|(J_{\omega_\alpha}^2+\Id)E_i-E_i|,
$
we get
$
\left|1-|\lambda_i|^2\right|\leq|J_{\omega_\alpha}^2+\Id|(x)
$ and $|\lambda_i|\leq C(n,K,D)$.
Therefore, we get
$$
\left||\omega_\alpha^m|^2-(m!)^2\right|\leq C|J_{\omega_\alpha}^2+\Id|,
$$
and so
$
\left|\|\omega_\alpha^m\|_2^2-(m!)^2\right|\leq C\delta_1^{1/4}
$
by Lemma \ref{pB4}.
Since we have
$
\|\nabla (\omega_\alpha^m)\|_2^2\leq C\delta_1
$
by Lemma \ref{pB4},
we get the proposition taking $\delta_1$ sufficiently small by Corollary \ref{papeb} (ii).
\end{proof}

The following theorem is the goal of this section.

\begin{Thm}
For any integer $n\geq 2$, there exists a constant $C(n)>0$ such that the following property holds.
Let $(M,g)$ be an $n$-dimensional closed Riemannian manifold with  $\Ric \geq (n-1) g$.
If there exists a $2$-form $\omega$ such that
\begin{itemize}
\item[(i)] $\|\nabla \omega\|_2^2\leq \delta\|\omega\|_2^2$,
\item[(ii)] $\|J_\omega^2+\Id\|_1\leq \delta^{1/4}\|\omega\|_2^2$,
\end{itemize}
for some $\delta>0$,
then we have
$
\lambda_1(g)\geq 2(n-1)- C(n)\delta^{1/2}.
$
\end{Thm}
\begin{Rem}
It is enough to prove the theorem when $\delta$ is small.
Thus, we can assume that $n=2m$ is an even integer by Lemma \ref{pB2}.
If $n=2$, then $\lambda_1(g)\geq 2(n-1)$ is the original Lichnerowicz estimate. 
If $n=4$, the conclusion of the theorem can also be deduced from Main Theorem 1.
\end{Rem}
\begin{proof}
We first assume that $\delta< \min\{1/4m^2,\delta_0(n,2,K,D)^2\}$.
Put $\omega_\alpha:=P_{\delta}(\omega)=\sum_{i=1}^k a_i \omega_i$.
Here, $\omega_i$ is the $i$-th eigenform of the connection Laplacian $\Delta_{C,2}$ with $\|\omega_i\|_2=1$ corresponding to the eigenvalue $\lambda_i(\Delta_{C,2})\leq \delta^{1/2}$ for each $i=1,\ldots, k$.
Similarly to Proposition \ref{orika}, we have 
$\|\omega_\alpha\|_\infty\leq C$.

Let $f\in C^\infty(M)$ be the first eigenfunction of the Laplacian with $\|f\|_2=1$.
If $\lambda_1(g)\geq 2(n-1)+1$, we get the theorem.
Thus, we assume that $\lambda_1(g)\leq 2(n-1)+1$.
Then, we have $\|f\|_\infty\leq C$ and $\|\nabla f\|_\infty \leq C$ by Lemma \ref{Linfes}.
By Lemma \ref{p4d} (i) and (iii), we have
\begin{equation}\label{AB5}
\begin{split}
&\frac{1}{\Vol(M)}\left|\int_M \langle\Delta (\iota(\nabla f)\omega_\alpha),\iota(\nabla f)\omega_\alpha\rangle-\lambda_1(g)\langle\iota(\nabla f)\omega_\alpha,\iota(\nabla f)\omega_\alpha\rangle\,d\mu_g
\right|\\
\leq &C\delta^{1/2}\|\omega_\alpha\|_2^2
\end{split}
\end{equation}
and
\begin{equation}\label{AB6}
\|d^\ast (\iota(\nabla f)\omega_\alpha)\|_2^2\leq C\delta\|\omega_\alpha\|_2^2.
\end{equation}
By (\ref{2b}), (\ref{AB5}), (\ref{AB6}) and the Bochner formula, we get
\begin{equation}\label{AB7}
\begin{split}
&\frac{n-1}{\Vol(M)}\int_M 
\langle J_{\omega_\alpha}\nabla f,J_{\omega_\alpha}\nabla f\rangle\,d\mu_g\\
\leq&\frac{1}{\Vol(M)}\int_M\Ric(J_{\omega_\alpha}\nabla f,J_{\omega_\alpha}\nabla f) \,d\mu_g\\
=&\frac{1}{\Vol(M)}\int_M \langle\Delta (\iota(\nabla f)\omega_\alpha),\iota(\nabla f)\omega_\alpha\rangle\,d\mu_g-\frac{1}{\Vol(M)}\int_M |\nabla (\iota(\nabla f)\omega_\alpha)|^2\,d\mu_g\\
\leq& \frac{\lambda_1(g)}{2\Vol(M)}\int_M \langle\iota(\nabla f)\omega_\alpha,\iota(\nabla f)\omega_\alpha\rangle\,d\mu_g+C\delta^{1/2}\\
=&\frac{\lambda_1(g)}{2\Vol(M)}\int_M \langle J_{\omega_\alpha}\nabla f,J_{\omega_\alpha}\nabla f \rangle\,d\mu_g+C\delta^{1/2}.
\end{split}
\end{equation}

Since $J_{\omega_\alpha}$ is anti-symmetric, we have
\begin{equation*}
\begin{split}
&\frac{1}{\Vol(M)}\left|\int_M \langle J_{\omega_\alpha}\nabla f,J_{\omega_\alpha}\nabla f\rangle\,d\mu_g- \int_M \langle \nabla f,\nabla f\rangle\,d\mu_g\right|\\
\leq&
\frac{1}{\Vol(M)}\left|\int_M \langle (J_{\omega_\alpha}^2+\Id)\nabla f,\nabla f\rangle\,d\mu_g\right|\\
\leq &C\|J_{\omega_\alpha}^2+\Id\|_1\leq C\delta^{1/4}
\end{split}
\end{equation*}
by Lemma \ref{pB4}.
Thus, taking $\delta$ sufficiently small, we get
\begin{equation}\label{AB8}
\|J_{\omega_\alpha}\nabla f\|_2^2
\geq \|\nabla f\|_2^2-C\delta^{1/4}= \lambda_1(g) -C\delta^{1/4}\geq n-C\delta^{1/4}\geq \frac{n}{2}
\end{equation}
by the Lichnerowicz estimate.
By (\ref{AB7}) and (\ref{AB8}), we get the theorem.
\end{proof}

\bibliographystyle{amsbook}

\begin{thebibliography}{99}
\bibitem{Ai2} M. Aino,
{Sphere theorems and eigenvalue pinching without positive Ricci curvature assumption,}
Calc. Var. Partial Differential Equations 58 (2019), no. 4, Art. 150, 85 pp.
\bibitem{AM} D. V. Alekseevsky, S. Marchiafava,
{Transformations of a quaternionic K\"{a}hler manifold,}
C. R. Acad. Sci. Paris S\'{e}r. I Math. 320 (1995), no. 6, 703--708.
\bibitem{Au} E. Aubry,
{Pincement sur le spectre et le volume en courbure de Ricci positive,}
Ann. Sci. \'{E}cole Norm. Sup. (4) 38 (2005), 387--405.
\bibitem{Be}
A. Besse,
{Einstein manifolds,}
Ergebnisse der Mathematik und ihrer Grenzgebiete, Springer-Verlag, Berlin (1987), xii+510 pp.
\bibitem{Ch0} J. Cheeger,
{Differentiability of Lipschitz functions on metric measure spaces,}
Geom. Funct. Anal. 9 (1999), 428--517.
\bibitem{Ch} J. Cheeger,
{Degeneration of Riemannian metrics under Ricci curvature bounds,}
Lezioni Fermiane, Scuola Normale Superiore, Pisa, (2001).
\bibitem{CC2} J. Cheeger, T. Colding,
{Lower bounds on Ricci curvature and the almost rigidity of warped products,}
Ann. of Math. (2) 144 (1996), 189--237.
\bibitem{CC1} J. Cheeger, T. Colding,
{On the structure of spaces with Ricci curvature bounded below. I,}
J. Differential Geom. 46 (1997), 406--480.
\bibitem{CC3} J. Cheeger, R. Colding,
{On the structure of spaces with Ricci curvature bounded below. III,}
J. Differential Geom. 54 (2000), no. 1, 37--74.
\bibitem{Co1} T. Colding,
{Shape of manifolds with positive Ricci curvature,}
Invent. Math. 124 (1996), 175--191.
\bibitem{Gig} N. Gigli,
{Nonsmooth differential geometry--an approach tailored for spaces with Ricci curvature bounded from below,}
Mem. Amer. Math. Soc. 251 (2018), no. 1196, v+161 pp. 
\bibitem{gr} J. F. Grosjean, 
{A new Lichnerowicz-Obata estimate in the presence of a parallel p-form,}
Manuscripta Math. 107 (2002), no. 4, 503--520.
\bibitem{Ho} S. Honda,
{Ricci curvature and almost spherical multi-suspension,}
Tohoku Math. J. (2) 61 (2009), 499--522. 
\bibitem{Ho0} S. Honda,
{A weakly second-order differential structure on rectifiable metric measure spaces,}
Geom. Topol. 18 (2014), 633--668. 
\bibitem{Ho1} S. Honda,
{Ricci curvature and $L^p$-convergence,}
J. Reine Angew. Math. 705 (2015), 85--154.
\bibitem{Hoor} S. Honda,
{Ricci curvature and orientability,}
Calc. Var. Partial Differential Equations 56 (2017), no. 6, Art. 174, 47 pp. 
\bibitem{Ho2} S. Honda,
{Spectral convergence under bounded Ricci curvature,}
J. Funct. Anal. 273 (2017), 1577--1662.
\bibitem{Ho3} S. Honda,
{Elliptic PDEs on compact Ricci limit spaces and applications,}
Mem. Amer. Math. Soc. 253 (2018), v+92 pp.
\bibitem{Pe1} P. Petersen,
{On eigenvalue pinching in positive Ricci curvature,}
Invent. Math. 138 (1999), 1--21.
\bibitem{Pe3} P. Petersen,
{Riemannian geometry. Third edition,}
Springer, Cham, 2016. xviii+499 pp.
\bibitem{Sa} T. Sakai,
{Riemannian geometry,}
Translated from the 1992 Japanese original by the author. Translations of Mathematical Monographs, 149. American Mathematical Society, Providence, RI, (1996), xiv+358 pp.
\bibitem{SY} R. Schoen, S.T. Yau,
{Lectures on differential geometry,}
Conference Proceedings and Lecture Notes in Geometry and Topology, I. International Press, Cambridge, MA, (1994) v+235 pp.
\bibitem{Se} U. Semmelmann, 
{Conformal Killing forms on Riemannian manifolds,}
Math. Z. 245 (2003), no. 3, 503–-527.
\end{thebibliography}

\end{document}